\documentclass{amsart}

\usepackage{amssymb}
\usepackage{amsthm}
\usepackage{amstext}
\usepackage{amsbsy}

\usepackage{enumerate}

\usepackage{subfigure}
\usepackage{graphicx}
\usepackage{pict2e}
\usepackage[font=small,labelfont=bf]{caption}

\usepackage[noadjust]{cite}

\usepackage[linktocpage]{hyperref}
\hypersetup{
    colorlinks=true,
    pdfborder={1 1 1},
}
\usepackage{mathtools}

\usepackage{todonotes}

\usepackage{tikz-3dplot}
\usepackage{tikz-cd}
\usepackage{tikz}

\newtheorem{thm}{Theorem}[section]

\newtheorem{lem}[thm]{Lemma}
\newtheorem{prop}[thm]{Proposition}

\newtheorem{claim}[thm]{Claim}

\newtheorem{quest}[thm]{Question}
\theoremstyle{remark}
\newtheorem{rem}[thm]{Remark}
\theoremstyle{definition}

\newcommand{\norm}[1]{\left\Vert#1\right\Vert}
\newcommand{\abs}[1]{\left\vert#1\right\vert}
\newcommand{\set}[1]{\left\{#1\right\}}
\newcommand{\R}{\mathbb{R}}
\newcommand{\C}{\mathbb{C}}
\newcommand{\N}{\mathbb{N}}
\newcommand{\Z}{\mathbb{Z}}

\renewcommand{\dim}{\mathbf{dim}}
\newcommand{\hdim}{\dim_\mathrm{H}}

\newcommand{\loc}{\mathrm{loc}}

\newcommand{\len}{\mathbf{len}}
\newcommand{\dist}{\mathbf{dist}}
\newcommand{\eqdef}{\overset{\mathrm{def}}=}
\newcommand{\diam}{\mathbf{diam}}

\begin{document}

\title{Newhouse Phenomena in the Fibonacci Trace Map}

\author[W. Yessen]{William Yessen}
\thanks{The author was supported in part by the National Science Foundation Postdoctoral Research Fellowship DMS-1304287}
\address{Department of Mathematics, Rice University, Houston, TX 77005}
\email{yessen@rice.edu}


\date{\today}

\begin{abstract}

We study dynamical properties of the Fibonacci trace map - a polynomial map that is related to numerous problems in geometry, algebra, analysis, mathematical physics, and number theory. Persistent homoclinic tangencies, stochastic sea of full Hausdorff dimension, infinitely many elliptic islands - all the conservative Newhouse phenomena are obtained for many values of the Fricke-Vogt invariant. The map has all the essential properties that were obtained previously for the Taylor-Chirikov standard map, and can be suggested as another candidate for the simplest conservative system with highly non-trivial dynamics.

\end{abstract}

\maketitle

\tableofcontents

\section{Introduction}\label{sec:intro}

In this paper we study the dynamics of the \emph{Fibonacci trace map}
\begin{equation}\label{e.tracemap}
T:\mathbb{R}^3\to \mathbb{R}^3, \ \ T(x, y, z)=(2xy-z, x, y).
\end{equation}
This map appears naturally in representation theory and geometry \cite{Cantat2009,Goldman2009,Roberts1994b}, analysis and mathematical physics (see \cite{Damanik2013,Mei2014,Baake1999} for surveys of old and recent results on applications in aperiodic lattice models, \cite{Romanelli2009,Sutherland1986,Baake1999} for quantum kicked systems and quantum walks, and \cite{Cantat2009,Iwasaki2007,Cantat2009b} for applications to Painlev\'e VI equation), and number theory (see \cite{Bowditch1998,Sasaki2008} for applications to Markov triples). Let us briefly describe some of these connections, and set the stage for our main result.

Consider the free group $\mathrm{F}_2 \eqdef \langle \alpha, \beta \rangle$. Let $\mathrm{Rep(\mathrm{F}_2})$ denote the set of representations of $\mathrm{F}_2$ into $\mathrm{SL}(2,\C)$. The group $\mathrm{SL}(2,\C)$ acts on $\mathrm{Rep}(\mathrm{F}_2)$ by conjugation, and for all $\rho \in \mathrm{Rep}(\mathrm{F}_2)$, the traces
\begin{align*}
x\eqdef \mathrm{Tr}(\rho(\alpha)), \hspace{5mm} y\eqdef \mathrm{Tr}(\rho(\beta)),\hspace{5mm} z \eqdef \mathrm{Tr}(\rho(\alpha \beta))
\end{align*}
are preserved under the action. We obtain a bijective correspondence $\rho \mapsto (x, y, z)$, $\rho \in \mathrm{Rep}(\mathrm{F}_2)$, between the algebraic quotient $\mathrm{Rep}(\mathrm{F}_2)/\mathrm{SL}(2,\C)$ (where $\mathrm{SL}(2, \C)$ acts by conjugation) and $\C^3$ (for details, see, for example, \cite{Goldman2009}). Now if $\mathrm{Aut}(\mathrm{F}_2)$ acts on $\mathrm{Rep}(\mathrm{F}_2)$ by composition, then this induces an action on $\C^3$ by polynomial diffeomorphisms (see \cite{Roberts1994b} and references therein for further details). These diffeomorphisms are called \emph{trace maps}.

It turns out that the trace of the commutator $[\rho(\alpha), \rho(\beta)]$ is invariant under the associated trace map and this leads to the {\it Fricke-Vogt invariant}\footnote{The Fricke-Vogt invariant appeared first more than a century ago in the results of Fricke-Klein and Vogt in a similar context; see \cite{Fricke1897,Fricke1896,Vogt1889} for the original works, and \cite{Goldman2009} for a modern exposition.}, the cubic polynomial
\begin{equation}\label{e.FV}
I(x, y, z)=x^2+y^2+z^2-2xyz-1.
\end{equation}
Therefore it is natural to consider the cubic level surfaces
\begin{align*}
S_V=\{(x, y, z)\ |\ I(x, y, z)=V\}
\end{align*}
that are invariant under the trace maps.

By studying the dynamics of trace maps, Cantat in \cite{Cantat2009} was able to give a detailed description of the dynamics of mapping class groups of the once punctured torus and the four punctured sphere. For example, in the (simpler) case of the once punctured torus, $\mathbb{T}_1$, we have $\pi_1(\mathbb{T}_1) = \mathrm{F}_2$ (the loop around the puncture corresponds to the commutator $[\alpha, \beta]$). Then the action of $\mathrm{Aut}(\mathrm{F}_2)$ on $\mathrm{Rep}(\mathrm{F}_2)$ induces an action of the mapping class group $\mathrm{MCG}^*(\mathbb{T}_1)$ on $\C^3$. Using the invariance of $S_V$, for each fixed $V\in \C$ we obtain a morphism from $\mathrm{MCG}^*(\mathbb{T}_1)$ to the group of polynomial diffeomorphisms on $S_V$ (see also \cite{Goldman2006,Goldman2003}).

In analysis, trace maps have been useful in the study of the Painlev\'e VI equation (see, for example, \cite{Cantat2009,Iwasaki2007}).

In physical applications, trace maps arise as renormalization maps in the spectral theory and the statistical mechanics of one-dimensional quasicrystal models, quasiperiodically kicked systems, and quantum walks. In this context, one-dimensional quasicrystals are modeled by endomorphisms (and, more often, automorphisms) of $\mathrm{F}_2$ (for a short exposition, see, for example, \cite{Baake1999}). For example, the widely studied case is given by $s\in \mathrm{Aut}(\mathrm{F}_2)$, $s(a) = ab$, $s(b) = a$, which leads to the trace map \eqref{e.tracemap}. Then a one-dimensional quasicrystal consisting of two "atoms" is represented by the infinite sequence obtained from the infinite repeated application of $s$:
\begin{align*}
s^n(a) = \underbrace{s\circ\cdots\circ s}_{n\text{ times}}(a), \hspace{2mm} n\rightarrow \infty.
\end{align*}
Given a model, it turns out that there exists a representation
\begin{align*}
\rho: \mathrm{F}_2 \rightarrow \mathrm{SL}(2, \C),
\end{align*}
such that much of the relevant physical information about the model is encapsulated in the traces of the matrices $\rho(s^n(a))$, $n\in\N$. Moreover, one is often interested in the limits as $n\rightarrow\infty$, which comes down to the dynamics of
\begin{align*}
T_V \eqdef T|_{S_V}
\end{align*}
for various values of $V\in \mathbb{K}$ ($\mathbb{K} = \R$ or $\mathbb{K} = \C$ depending on the model in question). This method has yielded great progress in the past few years (see \cite{Casdagli1986,Suto1987,Kohmoto1983,Ostlund1983} for the classical results and \cite{Yessen2011a,Yessen2012a,Yessen2011,Mei2013,Mei2014,Damanik2009c,Damanik2010,Damanik2008,Damanik2012,Damanik2013a,Damanik2013b,Damanik2009,Damanik2014e,Damanik2013X,Fillman2015x,DFG2014} and references therein for the more recent contributions).

The Fibonacci trace map belongs to a class of trace maps displaying the richest dynamics (see \cite{Cantat2009}). This class consists of trace maps associated to those automorphisms $s\in\mathrm{Aut}(F_2)$ that satisfy the condition that for some $k\in\N$, both "words" $s^k(a)$ and $s^k(b)$ contain the "letters" $a$ and $b$. In turns out that the dynamics of the trace maps from this class is qualitatively the same. Perhaps for historical reasons (see \cite{Kohmoto1983,Ostlund1983,Casdagli1986}) authors often choose the Fibonacci trace map as a representative example of this class.

In many physically motivated problems only maps $\{T_V\}$ that correspond to real and positive values  of the Fricke-Vogt invariant  are relevant (although in \cite{Yessen2012a} complex $V$ had to be considered in connection with some models in statistical mechanics). For example, in many of the aforementioned papers by Damanik et. al. where the Fibonacci trace map \eqref{e.tracemap} is used to study the spectral properties of the Fibonacci Hamiltonian (one of the most prominent one-dimensional aperiodic models), the map $T_{\frac{\lambda^2}{4}}$ is "responsible" for the spectral properties of the operator with the coupling constant $\lambda$.

The dynamics of the maps $T_V$ for $V\ge 0$ is completely understood. It is known that the set $\Lambda_V$ of bounded orbits (in forward and backward time) of $T_V$ for $V>0$ coincides with the nonwandering set and is a totally disconnected topologically mixing compact locally maximal hyperbolic set \cite{Casdagli1986,Cantat2009,Damanik2009}. As a consequence, points that do not escape to infinity in forward (backward) time are precisely those that form the stable (unstable) lamination. All other points escape to infinity either in forward or backward time (or both) depending whether they lie on the unstable but not stable, or stable but not unstable (or neither) lamination, respectively. The escape rate of points is superexponential, and escape regions have also been studied \cite{Roberts1996}. The topological entropy
\begin{align*}
h_{top}(T_V)=\log\frac{\sqrt{5}+1}{2}
\end{align*}
is the same for all $V\ge 0$ \cite{Cantat2009}. Also, it is known that
\begin{equation}\label{e.dimH}
\lim_{V\to 0+}\hdim\, \Lambda_V=2, \ \ \text{see \cite{Damanik2010}}, \ \ \text{and}
\end{equation}
\begin{equation}\label{e.dimH1}
\lim_{V\to +\infty}(\log V)\cdot \hdim\, \Lambda_V=4\log(1+\sqrt{2}), \ \ \text{see \cite{Damanik2008}}.
\end{equation}
For $V=0$ the corresponding surface $S_0$, called the \emph{Cayley cubic}, has four conic singularities,
 \begin{align}\label{eq:singularities}
  P_1 = (1,1,1),\hspace{2mm} P_2 = (-1,-1,1),\hspace{2mm} P_3 = (1,-1,-1),\hspace{2mm} P_4 = (-1,1,-1),
 \end{align}
and is smooth in other points. For $V > 0$, the surface deforms into a topological sphere with four punctures (see Figure \ref{fig:fvplots} (b)-(c)).

For $V \leq 0$, define
\begin{align}\label{eq:center-part}
 \mathbb{S}_V \eqdef S_V\cap [-1, 1]^3.
\end{align}
$\mathbb{S}_V$ is nonempty for $V \geq -1$, and is invariant under $T_V$. The map $T_0$ restricted to the central part of the Cayley cubic, $\mathbb{S}_0$, is a pseudo-Anosov transformation, and is a factor of the linear automorphism of the torus generated by the hyperbolic map
\begin{align}\label{eq:anosov}
\mathcal{A} = \begin{pmatrix}
1 & 1 \\
1 & 0
\end{pmatrix}.
\end{align}
The factor map is given by
\begin{align}\label{eq:factor}
\mathcal{F}(\theta, \phi) = (\cos 2\pi(\theta + \phi), \cos 2\pi\theta, \cos 2\pi\phi).
\end{align}

\begin{figure}[t!]
\subfigure[$V < 0$]{
\includegraphics[scale=0.22]{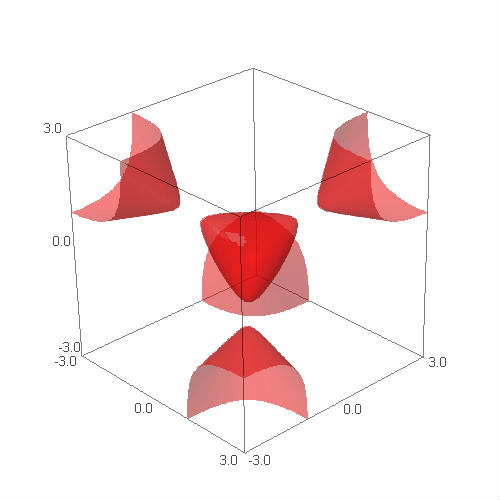}
}
\subfigure[$V = 0$]{
\includegraphics[scale=0.22]{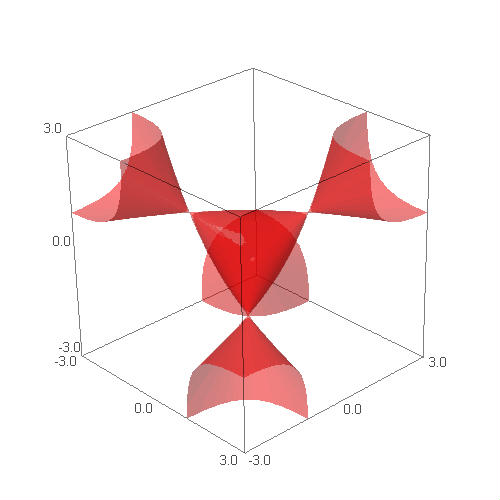}
}
\subfigure[$V > 0$]{
\includegraphics[scale=0.22]{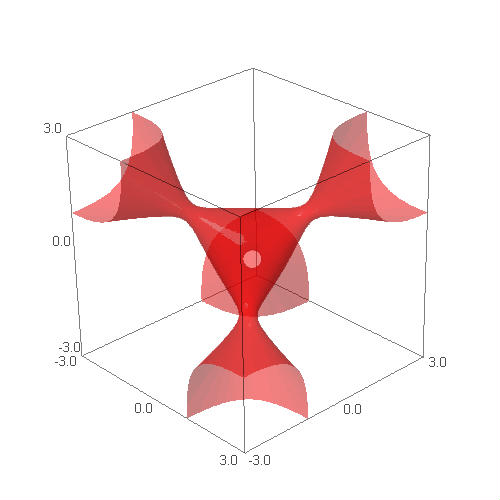}
}
\caption{Plots of the Fricke-Vogt invariant surfaces for a few values of $V$.}\label{fig:fvplots}
\end{figure}

For $V\in (-1,0)$ the surface $S_V$ consists of five connected components (see Figure \ref{fig:fvplots} (a)), one of them, $\mathbb{S}_V$ as defined in \eqref{eq:center-part}, is diffeomorphic to the two-dimensional sphere, and is invariant under the map $T_V$, and four others are non-compact, diffeomorfic to a disc. At $V = -1$, the bounded component degenerates into the point at the origin, and for $V < -1$, the bounded component disappeares completely.

Dynamics of the trace map on the unbounded components is trivial: every orbit escapes to infinity under both positive and negative iterates (it should be mentioned, however, that a famous open problem on Markov triples---the so-called {\it Frobenious Conjecture}---can be reformulated in terms of properties of some orbits of a semigroup generated by some trace map acting on the unbounded components of $S_V$; see \cite{Bowditch1998,Sasaki2008} for details).

In this paper we study the dynamics of the map $T_V$ restricted to the compact component $\mathbb{S}_V$, $V\in (-1, 0)$. This regime appears naturally in the models related to quantum walks in aperiodic media and quantum aperiodically kicked double rotators (see, for example, \cite{Romanelli2009,Sutherland1986} and a short review in \cite[Section 5]{Baake1999})\footnote{Kicked rotators are also sometimes called \emph{quantum standard maps} as they relate to the Chirikov standard map, discussed below, and were first studied in \cite{Casati1979}.}.

Notice that $T$ preserves the Euclidean volume. Then for any ${V}_1, {V}_2 \in (-1, 0)$ with ${V}_1 < {V}_2$, $T$ restricted to $\bigcup_{V \in ({V}_1, {V}_2)}\mathbb{S}_V$ also preserves the Euclidean volume, from which we see that for all $V\in (-1, 0)$, there exists an analytic area form on $\mathbb{S}_V$ that is preserved by $T_V$ and is analytic in $V$. Thus all the questions typically asked about (families of) low-dimensional conservative (and, in our case, analytic) systems can be asked about $T_V$.

Periodic orbits of the Fibonacci trace map $T_V$ (and other trace maps) for $V < 0$ were studied in \cite{Humphries2007,Humphries2015}. In the physics literature relating to the dynamics of $T_V$, $V\in (-1, 0)$ as it connects to the aforementioned models (\cite{Romanelli2009} and references therein, \cite{Sutherland1986}), the results are numeric and suggest a \emph{mixing} phenomenon similar to the one observed in the well known \emph{standard map} (which we discuss in more detail below).

\subsection{Statement of the main result (Theorem \ref{thm:main})}\label{sec:mainresult}
In this paper we prove

\begin{figure}[ht]
\subfigure[$V = -0.95$]{
\includegraphics[scale=0.28]{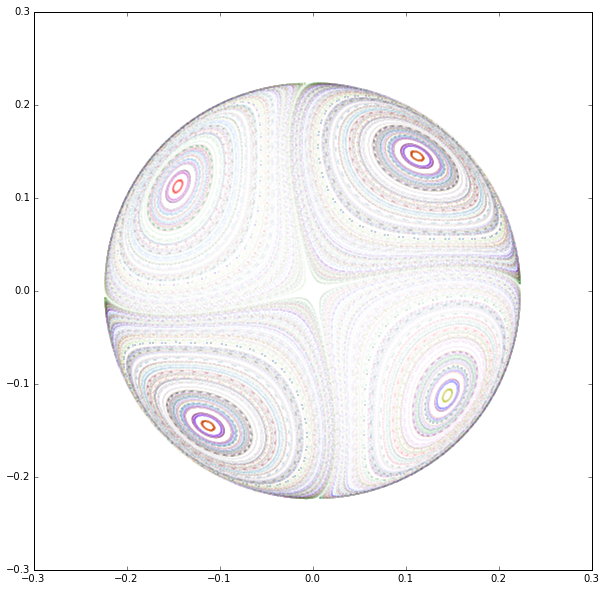}
}
\subfigure[$V = -0.7$]{
\includegraphics[scale=0.28]{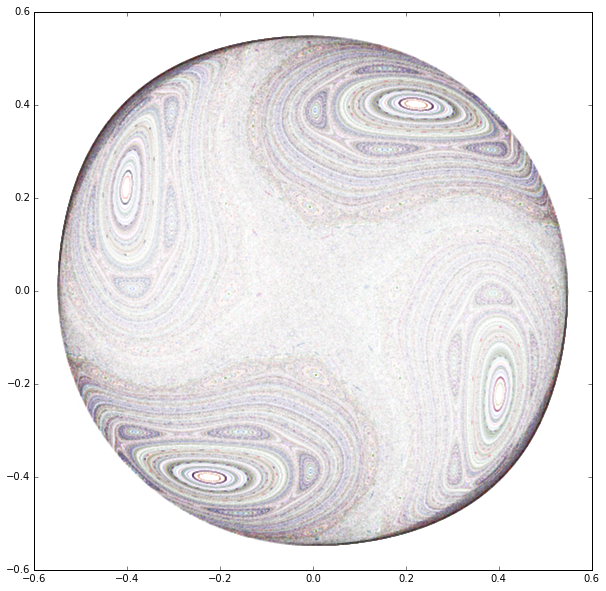}
}
\subfigure[$V = -0.5$]{
\includegraphics[scale=0.28]{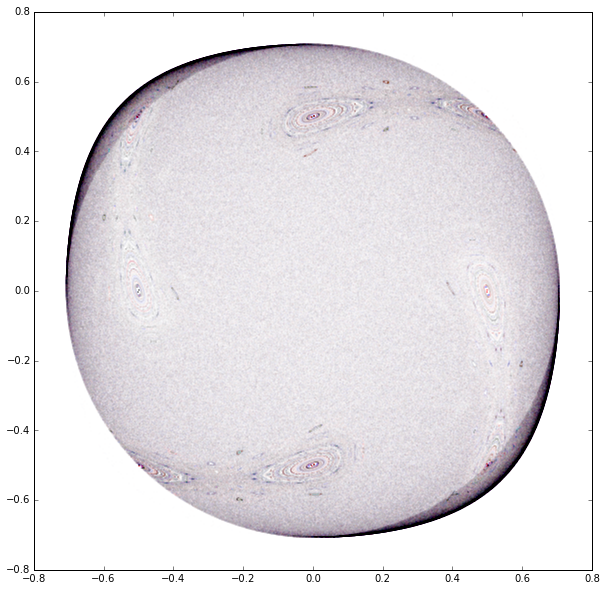}
}
\subfigure[$V = -0.2$]{
\includegraphics[scale=0.28]{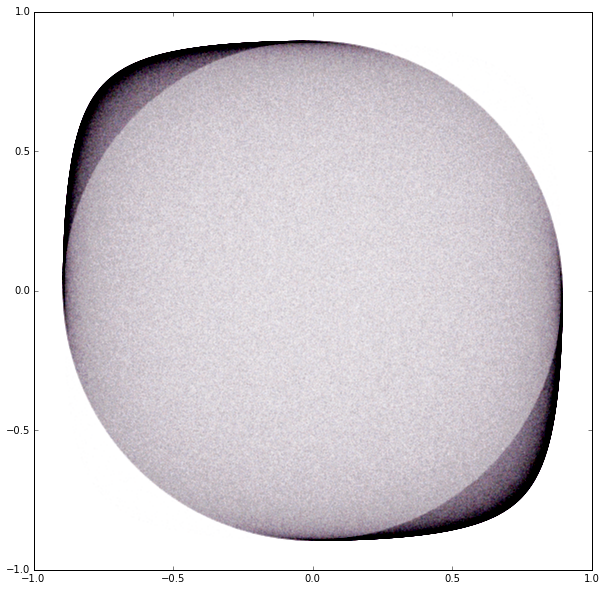}
}
\caption{Poincar\'e sections showing orbits of the Fibonacci trace map restricted to $\mathbb{S}_V$ for various values of $V < 0$. }\label{fig:tmap}
\end{figure}

\begin{thm}\label{thm:main}
 There exists $V_0 \in (-1, 0)$ such that for all $V \in (V_0, 0)$, the map $T_V$ has a locally maximal compact hyperbolic set $\Lambda_V$ in $\mathbb{S}_V$ with the following properties.
 \begin{enumerate}\itemsep0.5em
  \item The sequence $\set{\Lambda_V}$ is dynamically monotone; that is, for $V_2 > V_1$, $\Lambda_{V_2}$ contains the continuation of $\Lambda_{V_1}$.

  \item The Hausdorff distance from $\Lambda_V$ to $\mathbb{S}_0$ tends to zero as $V$ tends to zero.

  \item The Hausdorff dimension of $\Lambda_V$ tends to two as $V$ tends to zero.

  \item $\Lambda_V$ exhibits quadratic homoclinic tangencies that unfold generically with the parameter $V$.

  Moreover, there exists a residual set $\mathcal{R}\subset (V_0, 0)$ such that for all $V\in \mathcal{R}$ the following properties hold.

  \item The map $T_V$ has a nested sequence of invariant totally disconnected  hyperbolic sets $\Lambda_V^{(0)}\subseteq \Lambda_V^{(1)}\subseteq\cdots$, with $\Lambda_V^{(0)} = \Lambda_V$, and the Hausdorff dimension of $\Lambda_V^{(n)}$ tends to two as $n$ tends to infinity.

  \item The set $\Omega_V = \overline{\bigcup_{n\in\N} \Lambda_V^{(n)}}$ is a transitive invariant set of $T_V$ whose Hausdorff dimension is equal to two.

  \item Each point of $\Omega_V$ is accumulated by elliptic islands of $T_V$.
 \end{enumerate}
\end{thm}

Plots of some orbits of $T_V$ for various values of $V$ are given in Figure \ref{fig:tmap}. It is apparent that for values of $V$ close to $-1$, most of the phase space is laminated by invariant circles. As $V$ increases, elliptic islands survive, but the \emph{stochastic sea} (i.e. the set of orbits with positive Lyapunov exponent) "floods" the phase space. In Figure \ref{fig:tmap} (d) it looks like the chaotic sea fills most of the phase space completely; this is not the case, as Theorem \ref{thm:main} (7) asserts, but its Hausdorff dimension indeed increases to 2, and it indeed approaches $\mathbb{S}_0$ in the Hausdorff metric.

\subsection{Comparison with the Taylor-Chirikov standard family}\label{sec:connections}

It turns out that the dynamical properties of the family of maps $\{T_V\}_{V\in (-1,0)}$ are extremely close to the dynamics of the famous standard family (see \cite{Duarte1994,Gorodetski2012} and references therein). The main open problem in the modern  low-dimensional conservative dynamics is existence of the {\it mixed behavior} for analytic symplectic maps (see, for example, Section 6 in \cite{Pesin2010} and notice also the reference to automorphisms of $K3$ surfaces introduced by McMullen in \cite{McMullen2002} as another set of promising candidates). Mixed behavior means that together with a set of positive measure filled by the invariant families of elliptic curves (existence of this set in many cases is guarantied by the KAM theory) the stochastic sea also has positive measure. The classical and most famous specific system in this context is the standard family given by
 \begin{align}\label{eq:standard-map}
  f_k(x,y) = (x + y + k\sin(2\pi x), y + k\sin(2\pi x))\mathrm{~mod~} \Z^2.
 \end{align}
 It was suggested by Chirikov \cite{Chirikov1971} (see also \cite{Chirikov1979} for an extended exposition and \cite{Izraelev1980,Shepelyansky1995} for further discussion on the underlying physics) and Taylor\footnote{The British physicist J. B. Taylor was among the first to study the model, but the reports were never published. The first occurance of the term "Chirikov-Taylor map" traces back to \cite{Lichtenberg1992}. "Taylor-Chirikov map" is also a name given often to the standard map.} in the 1960s as a description of some phenomena in plasma physics, and also due to relations with the Frenkel-Kontorova model \cite{Kontorova1938}. Sinai conjectured that $f_k$ has positive measure-theoretic entropy for many values of the parameter $k>0$ \cite{Sinai1994}; due to Pesin's theory \cite{Pesin1977}, this is equivalent to the conjecture that $f_k$ exhibits mixed behavior.

In \cite{Gorodetski2012}, Gorodetski proved the following theorem

\begin{thm}[Gorodetski 2012]\label{thm:Gor1}
There exists $k_0 > 0$ and a residual set $\mathcal{R} \subset [k_0, \infty)$ such that for every $k\in\mathcal{R}$ there exists an infinite sequence of transitive locally maximal hyperbolic sets of the map $f_k$,
\begin{align*}
\Lambda_k^{(0)} \subseteq \Lambda_k^{(1)} \subseteq \cdots \subseteq \Lambda_k^{(n)}\subseteq\cdots
\end{align*}
that has the following properties:

\begin{enumerate}

\item The family of sets $\set{\Lambda_k^{(0)}}_{k\geq k_0}$ is dynamically increasing: for small $\epsilon > 0$, $\Lambda_{k+\epsilon}^{(0)}$ contains the continuation of $\Lambda_k^{(0)}$ at parameter $k + \epsilon$;

\item The set $\Lambda_k^{(0)}$ is $\delta_k$-dense in $\mathbb{T}^2$ for $\delta_k = \frac{4}{k^{1/3}}$;

\item The Hausdorff dimension $\hdim(\Lambda_k^{(n)})\rightarrow 2$ as $n\rightarrow \infty$;

\item $\Omega_k \eqdef \overline{\cup_{n\in\N}\Lambda_k^{(n)}}$ is a transitive invariant set of the map $f_k$, and $\hdim(\Omega_k) = 2$;

\item For any $x\in\Omega_k$, $k\in \mathcal{R}$, and any $\epsilon > 0$, the Hausdorff dimension
\begin{align*}
\hdim(B_\epsilon(x))\cap \Omega_k = \hdim(\Omega_k) = 2,
\end{align*}
where $B_\epsilon(x)$ is the open ball of radius $\epsilon$ centered at $x$;

\item Each point of $\Omega_k$ is an accumulation point of elliptic islands of the map $f_k$.

\end{enumerate}

\end{thm}

The family of hyperbolic sets $\set{\Lambda_k^{(0)}}$ that satisfies properties (1) and (2) was constructed by Duarte in \cite{Duarte1994}. He also showed that
\begin{align*}
\hdim(\Lambda_k^{(0)})\rightarrow 2\hspace{2mm}\text{ as }\hspace{2mm} k\rightarrow\infty,
\end{align*}
and that for topologically generic parameters the set $\Lambda_k^{(0)}$ is accumulated by elliptic islands. For an open set of parameters, Gorodetski's theorem provides invariant hyperbolic sets of Hausdorff dimension arbitrarily close to 2. Namely, the following theorem was also proved in \cite{Gorodetski2012}.

\begin{thm}[Gorodetski 2012]\label{thm:Gor2}
There exists $k_0 > 0$ such that for any $\xi > 0$ there exists an open and dense subset $U\subset [k_0, \infty)$ such that for every $k\in U$ the map $f_k$ has an invariant locally maximal hyperbolic set of Hausdorff dimension greater than $2 - \xi$ which is also $\delta_k$-dense in $\mathbb{T}^2$ for $\delta_k = \frac{4}{k^{1/3}}$.
\end{thm}

\begin{figure}[ht]
\subfigure[$k = 0.4$]{
\includegraphics[scale=0.28]{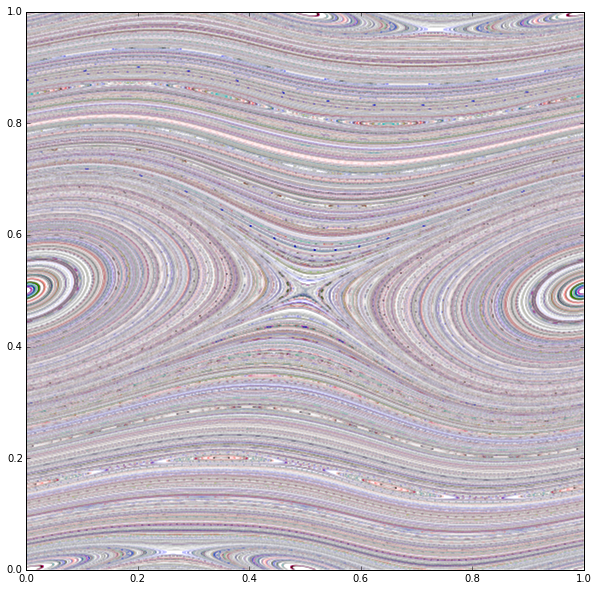}
}
\subfigure[$k = 0.8$]{
\includegraphics[scale=0.28]{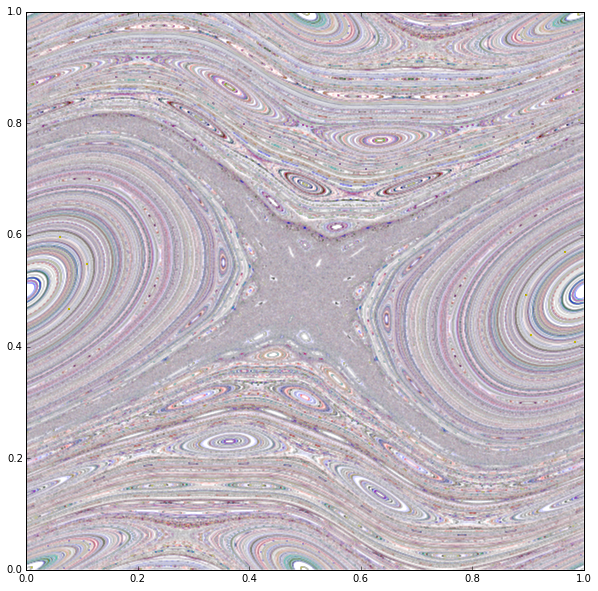}
}
\subfigure[$k = 1.5$]{
\includegraphics[scale=0.28]{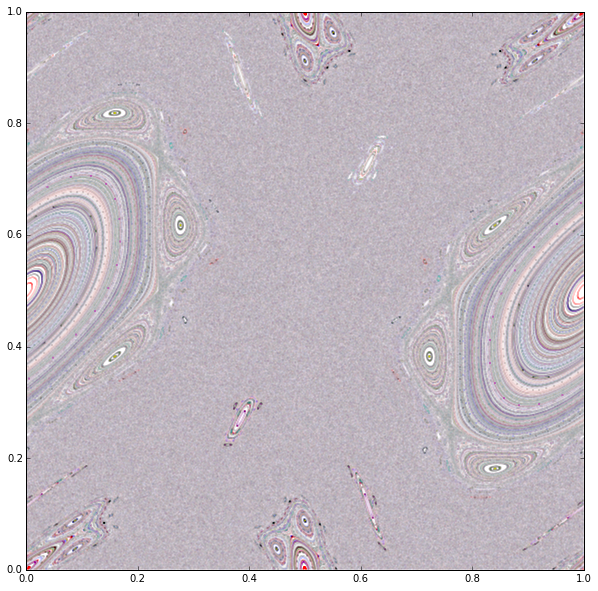}
}
\subfigure[$k = 5$]{
\includegraphics[scale=0.28]{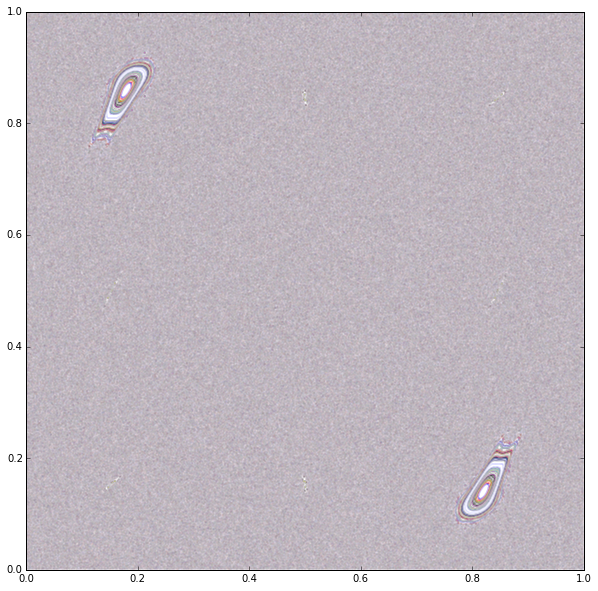}
}
\caption{Orbits of the standard map for various values of the parameter $k$. }\label{fig:stdmap}
\end{figure}

Notice that the properties of the family \eqref{eq:standard-map} that are outlined above also hold for the trace map $T_V$ as the value of $V$ approaches $0$ from below (compare Figures \ref{fig:tmap} and \ref{fig:stdmap}). Essentially all the questions about the standard family can be reformulated for the family $\set{T_V}$, and one must face the same difficulties in studying $\set{T_V}$ as in studying the standard family.

Theorems \ref{thm:Gor1} and \ref{thm:Gor2} for the standard family, and Theorem \ref{thm:main} for the Fibonacci trace map explain the difficulty associated with proving positive measure of the stochastic sea, as outlined in the introduction of \cite{Gorodetski2012}. Namely, one way to "inflate" the stochastic sea is to extend a hyperbolic set to a larger portion of the phase space through homoclinic bifurcations by varying the parameter of the family. On the other hand, small changes in the parameter create \emph{Newhouse domains} leading to loss of hyperbolicity (see \cite{Newhouse1970,Newhouse1974,Newhouse1979,Gonchenko2007,Robinson1986} for the dissipative case, and \cite{Duarte1994b,Duarte2000,Duarte2008,Gronchenko2005} for the conservative case). Then one could try to prove that the set of parameters for which Newhouse domains appear is in some sense small. Indeed, when the Hausdorff dimension of the initial Cantor set is smaller than one, the measure of the set of parameters corresponding to Newhouse domains is small and has zero density at the critical value \cite{Newhouse1976,Palis1987} (in our case the critical values are $k = \infty$ for the standard map and $V = 0$ for the trace map). When the Hausdorff dimension of the initial hyperbolic set is slightly larger than one, Palis and Yoccoz in \cite{Palis2009} give a similar result. They also conjecture that the result should hold for initial sets of arbitrary Hausdorff dimension, but for this their methods would need to be significantly sharpened (the proofs in \cite{Palis2009} are already extremely intricate); see the last paragraph on p. 3 of \cite{Palis2009}. Now we see that if one is to follow this approach, given Theorems \ref{thm:Gor1}, \ref{thm:Gor2}, and Theorem \ref{thm:main}, one must work through these difficulties.

\subsection{Simplest candidate for a map with mixed behavior, and some related questions}

Of course, as mentioned above, the main question associated with the standard family (and one we can ask about the Fibonacci trace map) is about the measure of the stochastic sea. One could ask more generally the following question \cite{Bunimovich2001,Xia2002}.

\begin{quest}
Does there exists an analytic symplectic map of a connected manifold with coexisting stochastic sea of positive measure and the KAM tori?
\end{quest}

Examples of $C^\infty$ maps with such behavior are given in \cite{Bunimovich2001,Donnay1988,Przytycki1982,Wojtkowski1981}. Some analytic non-uniformly hyperbolic maps on surfaces have also been shown to display this type of behavior \cite{Gerber2007,Gerber1985,Cantat2009}. To this day, however, no example of a real analytic symplectomorphism with mixed behavior has been given.

We believe that the family $\set{T_V}$ can be suggested alternatively to the standard family as a "simplest canditate" for an analytic symplectic map with mixed behavior; it also has the advantage of being a family of polynomial maps, so there is a chance that some of the advances in polynomial dynamics (as in \cite{Bedford1993,Bedford2004,Bedford2006}, for example) could yield a better understanding of its properties.

The following question has been posed to the author by John Franks.

\begin{quest}
Is the set of periodic points of $T_V$ for $V < 0$ dense in $\mathbb{S}_V$? At least for all $V<0$ sufficiently close to zero?
\end{quest}

We believe the answer is \emph{yes}, at least for all $V < 0$ close to zero. The techniques of the present paper do not yield this result. We should mention that the periodic points of $T_V$, $V < 0$, and other trace maps have been studied in \cite{Humphries2007,Humphries2015}, but the density question remains open. The main problem is that while every periodic point can be continued from $\mathbb{S}_0$ to $\mathbb{S}_V$ for $V < 0$ close to zero (due to hyperbolicity), it seems hard to prove that densely many of them can be continued in this way uniformly. The reason for this is that many periodic points become elliptic as they are continued along $V$. If $p$ is an elliptic periodic point on $\mathbb{S}_{V_0}$, $V_0 < 0$, then $p$ undergoes a period doubling bifurcation in $V$ for $V > V_0$, bifurcating into two hyperbolic periodic points.

\subsection*{Acknowledgment}

The author is indebted to Anton Gorodetski for suggesting this problem, and for a number of very helpful discussions.

\section{Proof of Theorem \ref{thm:main}}\label{sec:proof}

\subsection{Construction of locally maximal hyperbolic sets of large dimension}\label{sec:the-proof}

We begin with construction of compact locally maximal hyperbolic sets on $\mathbb{S}_0$ of large Hausdorff dimension. More importantly (as this will be used later in the proof), large Hausdorff dimension will follow as a consequence of large thickness, $\tau$, due to the following relation
\begin{align}\label{eq:hdim-thickness}
 \hdim(A) \geq \frac{\log 2}{\log\left(2 + \frac{1}{\tau(A)}\right)}.
\end{align}
(see Chapter 4 in \cite{Palis1993}).

Before we continue, let us recall the definition of thickness (for further details, see Chapter 4 in \cite{Palis1993}). Let $C \subset \mathbb{R}$ be a Cantor set\footnote{By a Cantor set we mean a nonempty compact perfect set with empty interior.} and denote by $I$ its convex hull. Any connected component of $I\setminus C$ is called a \emph{gap} of $C$. A \emph{presentation} of $C$ is given by an ordering $\mathcal{U} = \{U_n\}_{n \ge 1}$ of the gaps of $C$. If $u\in C$ is a boundary point of a gap $U$ of $C$, we denote by $K$ the connected component of $I\setminus (U_1\cup U_2\cup\ldots \cup U_n)$ (with $n$ chosen so that $U_n = U$) that contains $u$ and write
\begin{align*}
\tau(C, \mathcal{U}, u)=\frac{|K|}{|U|}.
\end{align*}
With this notation, the \emph{thickness} $\tau(C)$ is given by
\begin{align*}
\tau(C) = \sup_{\mathcal{U}} \inf_{u} \tau(C, \mathcal{U}, u).
\end{align*}

\begin{lem}\label{lem:horseshoes-on-S0}
 Fix $\epsilon_0 > 0$ small. For all $\epsilon \in (0, \epsilon_0]$, there exists a compact totally disconnected locally maximal hyperbolic set $\Upsilon_\epsilon\subset \mathbb{S}_0$ (away from the four conic singularities) such that
 \begin{align}\label{eq:horseshoes-on-S0-1}
  \text{for all}\hspace{2mm}\epsilon_1 > \epsilon_2,\hspace{2mm}\Upsilon_{\epsilon_1}\subseteq \Upsilon_{\epsilon_2},
 \end{align}
 and for any periodic $q\in \Upsilon_\epsilon$,
 \begin{align}\label{eq:horseshoes-on-S0-2}
  \lim_{\epsilon\rightarrow 0}\tau(\Upsilon_\epsilon\cap W_\loc^s(q)) = \infty\hspace{5mm}\text{and}\hspace{5mm}\lim_{\epsilon\rightarrow 0}\tau(\Upsilon_\epsilon\cap W_\loc^u(q)) = \infty
 \end{align}
\end{lem}

\begin{proof}
 For the sake of convenience, let us set $P = \set{P_i}_{i = 1,\dots,4}$ -- the set of conic singularities of $\mathbb{S}_0$ (see equation \eqref{eq:singularities} in Section \ref{sec:intro} above). With $\epsilon_0$ chosen sufficiently small, for $\epsilon\in (0,\epsilon_0]$, construct a Markov partition $\mathcal{P}_\epsilon$ on $\mathbb{T}^2$ for the map $\mathcal{A}$ from \eqref{eq:anosov}, satisfying the following conditions (below the map $\mathcal{F}$ is the semiconjugacy from equation \eqref{eq:factor}).
 \it
 \begin{itemize}\itemsep0.5em
  \item Each $p\in \mathcal{F}^{-1}(P)$ is contained in the interior of an element of $\mathcal{P}_\epsilon$;

  \item if $E_\epsilon$ is an element of $\mathcal{P}_\epsilon$ containing $p$, then $\diam(E) \leq \epsilon$;

%   \item $\dist(p, \partial E)\geq \epsilon/4$;

  \item If $\epsilon_1 > \epsilon_2$ and if $E_{\epsilon_1}\in\mathcal{P}_{\epsilon_1}$ and $E_{\epsilon_2}\in\mathcal{P}_{\epsilon_2}$ both contain $p$, then $E_{\epsilon_2}\subset E_{\epsilon_1}$.
 \end{itemize}
\rm
 Now projecting the interior of the elements of $\mathcal{P}_\epsilon$ containing $\mathcal{F}^{-1}(P)$ into $\mathbb{S}_0$ via $\mathcal{F}$ results in an open neighborhood $U_\epsilon$ of $P$ in $\mathbb{S}_0$. Define for convenience
 \begin{equation}\label{eq:away-from-sing}
  \mathbb{S}_{0,\epsilon} \eqdef \mathbb{S}_0\setminus U_\epsilon.
 \end{equation}
 Certainly with $\epsilon_0$ chosen sufficiently small in the beginning, for all $\epsilon\in (0,\epsilon_0]$ there exist points $x\in\mathbb{S}_{0,\epsilon}$ such that $\mathcal{O}_{T}(x)$, the full orbit of $x$ under $T$, does not intersect $U_\epsilon$ (periodic points of $T_0$ form a dense subset of $\mathbb{S}_0$). Denote the set of all such points by $\Upsilon_\epsilon$. It is easily seen that $\Upsilon_\epsilon$ is compact and invariant under $T_0$; furthermore, $\Upsilon_\epsilon$ is hyperbolic, $T_0|_{\Upsilon_\epsilon}$ being a factor of $\mathcal{A}|_{\mathcal{F}^{-1}(\Upsilon_\epsilon)}$ via the smooth map $\mathcal{F}$ (away from singularities). It is also clear that $\set{\Upsilon_\epsilon}$ forms a dynamically monotone family (i.e. equation \eqref{eq:horseshoes-on-S0-1} holds).

 Now take a point in $U_\epsilon$. Its stable and unstable manifolds are dense in $\mathbb{S}_0$ (since their lifts to $\mathbb{T}^2$ by $\mathcal{F}^{-1}$ are dense in $\mathbb{T}^2$). It follows that $\Upsilon_\epsilon$ is totally disconnected. It is also easy to see that $\Upsilon_\epsilon$ is locally maximal.

 It remains to show \eqref{eq:horseshoes-on-S0-2}.

 Pick a periodic point $q\in\Upsilon_{\epsilon_0}$, say of period $k$. Let $\tilde{U}\subset \mathbb{T}^2$ be a small neighborhood of a point $\tilde{q}$ in $\mathcal{F}^{-1}(q)$, in particular so that $\mathcal{F}(\tilde{U})\cap U_{\epsilon_0} = \emptyset$. Since the differential $D\mathcal{F}$ is bounded on $\tilde{U}$, $\mathcal{F}|_{\tilde{U}}$ is Lipschitz with a Lipschitz constant that does not depend on $\epsilon$. Thus the distortion of thickness of any subset of $\tilde{U}$ under application of $\mathcal{F}$ is bounded at most by multiplication by some positive constant that does not depend on $\epsilon$. It is therefore enough to prove
 \begin{align}\label{eq:arb-1}
 \begin{split}
  \lim_{\epsilon\rightarrow 0}\tau(\tilde{U}\cap W_\loc^s(\tilde{q})\cap \mathcal{F}^{-1}(\Upsilon_\epsilon)) &= \infty\hspace{5mm}\text{and} \\
  \lim_{\epsilon\rightarrow 0}\tau(\tilde{U}\cap W_\loc^u(\tilde{q})\cap \mathcal{F}^{-1}(\Upsilon_\epsilon)) &= \infty.
  \end{split}
 \end{align}

 Let us concentrate on the stable manifold, $W^s_\loc(\tilde{q})$ (the argument for the unstable manifold is similar). Also, henceforth we work with the map $\widetilde{\mathcal{A}} = \mathcal{A}^m$, where $m$ is a common multiple of $k$ and $6$. While this is not necessary, it ensures that $\tilde{q}$ as well as the points of $\mathcal{F}^{-1}(P)$ are fixed under $\widetilde{\mathcal{A}}$.

 Notice that $W_\loc^s(\tilde{q})\cap \mathcal{F}^{-1}(\Upsilon_\epsilon)$ is a Cantor set; let $\gamma\subset W_\loc^s(\tilde{q})$ be an arc which is the convex hull of a Cantor subset $K(\epsilon_0)$ of $W_\loc^s(\tilde{q})\cap \mathcal{F}^{-1}(\Upsilon_{\epsilon_0})$. Then for all $\epsilon\in(0,\epsilon_0]$, $\gamma$ is also the convex hull of the Cantor subset $K(\epsilon)$, and in fact for $\epsilon_1 > \epsilon_2$, $K(\epsilon_1)\subset K(\epsilon_2)$. Denote by ${E}_{\epsilon,(i,j)}$ the interior of the elements of $\mathcal{P}_\epsilon$ containing $\mathcal{F}^{-1}(P)$, with the two preimages of $P_i$, $i = 2, 3, 4$, lying in ${E}_{\epsilon,(i,1)}$ and ${E}_{\epsilon,(i,2)}$ ($P_1$ has a unique preimage); moreover, we may assume that the seven open sets $\{{E}_{\epsilon,(i,j)}\}$ are pairwise disjoint. The gaps of the Cantor set $K(\epsilon)$ in $\gamma$ can be classified into seven types, for each choice of $i_0\in \set{1,\dots,4}$ and $j_0\in\set{1,2}$ (when $i_0 = 1$, $j_0 = 1$): these are connected components of $\bigcup_{n\in\N}\widetilde{\mathcal{A}}^n({E}_{\epsilon,(i_0,j_0)})\cap \gamma$. The lengths of these connected components are controlled by the size of the rectangular elements ${E}_{\epsilon,(i_0, j_0)}$, which is in turn controlled by $\epsilon$, and the gaps for $\epsilon_1$ are embedded in those for $\epsilon_2$ if $\epsilon_1 > \epsilon_2$. We get \eqref{eq:arb-1} from this as follows.

 Let $\mathcal{U} = \set{U_n}_{n\in\N}$ be a presentation of gaps as in the definition of thickness above for the Cantor set $K(\epsilon_0)$. Take any $\epsilon\in (0, \epsilon_0)$ and  let $\tilde{\mathcal{U}} = \set{\tilde{U}_n}_{n\in\N}$ be a presentation of gaps for $K(\epsilon)$ such that $\tilde{U}_n\subset U_n$. Let
 \begin{align*}
 \delta = \min_{(i_0, j_0)}\set{\abs{E_{\epsilon_0, (i_0, j_0)}}}\hspace{2mm}\text{ and }\hspace{2mm}\tilde{\delta}=\max_{(i_0, j_0)}\set{\abs{E_{	\epsilon, (i_0, j_0)}}},
 \end{align*}
where in this case $\abs{\cdot}$ is the width of the rectangular elements of the Markov partition measured along the stable direction. Given $V$ an element of a presentation, denote its length by $\abs{V}$. Then we have
 \begin{align*}
 \frac{\abs{U_n}}{\abs{\tilde{U}_n}}\geq \frac{\delta}{\tilde{\delta}}\geq \frac{\delta}{\epsilon}.
 \end{align*}
 Now let $B_n^l$ and $B_n^r$ denote the two intervals left after removing $U_1, \dots, U_n$ from the convex hull of $K(\epsilon_0)$ such that $B_n^l$ contains one endpoint of $U_n$ and $B_n^r$ contains the other (respectively, $\tilde{B}_n^l$ and $\tilde{B}_n^r$ with $\tilde{U}$ in place of $U$). Denote the length as above by $\abs{\cdot}$. We have
 \begin{align*}
 \abs{\tilde{B}_n^{l, r}}\geq \abs{B_n^{l, r}}
 \end{align*}
and
 \begin{align*}
 \min\set{\frac{\abs{\tilde{B}_n^l}}{\abs{\tilde{U}_n}}, \frac{\abs{\tilde{B}_n^r}}{\abs{\tilde{U}_n}}} \geq \frac{\abs{U_n}}{\abs{\tilde{U}_n}}\min\set{\frac{\abs{B_n^l}}{\abs{U_n}}, \frac{\abs{B_n^r}}{\abs{U_n}}}.
 \end{align*}
It follows that
\begin{align*}
\tau(K(\epsilon))\geq \tau(K(\epsilon_0))\frac{\delta}{\epsilon}.
\end{align*}
\end{proof}

We also have the following simple
\begin{lem}\label{lem:prox-haus}
The sequence of sets $\set{\Upsilon_\epsilon}$ constructed above approaches $\mathbb{S}_0$ in the Hausdorff metric as $\epsilon\rightarrow 0$.
\end{lem}

\begin{proof}
Since for all $\epsilon_1 > \epsilon_2$, $\Upsilon_{\epsilon_1}\subset \Upsilon_{\epsilon_2}$, assuming that the statement is false, we must have an open subset of $\mathbb{S}_0$, call it $O$, such that for all $\epsilon$, $\Upsilon_\epsilon\cap O = \emptyset$. This, however, is impossible since the periodic points are dense in $\mathbb{S}_0$, and for every periodic point $p$ there exists $\epsilon > 0$ such that the orbit of $p$ does not intersect $U_\epsilon$ (the $\epsilon$ neighborhood of the singularities, as above), so that $p \in \Upsilon_\epsilon$.
\end{proof}

We are now ready to prove (1), (2), and (3) of Theorem \ref{thm:main}.

\subsection{Proof of Theorem \ref{thm:main}, (1), (2), and (3)}\label{sec:proof-1-2}
 In what follows, the notation, unless explicitly defined, coincides with that of the statement of Lemma \ref{lem:horseshoes-on-S0} and the proof thereof.

 With $\epsilon_0$ from Lemma \ref{lem:horseshoes-on-S0}, for $\epsilon\in (0, \epsilon_0]$, let $A_\epsilon$ be an open neighborhood of $\Upsilon_\epsilon$ such that $A_\epsilon\cap U_\epsilon = \emptyset$. Let $B_\epsilon$ be an open neighborhood of $\Upsilon_\epsilon$ whose compact closure lies in $A_\epsilon$ and $T_0(\overline{B_\epsilon})\subset A_\epsilon$. Clearly $T_0: B_\epsilon\rightarrow A_\epsilon$ is a diffeomorphism with a compact invariant locally maximal hyperbolic set $\Upsilon_\epsilon$, which contains $\Upsilon_\delta$ for every $\delta \in (0, \epsilon_0]$ with $\delta \geq \epsilon$. Define a family of smooth maps
 \begin{align}\label{eq:arb-3}
  \pi_{V, \epsilon}: \mathbb{S}_{0,\epsilon}\rightarrow \mathbb{S}_V,
 \end{align}
 satisfying the following properties.
 \it
 \begin{itemize}\itemsep0.5em
  \item For every $\epsilon\in(0, \epsilon_0]$, there exists $V_\epsilon\in (-1, 0)$ (increasing in $\epsilon$ and approaching zero as $\epsilon$ tends to zero) such that for all $V\in (V_\epsilon, 0)$, $\pi_{V, \epsilon}$ is a $C^2$ diffeomorphism onto its image.

  \item The family $\set{\pi_{V,\epsilon}}$ depends smoothly on $V\in (V_\epsilon, 0)$ and is continuous in $V$ at $V = 0$ (that is,  $\lim_{V\rightarrow 0}\pi_{V,\epsilon} = \mathrm{Id}|_{\mathbb{S}_{0,\epsilon}}$) in the $C^2$ topology.

  \item For any $\epsilon_2,\epsilon_1\in (0, \epsilon_0]$ with $\epsilon_2 > \epsilon_1$, for all $V\in (V_{\epsilon_1},0)$, $\pi_{V, \epsilon_2} = \pi_{V, \epsilon_1}|_{\mathbb{S}_{0, \epsilon_2}}$.
 \end{itemize}
 \rm
 \begin{rem}
  The family $\pi_{V, \bullet}$ is easy to define explicitly by considering, for example, projection along the gradient flow. Also, $C^2$ smoothness is of course not the best that one can have, but it is sufficient for our purposes.
 \end{rem}

 Now for any $\epsilon > 0$ and $V\in (V_\epsilon, 0)$,  let
 \begin{align*}
  f_{V,\epsilon}(x) = \pi_{V,\epsilon}^{-1}\circ T_V\circ\pi_{V,\epsilon}(x),
 \end{align*}
 where it is defined. Taking $V_\epsilon$ sufficiently close to zero to begin with ensures that $f_{V,\epsilon}$ is defined on $B_\epsilon$ with range in $A_\epsilon$, and in fact $f_{V,\epsilon}: B_\epsilon\rightarrow A_\epsilon$ is a $C^2$ diffeomorphism. Furthermore, $f_{V, \bullet}$ depends continuously on $V$ in the $C^2$ topology and for any $\epsilon_2,\epsilon_1\in (0,\epsilon_2]$ with $\epsilon_2 > \epsilon_1$, and any $V\in (V_{\epsilon_1}, 0)$, $f_{V,\epsilon_2}|_{B_{\epsilon_2}} = f_{V, \epsilon_1}|_{B_{\epsilon_2}}$ (this follows from the definition of $f_{V,\bullet}$ and properties of $\pi_{V,\bullet}$). Hence by considering $V$ sufficiently close to zero, we can ensure existence of compact locally maximal invariant hyperbolic sets for $f_{V,\bullet}$ of large thickness (thickness is a continuous function of the diffeomorphism in the $C^2$-topology -- see \cite[Theorem 2 in Section 4.3]{Palis1993}), and hence of large Hausdorff dimension. Projecting these hyperbolic sets back into $\mathbb{S}_V$ via $\pi_{V,\bullet}$ gives the desired dynamically monotone family of hyperbolic sets whose Hausdorff dimension approaches two as $V$ approaches zero. This proves (1) and (3). The proof of (2) follows directly from Lemma \ref{lem:prox-haus} and the properties of $\pi_{V, \bullet}$.

\subsection{Proof of Theorem \ref{thm:main}, (4)}\label{sec:proof4}

 Before we prove the three assertions of Theorem \ref{thm:main} (4) (namely existence of homoclinic tangencies; that the tangencies are quadratic; and that the tangencies unfold generically), we need some preparation, which is carried out in Section \ref{sec:technical}. We then proceed to prove Theorem \ref{thm:main} (4) in the following order. First, we hypothesize that tangencies between certain stable and unstable manifolds exist. In Section \ref{sec:quadratic} we demonstrate that if these tangencies indeed exist, they must be quadratic. In Section \ref{sec:unfolding} we prove that if the tangencies exist, then they unfold generically. Finally, in Section \ref{sec:existence} we show that the tangencies indeed exist.

 In the following sections we use standard results, notation and terminology from the theory of (normally and partially) hyperbolic dynamical systems. Good references would be \cite{Hasselblatt2006} and \cite{Hasselblatt2002}.

 \subsubsection{Technical Preparation}\label{sec:technical}

 Recall that $P_1 = (1,1,1)$ is a (fixed under $T$) singularity of $S_0$. This singularity lies on the curve $\rho$ of period-two periodic points, and this curve is normally hyperbolic. Let us denote by $W^s(\rho)$ (respectively, $W^u(\rho)$) the stable (respectively, the unstable) manifold of $\rho$. Denote the strong-stable (respectively, the strong-unstable) manifold of $P_1$ by $W^{ss}(P_1)$ (respectively, $W^{uu}(P_1)$).

 We have the following statement, which is not difficult to see and has been used without explicit proof in a number of papers \cite{Damanik2010a, Damanik2013a}, and proved explicitly (in a slightly more general context) in \cite[Lemma 4.8]{Mei2014}.

 \begin{lem}\label{lem:quadratic-surfaces}
  The stable (respectively, the unstable) manifold of $\rho$ is quadratically tangent to $\mathbb{S}_0$ along the strong-stable (respectively, the strong-unstable) manifold of $P_1$.
 \end{lem}

 \begin{proof}
  We only prove the result for the stable manifold; the analogous result for the unstable manifold follows easily from the time reversal symmetry of $T$, namely
  \begin{align}\label{eq:sym}
  T^{-1}(x,y,z) = (\sigma\circ T\circ \sigma)(x,y,z)\hspace{5mm}\text{with}\hspace{5mm}\sigma(x,y,z)=(z,y,x).
  \end{align}

  Let $F$ be a fundamental domain along $W^{ss}(P_1)$ in a small neighborhood of $P_1$, and let $\gamma$ be a twice differentiable curve lying in $W^s(\rho)$ and crossing $F$ transversally. It is known that $W^s(\rho)$ intersects transversally the surfaces $S_{V > 0}$ and does not intersect any $S_{V < 0}$ (see Proposition 4.10 in \cite{Yessen2011}). The manifold $W^s(\rho)$ is smoothly foliated by $\set{W^s(\rho)\cap S_V}_{V\geq 0}$ (see \cite[Theorem B]{Pugh1997}), and since $\gamma$ is transversal to $F$, it follows that $\gamma$ is also transversal to $W^s(\rho)\cap S_V$ for all $V > 0$ sufficiently small, and hence also to the surfaces $S_{V > 0}$ near $F$. A parameterization of $\rho$ is easily computed (by solving $T^2 = T$) to be
  \begin{align*}
   \rho(x) = \left(x, \frac{x}{2x - 1}, x\right)\hspace{2mm}\text{ for }\hspace{2mm} x \neq \frac{1}{2},
  \end{align*}
  and at $x = 1$, $\rho(1) = P_1$; moreover, as can be verified by direct computation, $\rho$ is transversal to the surfaces $S_{V > 0}$ and lies entirely in the plane $\set{x = z}$. In the plane $\set{x = z}$, $\rho$ can be viewed as the graph of the function $x\mapsto I(\rho(x))$, where $I$ is the Fricke-Vogt invariant from \eqref{e.FV}. A straightforward computation shows that $I'(\rho(x))|_{x = 1} = 0$ and $I''(\rho(x))|_{x = 1} > 0$. Now $\gamma$ can be parameterized similarly by projecting $\gamma$ along the smooth foliation $\set{W^s(\rho)\cap S_{V}}_{V > 0}$ onto an arc of $\rho$, and the claim follows.
 \end{proof}

Define a diffeomorphism $\Phi$ from a neighborhood $U$ of $P_1$ in $\mathbb{R}^3$ into $\R^3$ that rectifies the stable and unstable manifolds of $\rho$ and in addition possesses the following properties.
\begin{itemize}\itemsep0.5em
  \item $\Phi(1,1,1) = (0,0,0)$.

  \item $\Phi(W^s(\rho)\cap U)$ is part of the $xy$-plane.

  \item $\Phi(W^u(\rho)\cap U)$ is part of the $xz$-plane.

  \item $\Phi(W^{ss}(P_1))$, where $W^{ss}(P_1) = W^s(\rho)\cap\mathbb{S}_0$, is part of the non-negative $y$ axis.

  \item $\Phi(W^{uu}(P_1))$, where $W^{uu}(P_1) = W^u(\rho)\cap \mathbb{S}_0$, is part of the non-negative $z$ axis.

  \item Implied by these properties, $\Phi(\rho)$ is part of the $x$ axis.

  \item $\Phi$ maps the surfaces $U\cap \mathbb{S}_V$ with $V < 0$ into the region of $\R^3$ defined by $\set{(x, y, z)\in\R^3: y > 0, z > 0}$.
 \end{itemize}

 We assume that $\epsilon$ in the proof of Lemma \ref{lem:horseshoes-on-S0} is chosen so small, that the compact closure of the component of the neighborhood $U_\epsilon\subset\mathbb{S}_0$ of singularities that contains $P_1$ lies inside of $U$. \emph{From now on we redefine $U_\epsilon$ to denote the component of the $\epsilon$ neighborhood of singularities that contains $P_1$}.

 Let
 \begin{align}\label{eq:rect-dynamics}
  \mathcal{G}_V(x) = \Phi\circ T_V\circ\Phi^{-1}(x)\hspace{2mm}\text{ and }\hspace{2mm}\mathcal{G}(x) = \Phi\circ T\circ\Phi^{-1}(x)
 \end{align}
 for all $x\in \Phi(U)$ such that $T\circ\Phi^{-1}(x)\in U$; that is, $\mathcal{G}_V$ (respectively, $\mathcal{G}$) is the push-forward of the dynamical system $(\mathbb{S}_V, T_V)$ (respectively, $(\R^3, T)$) under the change of coordinates $\Phi$, in a neighborhood of $P_1$.

Observe that for a given $\epsilon > 0$, there exists $\delta > 0$ such that for all $V \in (-\delta, \delta)$, $\Upsilon_\epsilon$ admits a continuation $\Lambda_{V, \epsilon}$ to the surface $S_V$ (and in case $V < 0$, only $\mathbb{S}_V$ component of $S_V$ is relevant -- recall from Section \ref{sec:intro} that the points on the other four smooth components of ${S}_V$, $V < 0$, escape to infinity -- and in this case the continuation $\Lambda_{V, \epsilon}$ was constructed in the previous section).

\begin{rem}
Let us remark that $\Upsilon_\epsilon$ cannot be continued to $\mathbb{S}_V$ for all $V \in (0, -1)$ since many hyperbolic periodic points become elliptic for some values of $V < 0$. See \cite{Humphries2007} for a more detailed discussion.
\end{rem}

Extend $\pi_{V, \epsilon}$ in the previous section to the positive values of $V$ (sufficiently close to zero depending on $\epsilon$) and set
\begin{align*}
\mathbb{S}_{V, \epsilon}\eqdef \pi_{V, \epsilon}(\mathbb{S}_{0,\epsilon}).
\end{align*}
Notice that for $V$ sufficiently close to zero (that is, if we take $\delta>0$ sufficiently small above, depending on $\epsilon$), the continuation $\Lambda_{V,\epsilon}$ of $\Upsilon_\epsilon$ is contained in $\mathbb{S}_{V,\epsilon}$.
 Define
 \begin{align}\label{eq:p-hyperb}
  \Lambda_\epsilon \eqdef \bigcup_{V\in [-\frac{\delta}{2}, \frac{\delta}{2}]}\Lambda_{V, \epsilon}.
 \end{align}
 It isn't difficult to see that $\Lambda_\epsilon$ is a compact $T$-invariant partially hyperbolic subset of the smooth three-manifold $\bigcup_{V\in(-\delta, \delta)}\mathbb{S}_{V, \epsilon}$. Indeed, since $T_V$ depends smoothly on $V$, the dynamical system $T$ restricted to $\bigcup_{V\in (-\delta, \delta)}\mathbb{S}_{V,\epsilon}$ can be viewed, up to a smooth conjugacy, as a skew product of the identity map on the interval $(-\delta, \delta)$ and a smooth map on $\mathbb{S}_{0, \epsilon}$, given by
 \begin{align*}
 (V, x)\mapsto (V, \pi_{V, \epsilon}^{-1}\circ T_V\circ\pi_{V, \epsilon}(x))
 \end{align*}
 acting on $(-\delta, \delta)\times \mathbb{S}_{V, 0}$. For all the technical details, see \cite[Proposition 4.9]{Yessen2011}.

In the following proposition we work in the coordinate system induced by $\Phi$, and the neighborhood $U$ is a neighborhood of $P_1$ as above (so that $\Phi(U)$ is a neighborhood of the origin).

\begin{prop}\label{prop:c2-continuity}
Fix $\epsilon > 0$. Let $F$ denote the segment of the positive $y$ (respectively, the positive $z$) axis that lies in $\Phi(U)$ and has the origin as one of its endpoints. Parameterize the images under $\Phi$ of the center-unstable (respectively, center-stable) manifolds of $\Lambda_\epsilon$ in the neighborhood $U$ by the point of intersection with $F$, and denote this parameter by $\eta$. Call the corresponding manifold $W_\eta^{cu}$ (respectively, $W_\eta^{cs}$). Then there exists $\hat\delta > 0$ such that for all $\eta \in F$, the manifolds
\begin{align*}
W_\eta^{cs}\cap \Phi(U)\cap \bigcup_{V\in(-\hat\delta, \hat\delta)}\Phi(\mathbb{S}_V)\hspace{2mm}\text{(respectively,}\hspace{2mm}W_\eta^{cs}\cap \Phi(U)\cap \bigcup_{V\in(-\hat\delta, \hat\delta)}\Phi(\mathbb{S}_V)\text{)}
\end{align*}
approach the $xz$ (respectively, the $xy$) plane in the $C^2$ topology as $\eta$ approaches the origin.

 \end{prop}

 \begin{proof}
  We give the proof only for the center-unstable manifolds; analogous proof for the center-stable manifolds follows by either repeating the present proof with obvious modifications, or by using the symmetry \eqref{eq:sym}. Thus in what follows let $F$ denote the segment of the positive $y$ axis (which, in the coordinates induced by $\Phi$, is the (local) strong-stable manifold of the origin).

  Define a \emph{center-unstable cone field}, $K^{cu}$, in the neighborhood $\Phi(U)$ as follows. For $p\in \Phi(U)$, denote by $\mathbf{v} = (v_x, v_y, v_z)$ a vector in $\mathrm{T}_p\R^3$, the tangent space of $\R^3$ at $p$. Given $\kappa > 0$, the center-unstable cone at $p$ of size $\kappa$ is defined by
  \begin{align}\label{eq:cs-cone}
   K^{cu}_\kappa(p)\eqdef \set{\mathbf{v} \in \mathrm{T}_p\R^3: \norm{v_y} \leq \kappa(\norm{v_x} + \norm{v_z})}.
  \end{align}
  Observe that the center-unstable cone field is invariant under the action by the differential of $\mathcal{G}$ at the origin,
  \begin{align}\label{eq:differential}
   D\mathcal{G}_\mathbf{0} = \begin{pmatrix}
                              -1 & 0 & 0\\
                              0 & \frac{1}{\lambda} & 0\\
                              0 & 0 & \lambda
                             \end{pmatrix},
   \hspace{2mm}\text{where}\hspace{2mm}\lambda = \frac{\sqrt{5} + 1}{2}.
  \end{align}
  In fact, $D\mathcal{G}_\mathbf{0}(K_\kappa^{cu}(p))\subset K_{\kappa/\lambda}^{cu}(D\mathcal{G}_\mathbf{0}(p))$. It follows that for all $p$ sufficiently close to the origin, we have
  \begin{align*}
  D\mathcal{G}_p(K^{cu}_\kappa(p))\subset K^{cu}_{\kappa/\sqrt{\lambda}}(\mathcal{G}(p)).
  \end{align*}
  In particular, we may assume that the neighborhood $U$ is taken sufficiently small to begin with.

We need the following result which can be derived from general principles; for completeness, we include an outline of the main ideas of the proof.

\begin{lem}\label{lem:cu-transversality}
For all $\epsilon > 0$ there exists $\hat\delta > 0$ such that for all $V\in (-\hat\delta, \hat\delta)$, the center-stable and the center-unstable manifolds of $\Lambda_\epsilon$ intersect the surfaces $S_V$ transversally. In particular, they intersect the surface $\mathbb{S}_V$ transversally.
\end{lem}

\begin{proof}[Outline of the proof]
For any $V\in(-\hat\delta, \hat\delta)$, denote by $H_V: \Lambda_{0, \epsilon}\rightarrow \Lambda_{V, \epsilon}$ a homeomorphism conjugating $T|_{\Lambda_{0, \epsilon}}$ to $T|_{\Lambda_{V, \epsilon}}$. Then from \cite[Theorem 7.3]{Hirsch1968} we know that $H_V$ is unique and is the fixed point of a certain contraction on a certain Banach space of continuous maps (see the proof of Theorem 7.3 in \cite{Hirsch1968} for the details); furthermore, since $T|_{\mathbb{S}_{V, \epsilon}}$ is smooth, and hence Lipschitz continuous, then this contraction is also Lipschitz in $V$. It follows that $H_V$ is also Lipschitz continuous in the $C^0$ topology. It follows that for any fixed $x\in \Lambda_{0, \epsilon}$, the smooth curve $V\mapsto H_V(x)$ satisfies
\begin{align*}
\norm{H_0(x) - H_V(x)}\leq C\abs{V},
\end{align*}
with a constant $C > 0$ that does not depend on $V$ or on $x$. It follows that the curve $V\mapsto H_V(x)$ is transversal to the surfaces $\mathbb{S}_{V, \epsilon}$, and since this curve is the intersection of the center-stable and the center-unstable manifolds that contain the point $x$, these manifolds must also be transversal to $\mathbb{S}_{V, \epsilon}$.
\end{proof}

 It follows that there exists $\hat\delta > 0$ such that for all $\eta \in F$ and $V\in (-\hat\delta, \hat\delta)$, ${W}_\eta^{cu}\cap \Phi(U)$ intersects $\Phi(\mathbb{S}_V\cap U)$ transversally. Now take a fundamental domain along $F$, and call it $\hat{F}$. Then there exists $\kappa > 0$ such that for every $\eta\in \hat{F}$
 \begin{align*}
 \text{if}\hspace{2mm} p\in W_\eta^{cu}\cap \Phi(U)\cap \bigcup_{V\in (-\hat\delta, \hat\delta)}\Phi(\mathbb{S}_V)\hspace{2mm}\text{then}\hspace{2mm}\mathrm{T}_p \Phi(W_\eta^{cu})\subset K^{cu}_\kappa(p).
 \end{align*}

  By invariance and contraction of the center-unstable cones under the action of $D\mathcal{G}$, it follows that for any fixed $\eta\in \hat{F}$, the images of
  \begin{align*}
  W_\eta^{cu}\cap \Phi(U)\cap\bigcup_{V\in(-\hat\delta, \hat\delta)}\Phi(\mathbb{S}_V)
  \end{align*}
  under the iteration by $\mathcal{G}$ approach the $xz$ plane in the $C^1$ topology. Equivalently, as $\eta$ tends to the origin,
  \begin{align*}
  W_\eta^{cu}\cap \Phi(U)\cap\bigcup_{V\in (-\hat\delta, \hat\delta)}\Phi(\mathbb{S}_V)
  \end{align*}
  approaches the $xz$ plane in the $C^1$ topology. This is slightly coarser than what we need. To improve this result, we consider the action of $(\mathcal{G}, D\mathcal{G})$ on the tangent bundle of $W_\eta^{cu}$, $TW_\eta^{cu}$. Elementary computation shows that at the origin $(\mathbf{0},\mathbf{0})\in\R^3\times\R^3$,
  \begin{align*}
  \renewcommand*{\arraystretch}{1.5}
   D(\mathcal{G}, D\mathcal{G})(\mathbf{0},\mathbf{0}) = \left(\begin{array}{c|c}
    D\mathcal{G}_\mathbf{0} & 0\\
    \hline
    0 & D\mathcal{G}_\mathbf{0}
   \end{array}\right).
  \end{align*}
  Thus near $(\mathbf{0}, \mathbf{0})$, we can apply to $(\mathcal{G}, D\mathcal{G})$ a cone condition analogous to the one for $\mathcal{G}$ to conclude that $TW_\eta^{cu}$ approaches the tangent bundle of the $xz$ plane as $\eta$ approaches the origin; in other words, $W_\eta^{cu}$ approaches the $xz$ plane in the $C^2$ topology.
 \end{proof}

 \begin{lem}\label{lem:annoying}
  Suppose $G = (G_x, G_y, G_z): \R^3\rightarrow \R^3$ is a $C^2$ diffeomorphism. Let $M > 0$, $C > 1$. Then there exists $\xi_0 \in (0, \lambda-1)$, with $\lambda$ as in \eqref{eq:differential}, such that for all $\xi\in(0, \xi_0)$ the following holds.

  Suppose that $G$ satisfies the following conditions.
  \begin{enumerate}[(i)]\itemsep0.5em
   \item $\norm{\nabla\partial_\beta G_\alpha} < M$ for $\alpha, \beta\in \set{x,y,z}$.

   \item $G$ is $\xi$-close to $D\mathcal{G}_\mathbf{0}$ from \eqref{eq:differential} in the $C^1$ topology.
  \end{enumerate}
  Assume also that $\gamma = (\gamma_x,\gamma_y,\gamma_z)$ is a $C^2$ curve in $\R^3$ that satisfies the following conditions.
  \begin{enumerate}[(1)]\itemsep0.5em
   \item $\norm{\gamma'(t)} < C$ for all $t$;

   \item $\gamma_x(t) = t$;

   \item $\frac{\lambda - \xi - 1}{4}\abs{\gamma_z''(t)} > 3MC^2$ for all $t$;

   \item $\frac{\lambda - \xi - 1}{4}\abs{\gamma_z''(t)} > \xi\abs{\gamma_y''(t)}$ for all $t$.
  \end{enumerate}
  Then we have
  \begin{align}\label{eq:arb-6}
   \abs{\frac{d^2}{dt^2}G_z\circ\gamma(t)} > \abs{\gamma_z''(t)} + \frac{\lambda - \xi - 1}{2}\abs{\gamma_z''(t)}\hspace{2mm}\text{for all}\hspace{2mm}t
  \end{align}
  and
  \begin{align}\label{eq:arb-7}
   \frac{\lambda-\xi - 1}{4}\abs{\frac{d^2}{dt^2}G_z\circ\gamma(t)}> \xi\abs{\frac{d^2}{dt^2}G_y\circ\gamma(t)} + \frac{(\lambda - 1)^2}{8}\abs{\gamma_z''(t)}\hspace{2mm}\text{for all}\hspace{2mm} t.
  \end{align}
 \end{lem}

 \begin{proof}

 With the hypothesis (2) in mind, a direct computation gives
  \begin{align}\label{eq:aux1}
  \begin{split}
   \frac{d^2}{dt^2}G_z\circ\gamma(t) = &\nabla\partial_x G_z\cdot\gamma' + \gamma_y'\nabla\partial_y G_z\cdot\gamma' \\
   + &\gamma_y''\partial_yG_z + \gamma_z'\nabla\partial_zG_z\cdot\gamma' + \gamma_z''\partial_zG_z.
   \end{split}
  \end{align}
  An application of the triangle and the Cauchy-Schwartz inequalities together with our hypothesis (i) and (1) gives
  \begin{align*}
  \abs{\nabla\partial_x G_z\cdot\gamma'+\gamma_y'\nabla\partial_y G_z\cdot\gamma'+\gamma_z'\nabla\partial_zG_z\cdot\gamma'}\leq 3MC^2.
  \end{align*}
  On the other hand, since by hypothesis (2) $G$ is $\xi$ close to $D\mathcal{G}_{\mathbf{0}}$ in the $C^1$ topology, we have
  \begin{align*}
  \abs{\gamma_y''\partial_yG_z}\leq \xi\abs{\gamma_y''}\hspace{2mm}\text{ and }\hspace{2mm}\abs{\gamma_z''\partial_zG_z}\leq (\lambda-\xi)\abs{\gamma_z''}.
  \end{align*}
  Thus from \eqref{eq:aux1} together with (3) and (4) we get
  \begin{align}\label{eq:arb-8}
  \begin{split}
   \abs{\frac{d^2}{dt^2}G_z\circ\gamma(t)} & > \abs{\gamma_z''} + (\lambda - \xi - 1)\abs{\gamma_z''} - 3MC^2 - \xi\abs{\gamma_y''} \\
   & > \abs{\gamma_z''} + \frac{\lambda - \xi - 1}{2}\abs{\gamma_z''},
   \end{split}
  \end{align}
  and \eqref{eq:arb-6} follows.

  Now observe that
  \begin{align*}
   \frac{d^2}{dt^2}G_y\circ\gamma(t) = &\nabla\partial_x G_y\cdot \gamma' + \gamma_y'\nabla\partial_yG_y\cdot\gamma' \\
   + &\gamma_y''\partial_yG_y + \gamma_z'\nabla\partial_zG_y\cdot\gamma' + \gamma_z'' \partial_zG_y,
  \end{align*}
  which by the hypotheses (i), (ii), (1), and (2) gives
  \begin{align}\label{eq:arb-9}
   \abs{\frac{d^2}{dt^2}G_y\circ\gamma(t)} < 3MC^2 + \left(\frac{1}{\lambda} + \xi\right)\abs{\gamma_y''} + \xi\abs{\gamma_z''}.
  \end{align}
  From the inequalities \eqref{eq:arb-8} and \eqref{eq:arb-9} in conjunction with the hypotheses (3) and (4), we get
  \begin{align*}
   & \frac{\lambda - \xi - 1}{4}\abs{\frac{d^2}{dt^2}G_z\circ\gamma(t)} - \xi \abs{\frac{d^2}{dt^2}G_y\circ\gamma(t)}\\
   & > \left(\frac{\lambda - \xi - 1}{4} + \frac{(\lambda - \xi - 1)^2}{8}\right)\abs{\gamma_z''} - 3\xi MC^2 - \xi\left(\frac{1}{\lambda} + \xi\right)\abs{\gamma_y''} - \xi^2\abs{\gamma_z''}\\
   & > \left(\frac{\lambda - \xi - 1}{4} + \frac{(\lambda - \xi - 1)^2}{8} - \frac{\xi(\lambda - \xi - 1)}{4} - \frac{\lambda - \xi - 1}{4}\left(\frac{1}{\lambda} + \xi\right) - \xi^2\right)\abs{\gamma_z''}
  \end{align*}
  Now observe that the right most side of the inequality above tends to
  \begin{align*}
   \left(\left(1 - \frac{1}{\lambda}\right)\frac{\lambda - 1}{4} + \frac{(\lambda - 1)^2}{8}\right)\abs{\gamma_z''}
  \end{align*}
  as $\xi$ tends to zero. So for all $\xi > 0$ sufficiently small,
  \begin{align*}
  \frac{\lambda - \xi - 1}{4}\abs{\frac{d^2}{dt^2}G_z\circ\gamma(t)} - \xi \abs{\frac{d^2}{dt^2}G_y\circ\gamma(t)} > \frac{(\lambda - 1)^2}{8}\abs{\gamma_z''(t)},
  \end{align*}
  and \eqref{eq:arb-7} follows.
 \end{proof}

 \begin{lem}\label{lem:annoying-2}
  Let $M > 0$, $C_0 > 0$, and $\Delta_0 = \frac{1}{4}$. There exist $\xi_0 = \xi_0(C_0, \Delta_0) > 0$ not larger than $\xi_0$ from Lemma \ref{lem:annoying} such that for all $\xi \in (0, \xi_0)$ and $\Delta\in(0,\Delta_0)$ we have the following.

  Suppose that $G$ is as in Lemma \ref{lem:annoying}. Assume that $\gamma$ is also as in Lemma \ref{lem:annoying}, with $C = 1 + C_0 + \Delta_0$ and the additional assumption that $\abs{\gamma_y'(t)} < C_0$ and $\abs{\gamma_z'(t)} < \Delta$ for all $t$. Assume that for some $m\in\N$ and all $n\in\set{1, \dots, m}$ the following hold.

  \begin{itemize}\itemsep0.5em

  \item $G^n\circ\gamma$ can be reparameterized as $\gamma^{(n)}(s) = (s, \gamma_y^{(n)}(s), \gamma_z^{(n)}(s))$ (that is, $\abs{\frac{d}{dt}G_x^n\circ\gamma} > 0$), with $\gamma^{(0)} = \gamma$.

  \item $\abs{\frac{d}{ds}\gamma_y^{(n)}(s)} < C_0$ and $\abs{\frac{d}{ds}\gamma_z^{(n)}(s)} < \Delta$ for all $s$.
  \end{itemize}

  Then we have $\abs{\frac{d^2}{ds^2}\gamma_z^{(n)}(\tilde{s})} > \abs{\frac{d^2}{dt^2}\gamma_z^{(n - 1)}(\tilde{t})}$ for all $n\in\set{1,\dots,m-1}$ and $\tilde{t}$, $\tilde{s}$ such that $\gamma^{(n)}(\tilde{s})=\gamma^{(n-1)}(\tilde{t})$.
 \end{lem}

\begin{rem}\label{rem:notation}
Each time that we take the image of $\gamma^{(n-1)}$ under $G$, we \emph{reparameterize} $G\circ\gamma^{(n-1)}(t)$ to obtain $\gamma^{(n)}(s)$ (i.e. the curve whose image coincides with that of $G\circ\gamma^{(n-1)}$, but $\frac{d}{ds}\gamma_x^{(n)}(s) = 1$). In this context, $s$ is always the parameter of $\gamma^{(n)}$ (with the above property), and $t$ is the parameter of $\gamma^{(n-1)}$, also with the property that $\frac{d}{dt}\gamma_x^{(n-1)}(t) = 1$.
\end{rem}

 \begin{proof}[Proof of Lemma \ref{lem:annoying-2}]

From the hypothesis we have for all $n\in\set{1,\dots,m}$,
\begin{align*}
 \abs{\frac{d}{ds}\gamma_y^{(n)}(\tilde{s})} < C_0\hspace{2mm}\text{ and }\hspace{2mm}\abs{\frac{d}{ds}\gamma_z^{(n)}(\tilde{s})}< \Delta
\end{align*}
for all $\tilde{s}$. Hence for all $n\in\set{1,\dots,m}$, $\gamma^{(n)}$ satisfies the first two conditions of Lemma \ref{lem:annoying}:
\begin{itemize}\itemsep0.5em
   \item $\norm{\frac{d}{ds}\gamma^{(n)}(\tilde{s})} < C$ for all $\tilde{s}$;
   \item $\gamma_x^{(n)}(s) = s$.
  \end{itemize}

  \begin{claim}\label{c:annoing-2-c2}
   If for $n\in\set{1,\dots,m}$, $\gamma^{(n-1)}$ satisfies the last two conditions of Lemma \ref{lem:annoying}, then $\gamma^{(n)}$ also satisfies those conditions; that is,
   \begin{enumerate}\itemsep0.5em
    \item $\frac{\lambda - \xi - 1}{4}\abs{\frac{d^2}{ds^2}\gamma_z^{(n)}(\tilde{s})} > 3MC^2$ for all $\tilde{s}$;
    \item $\frac{\lambda - \xi - 1}{4}\abs{\frac{d^2}{ds^2}\gamma_z^{(n)}(\tilde{s})} > \xi\abs{\frac{d^2}{ds^2}\gamma_y^{(n)}(\tilde{s})}$ for all $\tilde{s}$;
   \end{enumerate}
   and, moreover,
   \begin{enumerate}\itemsep0.5em
   \setcounter{enumi}{2}

   \item $\abs{\frac{d^2}{ds^2}\gamma_z^{(n)}(\tilde{s})} > \abs{\frac{d^2}{dt^2}\gamma_z^{(n-1)}(\tilde{t})}$ for all $\tilde{t}$, where $\tilde{t}$ and $\tilde{s}$ are such that $\gamma^{(n)}(\tilde{s})=G\circ\gamma^{(n-1)}(\tilde{t})$.
   \end{enumerate}
  \end{claim}

   Observe that (3) implies (1). Let us first establish (3), and then prove (2).

   \begin{proof}[Proof of Claim \ref{claim:arb2} (3)]
   Since $\gamma^{(n-1)}$ satisfies the last two conditions of Lemma \ref{lem:annoying}, and as we have seen, for every $k\in\set{1,\dots,m}$, $\gamma^{(k)}$ satisfies the first two conditions of Lemma \ref{lem:annoying}, $\gamma^{(n-1)}$ satisfies the hypothesis of Lemma \ref{lem:annoying}. Application of the lemma then gives
   \begin{align*}
    \abs{\frac{d^2}{dt^2}G_z\circ\gamma^{(n-1)}(\tilde{t})} > \abs{\frac{d^2}{dt^2}\gamma_z^{(n-1)}(\tilde{t})} + \frac{\lambda - \xi - 1}{2}\abs{\frac{d^2}{dt^2}\gamma_z^{(n-1)}(\tilde{t})}
   \end{align*}
   and
   \begin{align*}
    \frac{\lambda-\xi - 1}{4}\abs{\frac{d^2}{dt^2}G_z\circ\gamma^{(n-1)}(\tilde{t})} > \xi\abs{\frac{d^2}{dt^2}G_y\circ\gamma^{(n-1)}(\tilde{t})}
    + \frac{(\lambda - 1)^2}{8}\abs{\frac{d^2}{dt^2}\gamma^{(n-1)}_z(\tilde{t})}
   \end{align*}
   for all $\tilde{t}$. We now deduce the conclusion of Claim \ref{c:annoing-2-c2} (3) from these two inequalities.

Observe that
   \begin{align*}
   \gamma_z^{(n)}(\tilde{s})=G_z\circ \gamma^{(n-1)}(\tilde{t})\hspace{2mm}\text{ with }\hspace{2mm}\tilde{t}=\left(G_x\circ \gamma^{(n-1)}\right)^{-1}(\tilde{s}).
   \end{align*}
   So we have
   \begin{align}\label{eq:arb-14}
   \begin{split}
    \frac{d^2}{ds^2}\gamma_z^{(n)}(\tilde{s}) = &\\
    &\left(\frac{d}{dt}G_x\circ\gamma^{(n-1)}(\tilde{t})\right)^{-3}\frac{d^2}{dt^2}G_z\circ\gamma^{(n-1)}(\tilde{t})\frac{d}{dt}G_x\circ\gamma^{(n-1)}(\tilde{t})\\
    - &\left(\frac{d}{dt}G_x\circ\gamma^{(n-1)}(\tilde{t})\right)^{-3}\frac{d}{dt}G_z\circ\gamma^{(n-1)}(\tilde{t})\frac{d^2}{dt^2}G_x\circ\gamma^{(n-1)}(\tilde{t}).
    \end{split}
   \end{align}
   Keeping in mind that $\gamma_x^{(n-1)}(\tilde{t}) = \tilde{t}$, let us estimate from above the second derivative of $G_x\circ\gamma^{(n-1)}$ at $\tilde{t}$. We have
   \begin{align*}
    \abs{\frac{d^2}{dt^2}G_x\circ\gamma^{(n-1)}(\tilde{t})}- 3MC^2 = &\\
    &\left|\nabla\partial_xG_x(\gamma^{(n-1)}(\tilde{t}))\cdot\frac{d}{dt}\gamma^{(n-1)}(\tilde{t})\right. + \\
    &\frac{d}{dt}\gamma_y^{(n-1)}(\tilde{t})\nabla\partial_yG_x(\gamma^{(n-1)}(\tilde{t}))\cdot\frac{d}{dt}\gamma^{(n-1)}(\tilde{t})\\
    &+ \frac{d^2}{dt^2}\gamma_y^{(n-1)}(\tilde{t})\partial_yG_x(\gamma^{(n-1)}(\tilde{t})) \\
    &+ \frac{d}{dt}\gamma_z^{(n-1)}(\tilde{t})\nabla\partial_zG_x(\gamma^{(n-1)}(\tilde{t}))\cdot\frac{d}{dt}\gamma^{(n-1)}(\tilde{t}) \\
    &+ \left.\frac{d^2}{dt^2}\gamma_z^{(n-1)}(\tilde{t})\partial_zG_x(\gamma^{(n-1)}(\tilde{t}))\right|\\
    & - 3MC^2,
   \end{align*}
   and the right side is strictly smaller than
   \begin{align*}
        &\xi\abs{\frac{d^2}{dt^2}\gamma_y^{(n-1)}(\tilde{t})} + \xi\abs{\frac{d^2}{dt^2}\gamma_z^{(n-1)}(\tilde{t})}\\
        < &\frac{\lambda - \xi - 1}{4}\abs{\frac{d^2}{dt^2}\gamma_z^{(n-1)}(\tilde{t})} + \xi\abs{\frac{d^2}{dt^2}\gamma_z^{(n-1)}(\tilde{t})}.
    \end{align*}
  So we have
  \begin{align}\label{eq:arb-13}
  \abs{\frac{d^2}{dt^2}G_x\circ\gamma^{(n-1)}(\tilde{t})} < \frac{\lambda - \xi - 1}{4}\abs{\frac{d^2}{dt^2}\gamma_z^{(n-1)}(\tilde{t})} + \xi\abs{\frac{d^2}{dt^2}\gamma_z^{(n-1)}(\tilde{t})} + 3MC^2.
  \end{align}
   This estimate then gives the following estimate on the second derivative of $\gamma_z^{(n)}$ from below.
   \begin{align}\label{eq:arb-15}
   \begin{split}
    \abs{\frac{d^2}{ds^2}\gamma_z^{(n)}(\tilde{s})} & \geq \abs{\frac{d}{dt}G_x\circ\gamma^{(n-1)}(\tilde{t})}^{-3} \times \\
    &\times \left(\abs{\frac{d^2}{dt^2}G_z\circ\gamma^{(n-1)}(\tilde{t})}\abs{\frac{d}{dt}G_x\circ\gamma^{(n-1)}(\tilde{t})}\right.\\
    &- 3\Delta MC^2\\
    & - \frac{\lambda - \xi - 1}{4}\Delta\abs{\frac{d^2}{dt^2}\gamma_z^{(n-1)}(\tilde{t})}\\
    & \left.- \Delta\xi\abs{\frac{d^2}{dt^2}\gamma_z^{(n-1)}(\tilde{t})}\right),
   \end{split}
   \end{align}
   which is strictly larger than
   \begin{align}\label{eq:arb-16}
   \begin{split}
    & \abs{\frac{d}{dt}G_x\circ\gamma^{(n-1)}(\tilde{t})}^{-3} \times \\
    &\times \left(\abs{\frac{d^2}{dt^2}\gamma_z^{(n-1)}(\tilde{t})}\right.\abs{\frac{d}{dt}G_x\circ\gamma^{(n-1)}(\tilde{t})}\\
    &+ \frac{\lambda - \xi - 1}{2}\abs{\frac{d^2}{dt^2}\gamma_z^{(n-1)}(\tilde{t})}\abs{\frac{d}{dt}G_x\circ\gamma^{(n-1)}(\tilde{t})}\\
    & - 3\Delta MC^2 \\
    &- \frac{\lambda - \xi - 1}{4}\Delta\abs{\frac{d^2}{dt^2}\gamma_z^{(n-1)}(\tilde{t})} \\
    &\left.- \Delta\xi\abs{\frac{d^2}{dt^2}\gamma_z^{(n-1)}(\tilde{t})}\right).
    \end{split}
   \end{align}
   This quantity is larger than $\abs{\frac{d^2}{dt^2}\gamma_z^{(n-1)}(\tilde{t})}$ if and only if
   \begin{align*}
    & \abs{\frac{d}{dt}G_x\circ\gamma^{(n-1)}(\tilde{t})}^{-3} \times\\
    \times & \left(\abs{\frac{d}{dt}G_x\circ\gamma^{(n-1)}(\tilde{t})} + \frac{\lambda - \xi - 1}{2}\abs{\frac{d}{dt}G_x\circ\gamma^{(n-1)}(\tilde{t})}\right. \\
    & \left.- \frac{\lambda - \xi - 1}{4}\Delta - \frac{\lambda - \xi - 1}{4}\Delta - \Delta\xi \right)\\
    & > 1.
   \end{align*}
   Now notice that
   \begin{align}\label{eq:arb-12}
   \begin{split}
    1 + \xi(1 + \Delta_0 + C_0) & > \abs{\frac{d}{dt}G_x\circ\gamma^{(n-1)}(\tilde{t})} \\
    &= \left|\partial_xG_x(\gamma^{(n-1)}(\tilde{t}))\right.\\
    &+ \frac{d}{dt}\gamma_y^{(n-1)}(\tilde{t})\partial_yG_x(\gamma^{(n-1)}(\tilde{t}))\\
    &+ \left.\frac{d}{dt}\gamma_z^{(n-1)}(\tilde{t})\partial_zG_x(\gamma^{(n-1)}(\tilde{t}))\right|\\
    &> 1 - \xi(1 + \Delta_0 + C_0).
   \end{split}
   \end{align}
   Since $\Delta < \Delta_0$, this combined with \eqref{eq:arb-12} ensures that the previous inequality holds provided that
   \begin{align*}
    \frac{u}{v^3}\left(1+ \frac{\lambda - \xi - 1}{2}\right)
    - \frac{1}{u^3}\left(\frac{\Delta_0(\lambda - \xi - 1)}{2} + \Delta_0\xi\right) > 1,
   \end{align*}
   with
   \begin{align*}
   u = 1 - \xi(1 + \Delta_0 + C_0)
   \end{align*}
   and
   \begin{align*}
   v = 1 + \xi(1 + \Delta_0 + C_0).
   \end{align*}
   This inequality holds provided that $\xi_0$ is chosen sufficiently small to begin with (recall that $\Delta_0 = \frac{1}{4}$ is fixed).
   \end{proof}

   Now we proceed to show (2) of Claim \ref{claim:arb2}: $\frac{\lambda - \xi - 1}{4}\abs{\frac{d^2}{ds^2}\gamma_z^{(n)}(\tilde{s})} > \xi\abs{\frac{d^2}{ds^2}\gamma_y^{(n)}(\tilde{s})}$ for all $\tilde{s}$.

   \begin{proof}[Proof of Claim \ref{claim:arb2} (2)]

   The expression for the second derivative of $\gamma_y^{(n)}$ is similar to that from \eqref{eq:arb-14}, with $\gamma_y$ in place of $\gamma_z$. So the desired inequality holds if
   \begin{align*}
    \frac{\lambda - \xi - 1}{4}&\left|\frac{d^2}{dt^2}G_z\circ\gamma^{(n-1)}(\tilde{t})\right.\frac{d}{dt}G_x\circ\gamma^{(n-1)}(\tilde{t})- \left.\frac{d}{dt}G_z\circ\gamma^{(n-1)}(\tilde{t})\frac{d^2}{dt^2}G_x\circ\gamma^{(n-1)}(\tilde{t})\right| \\
    - \xi&\abs{\frac{d^2}{dt^2}G_y\circ\gamma^{(n-1)}(\tilde{t})\frac{d}{dt}G_x\circ\gamma^{(n-1)}(\tilde{t}) - \frac{d}{dt}G_y\circ\gamma^{(n-1)}(\tilde{t})\frac{d^2}{dt^2}G_x\circ\gamma^{(n-1)}(\tilde{t})}\\
    > & 0.
   \end{align*}
   This inequality holds if
   \begin{align*}
    \frac{\lambda - \xi - 1}{4}&\abs{\frac{d^2}{dt^2}G_z\circ\gamma^{(n-1)}(\tilde{t})\frac{d}{dt}G_x\circ\gamma^{(n-1)}(\tilde{t})}\\
    -\frac{\lambda - \xi - 1}{4}&\abs{\frac{d}{dt}G_z\circ\gamma^{(n-1)}(\tilde{t})\frac{d^2}{dt^2}G_x\circ\gamma^{(n-1)}(\tilde{t})}\\
    -\xi&\abs{\frac{d^2}{dt^2}G_y\circ\gamma^{(n-1)}(\tilde{t})\frac{d}{dt}G_x\circ\gamma^{(n-1)}(\tilde{t})}\\
    -\xi&\abs{\frac{d}{dt}G_y\circ\gamma^{(n-1)}(\tilde{t})\frac{d^2}{dt^2}G_x\circ\gamma^{(n-1)}(\tilde{t})} \\
    > &0.
   \end{align*}
   From equations \eqref{eq:arb-12} and \eqref{eq:arb-13}, together with the conditions that are satisfied by $\abs{\frac{d^2}{dt^2}G_z\circ\gamma^{(n-1)}(\tilde{t})}$ and $\abs{\frac{d^2}{dt^2}G_y\circ\gamma^{(n-1)}(\tilde{t})}$, and the fact that $\Delta_0 = \frac{1}{4}$, it follows that the above inequality holds if the following holds:
   \begin{align*}
    &\left(\frac{(\lambda - 1)^2}{8}u - \left(\frac{\lambda - \xi - 1}{16} + \xi C_0\right)\left(\frac{\lambda - \xi - 1}{2} + \xi\right)\right)\abs{\frac{d^2}{dt^2}G_z\circ\gamma^{(n-1)}(\tilde{t})} > 0,
   \end{align*}
   with $u = 1 - \xi(1 + \Delta_0 + C_0)$ as above, which certainly holds for all $\xi$ sufficiently small (which can be ensured by taking $\xi_0$ small to begin with).
  \end{proof}

  The conclusion of Lemma \ref{lem:annoying-2} follows by induction on $n$, repeatedly applying Claim \ref{c:annoing-2-c2} starting with $\gamma^{(0)}$.
  \end{proof}

 \subsubsection{Order of tangencies}\label{sec:quadratic}

 In this section we prove that the tangencies whose existence will be proved in Section \ref{sec:existence} are quadratic. In what follows, $U$ and $U_\epsilon$ are as above (see in particular the paragraph preceding equation \eqref{eq:rect-dynamics}).

 Let us fix $\epsilon > 0$, so small that $\overline{U}_\epsilon$ lies entirely inside $U$. Let $F^s$ and $F^u$, respectively, denote open intervals along the positive $y$ and the positive $z$ axes that lie entirely inside $\Phi(U)\setminus\Phi(\overline{U}_\epsilon)$. For $\delta > 0$ small, denote by $B_\delta^s$ and $B_\delta^u$, respectively, the strips $(-\delta, \delta)\times F^s$ and $(-\delta,\delta)\times F^u$ in the $xy$ and $xz$ planes. Assuming that $\delta$ is sufficiently small, the orthogonal projection from the part of $\Phi(\mathbb{S}_{0,\epsilon})$ into the $xz$ plane that projects onto $B_\delta^u$ is a $C^\infty$ diffeomorphism. We may assume that $F^s$ (respectively, $F^u$) contains a fundamental domain for $\mathcal{G}$ (otherwise we can take smaller $\epsilon$), so that it contains intersections with the unstable (respectively, stable) mnifolds on $\mathbb{S}_{0,\epsilon}$. Now, assuming that $\delta > 0$ is sufficiently small, the image of the orthogonal projection into $B_\delta^u$ of the stable manifolds of $\Phi(\Lambda_{0,\epsilon})$ that intersect $F^u$ can be viewed as graphs of functions $\mathfrak{s}: x\rightarrow z$ (since they intersect $F^u$ transversally). Let us parameterize the family of these functions (or, equivalently, their graphs) by the point of intersection of the graph with the $z$ axis, call this parameter $\eta$, and denote the corresponding function by $\mathfrak{s}_\eta$. Clearly there exists $\Delta > 0$ such that for each $\eta$, $\abs{\mathfrak{s}_\eta'} + \abs{\mathfrak{s}_\eta''} < \Delta$ (these functions---or, equivalently, the graphs---depend continuously on $\eta$ in the $C^2$ topology). By $C^2$-continuous dependence of stable manifolds of $\Lambda_{V, \epsilon}$ on the parameter $V$, it follows that the same holds when $\Lambda_{0, \epsilon}$ and $\mathbb{S}_{0,\epsilon}$ above are replaced by $\Lambda_{V,\epsilon}$ and $\mathbb{S}_{V,\epsilon}$, for all $V$ sufficiently close to zero.

 Now take an arc $\upsilon$ of an unstable manifold of $\Lambda_{V, \epsilon}$ such that the orthogonal projection of $\Phi(\upsilon)$ into the $xy$ plane falls entirely inside $B_\delta^s$. Suppose there exists $m\in\N$ and $\eta$ such that $\mathcal{G}^m\circ\Phi(\upsilon)$ is tangent to the stable manifold corresponding to $\eta$, at a point which projects orthogonally into $B^u_\delta$. Then the projection of $\mathcal{G}^m\circ\Phi(\upsilon)$ into $B_\delta^u$ can be viewed as the graph of a function $\mathfrak{u}: x\rightarrow z$, near the point of tangency with $\mathfrak{s}_\eta$. Obviously if $\abs{\mathfrak{u}''}$ at the point of tangency is not smaller than $\Delta$, then the tangency must be quadratic. In what follows, we prove that, assuming $V$ is initially chosen sufficiently close to zero and $\delta > 0$ sufficiently small, we do in fact have $\abs{\mathfrak{u}''} > \Delta$.

 In what follows, we assume that
 \begin{align*}
 \Delta < \frac{1}{8}\hspace{2mm}\text{ and }\hspace{2mm}\abs{\mathfrak{s}'_\eta} + \abs{\mathfrak{s}''_\eta} < \Delta;
\end{align*}
this assumption is justified by Proposition \ref{prop:c2-continuity}.

 Write $\mathcal{G} = (\mathcal{G}_x, \mathcal{G}_y,\mathcal{G}_z)$. Set
 \begin{align}\label{eq:arb-17}
  M = \sup_{\substack{\alpha,\beta\in\set{x,y,z}\\ p\in \Phi(U)}}\norm{\nabla\partial_\beta \mathcal{G}_\alpha(p)}\hspace{2mm}\text{and}\hspace{2mm}
  \xi = \norm{\mathcal{G}|_U - D\mathcal{G}_\mathbf{0}}_{C^1}.
 \end{align}
 Take $C_0 = 1$, $\Delta_0 = \frac{1}{4}$ and take $U$ initially so small, that $\xi < \xi_0$, with $\xi_0 = \xi_0(C_0, \frac{1}{4})$ as in Lemma \ref{lem:annoying-2}. Now to finish the proof, it is enough to prove that there exists $n_0\in\N$, $n_0 < m$, such that $\gamma \eqdef \mathcal{G}^{n_0}\circ\Phi(\upsilon)$ satisfies the hypothesis of Lemma \ref{lem:annoying-2}, with $M$, $\Delta_0$ and $C_0$ as above, and, in addition, $\abs{\gamma_z''(t)} > \Delta$ for all $t$. We can assume, of course, that for every $n\in \set{1,\dots, m}$, $\mathcal{G}^n\circ\Phi(\upsilon)$ lies entirely in $\Phi(U)$ (by taking $\upsilon$ shorter as necessary).

 Take $\upsilon_0$ to be an arc of an unstable manifold of $\Lambda_{0, \epsilon}$ such that $\Phi(\upsilon_0)$ projects into $B_\delta^s$ and intersects $F^s$. Combination of Lemma \ref{lem:quadratic-surfaces} with Proposition \ref{prop:c2-continuity} guarantee that if $\delta > 0$ is taken sufficiently small to begin with, then there exists $n_0\in\N$ such that $\mathcal{G}^{n_0}\circ\Phi(\upsilon_0)$ satisfies the hypothesis of Lemma \ref{lem:annoying} with $M$ as in \eqref{eq:arb-17}, $G = \mathcal{G}$, $\xi_0$ as above, and, in addition,
 \begin{align*}
 &\abs{\frac{d}{dt}\mathcal{G}^{(n_0)}\circ\Phi(\upsilon_0)_y({t})} < 1,\\
 &\abs{\frac{d}{dt}\mathcal{G}^{(n_0)}\circ\Phi(\upsilon_0)_z({t})} < \Delta_0,\\
 \text{and }\hspace{2mm}&\abs{\frac{d^2}{dt^2}\mathcal{G}^{(n_0)}\circ\Phi(\upsilon_0)_z({t})} > \Delta
 \end{align*}
for all ${t}$. Hence there exists $n_0\in\N$ such that the same holds for \emph{every} arc $\upsilon_0$ of \emph{every} unstable manifold of $\Lambda_{0,\epsilon}$, such that the image of $\upsilon_0$ under $\Phi$ intersects $F^s$ and projects into $B_\delta^s$. By continuity, the same holds with $\Lambda_{V, \epsilon}$ in place of $\Lambda_{0,\epsilon}$ for all $V$ sufficiently close to zero.

 Define an \emph{unstable cone field}, $K^{u}$ in the neighborhood $\Phi(U)$ as follows. For $p\in \Phi(U)$, denote by $\mathbf{v} = (v_x, v_y, v_z)$ a vector in $\mathrm{T}_p\R^3$, the tangent space of $\R^3$ at $p$. Given $\kappa > 0$, the unstable cone at $p$ of size $\kappa$ is defined by
 \begin{align}\label{eq:cone-field}
  K_\kappa^{u}(p) \eqdef\set{\mathbf{v}\in \mathrm{T}_p\R^3: \norm{v_x} + \norm{v_y} \leq \kappa\norm{v_z}}.
 \end{align}
From \eqref{eq:differential} we see that $K_\kappa^u$ is invariant under $D\mathcal{G}_\mathbf{0}$ for any $\kappa$. Then, by continuity, assuming $U$ is taken sufficiently small to begin with, depending on $\kappa$, we see that the unstable cone field is invariant under the action by $D\mathcal{G}$ restricted to $U$.

 Take $\kappa_0 \in (0, 1)$ so that the unstable cone field of size $\kappa_0$ is transversal to the arcs of stable manifolds given above as the graphs of the functions $\set{\mathfrak{s}_\eta}_\eta$. Also, by Proposition \ref{prop:c2-continuity}, we may assume that the tangent space of the curve $\mathcal{G}^{n_0}\circ\Phi(\upsilon)$ from above lies inside the center-unstable cone field from \eqref{eq:cs-cone} of size $\kappa_0$. Now if for some $n\in\set{n_0,\dots, m-1}$, $\abs{\frac{d}{dt}\mathcal{G}^n_x\circ\Phi(\upsilon)(t)} = 0$, then $\frac{d}{dt}\mathcal{G}_x\circ\Phi(\upsilon)$ falls into an unstable cone of size $\kappa_0$, which, by the invariance of the unstable cone field, cannot happen if $t$ is such that $\mathcal{G}^m\circ\Phi(\upsilon)(t)$ is close to the point of tangency. Hence if we set $\gamma^{(0)} = \mathcal{G}^{n_0}\circ\Phi(\upsilon)$, then for each $n\in\set{n_0, \dots, m - 1}$, $\mathcal{G}^n\circ\Phi(\upsilon)$ can be reparameterized as $\gamma^{(n - n_0)}$ as in the statement of Lemma \ref{lem:annoying-2}. Indeed, since $\abs{\frac{d}{dt}\gamma_x^{(n - n_0)}(t)} = 1$, if $\abs{\frac{d}{dt}\gamma_y^{(n-n_0)}(t)} \geq 1$, then $\frac{d}{dt}\gamma^{(n - n_0)}(t)$ would fall out of the center-stable cone of size $\kappa_0$, since $\kappa_0 < 1$. Similarly, if $\abs{\frac{d}{dt}\gamma_z^{(n-n_0)}(t)} \geq \Delta_0$, then $\frac{d}{dt}\gamma^{(n-n_0)}(t)$ would lie in $K_{\Delta_0}^u(\gamma^{(n-n_0)}(t))$, which is transversal to the stable manifolds in question (so, by invariance of the cones $K_\kappa^u$, $\mathcal{G}^m\circ\Phi(\upsilon)$ couldn't be tangent to one of the stable manifolds in question). We can now apply Lemma \ref{lem:annoying-2}.

\subsubsection{Generic unfolding of tangencies}\label{sec:unfolding}

We now prove that the tangencies from the previous section unfold generically in the parameter $V$.

Let us consider the center-stable and the center-unstable manifolds that contain the arcs of the stable and the unstable manifolds that were considered in the previous section. We can parameterize these families by the same parameter $\eta$ -- the point of intersection with the interval $F^s$ or $F^u$, respectively. Denote them, respectively, by $W^{cs}_\eta$ and $W^{cu}_\eta$.

For every $N\in\N$, there exists $V_N < 0$ such that if we take $V$ in the previous section not smaller than $V_N$, then the arcs of the unstable manifolds have to be iterated at least $N$ times under $\mathcal{G}$ before a tangency with a stable manifold occurs in $\Phi(U)\setminus\Phi(\overline{U}_\epsilon)$ (as in the previous section). On the other hand, as $N$ tends to infinity, the images of $W^{cu}_\eta$ under $\mathcal{G}^N$ approach the $xz$ plane in the $C^2$ topology, uniformly in $\eta$ -- this is due to Proposition \ref{prop:c2-continuity}.

Thus, fix $V_0 < 0$ so close to zero that if $V\in (V_0, 0)$, then $N$ is sufficiently large so that the manifolds $\mathcal{G}^{N + k}(W^{cu}_\eta)$ intersect transversally with the manifolds  $W^{cs}_\eta$ inside $\Phi(U)$, for all $k\in \N$. Say for some $k_0\in\N$, there is an arc of an unstable manifold, say $\upsilon$, in $\mathcal{G}^{N+k_0}(W^{cu}_{\eta_1})$ that intersects an arc of a stable manifold in $W^{cs}_{\eta_2}$, say $\varsigma$, tangentially (tangency taking place in the ambient manifold $\Phi(\mathbb{S}_V)$ for some $V$). Denote by $\gamma$ the curve of transversal intersection of $W^{cs}_{\eta_2}$ and $\mathcal{G}^{N+k_0}(W^{cu}_{\eta_1})$ (which of course contains the point of tangency $\upsilon\cap\varsigma$). The arc $\upsilon$ is tangent to $\gamma$ in $\mathcal{G}^{N+k_0}(W^{cu}_{\eta_1})$; moreover, in order to show that the tangency $\upsilon\cap \varsigma$ unfolds in the parameter $V$, it is enough to show that the tangency of $\upsilon$ with $\gamma$ unfolds in $V$, with both $\upsilon$ and $\gamma$ viewed as submanifolds of $\mathcal{G}^{N+k_0}(W^{cu}_{\eta_1})$.

Now, near the point of tangency, project the arcs $\set{\mathcal{G}^{N+k_0}(W_{\eta_1}^{cu})\cap \Phi(\mathbb{S}_V)}_V$ orthogonally into the $xz$ plane; when projected orthogonally into the $xz$ plane, they can be viewed as graphs of functions $\mathfrak{u}_V: x\rightarrow z$ and, as we showed in Section \ref{sec:quadratic}, $\abs{\mathfrak{u}_V''} > \Delta$. Projection of $\gamma$ can be viewed as the graph of $\mathfrak{g}: x\rightarrow z$. Then if $\mathfrak{g}$ is sufficiently flat (that is, say, $\abs{\mathfrak{g}''} < \frac{\Delta}{2}$), then the tangency must unfold. Indeed, if it did not, then the tangencies would occur along the graph of $\mathfrak{g}$, and for different values of $V$, we claim that the graphs of $\mathfrak{u}_V$ would intersect, and this would imply intersections of unstable manifolds, which is impossible since unstable manifolds are one-dimensional injectively immersed submanifolds of $\mathbb{S}_V$, and no two unstable manifolds of two points with distinct orbits can intersect. More precisely, we have
\begin{claim}\label{claim:arb3}
 Let $J = [a,b]$ be a compact nondegenerate interval in $\R$. Suppose that $g: J\rightarrow\R$ is $C^2$. Suppose further that $\Delta > 0$ and $\abs{g''} < \Delta/2$ uniformly on $J$. Then if $u, v: J\rightarrow\R$ are $C^2$ functions such that $u'', v'' > \Delta$ uniformly on $J$ and the graphs of $u$ and $v$ are both tangent to the graph of $g$ at the points $(\alpha, u(\alpha))\in J\times \R$ and $(\beta, v(\beta))\in J\times \R$, respectively, then the graphs of $u$ and $v$ intersect over the interval $[\alpha, \beta]$ (assuming, without loss of generality, that $\alpha < \beta$).
\end{claim}

\begin{proof}
 Let us assume without loss of generality that $\alpha < \beta$. Define the functions $U \eqdef u - g$ and $V \eqdef v - g$. From the hypothesis, it follows that $U'', V'' > 0$, while $U'(\alpha) = V'(\beta) = 0$; hence $U$ is monotone increasing on $[\alpha, \beta]$, while $V$ is monotone decreasing on $[\alpha, \beta]$. On the other hand, $U(\alpha) = V(\beta) = 0$. Hence the graphs of $U$ and $V$ must intersect over the interval $[\alpha, \beta]$, and so the graphs of $u$ and $v$ must also intersect over the interval $[\alpha, \beta]$.
\end{proof}

\begin{rem}
 Notice the assumption in the claim above that $u'', v'' > \Delta$. In fact, we may assume $\abs{u''}, \abs{v''} > \Delta$ as long as $u''$ and $v''$ have the same sign.
\end{rem}

Now let us show how we can make sure that $\abs{\mathfrak{g}''} < \Delta/2$. Assuming that $N$ is sufficiently large (which can be ensured by taking $V$ sufficiently close to zero) and, for convenience, even, $\mathcal{G}^{-N/2}(W^{cs}_{\eta_2})$ and $\mathcal{G}^{N/2 + k_0}(W^{cu}_{\eta_1})$ are both close in the $C^2$ topology to the $xy$ and the $xz$ planes, respectively; hence orthogonal projection of their intersection onto the $xz$ plane can be considered as the graph of a function $\tilde{\mathfrak{g}}: x\rightarrow z$ with $\abs{\tilde{\mathfrak{g}}'} + \abs{\tilde{\mathfrak{g}}''} < \frac{\Delta}{2}$. Certainly $\mathcal{G}^{N/2 + k_0}(\upsilon)$ projects onto the graph of some $\tilde{\mathfrak{u}}: x\rightarrow z$ that exhibits a tangency with $\tilde{\mathfrak{g}}$. On the other hand, assuming that $N/2 > n_0$, with $n_0$ as in Section \ref{sec:quadratic}, we see that (as has been proved in the previous section) $\abs{\tilde{\mathfrak{u}}''} > \Delta$, and the same holds for the arcs of the unstable manifolds close to $\upsilon$. Hence the tangency of $\tilde{\mathfrak{u}}$ with $\tilde{\mathfrak{g}}$ unfolds, which of course implies that the tangency of $\upsilon$ with $\varsigma$ unfolds, as desired.

It now remains to prove that the unfolding occurs with nonzero speed in the parameter $V$. To this end, let us consider the following setup. As above, let $\upsilon$ be an unstable manifold lying in $W^{cu}_{\eta_1}$ for some $\eta_1$, and let $\varsigma$ be a stable manifold lying in some $W^{cs}_{\eta_2}$ for some $\eta_2$, such that for some $k\in\N$, $\mathcal{G}^{N+k_0}(\upsilon)$ intersects $\varsigma$ tangentially (on the surface $\Phi(\mathbb{S}_{V_1})$ for some $V_1 \in (V_0, 0)$ with $V_0$ as above). Take $\delta > 0$ small (in particular smaller than $(V_1 - V_0)/2$) and project all the stable manifolds $W^{cs}_{\eta_2}\cap \Phi(\mathbb{S}_V)$ with $V\in (V_1 - \delta, V_1 + \delta)$ orthogonally into the $xz$ plane. Each of these manifolds can be viewed, as we have done above, as the graph of a function $\mathfrak{s}_V: x \rightarrow z$, and these functions depend continuously on $V$ in the $C^2$ topology. Also project orthogonally all the unstable manifolds $\mathcal{G}^{N+k_0}(W^{cu}_{\eta_1})\cap\Phi(\mathbb{S}_V)$ into the $xz$ plane, for $V \in (V_1 - \delta, V_1 + \delta)$. These can be viewed, as above, as graphs of functions $\mathfrak{u}_V: x\rightarrow z$.

\begin{rem}
Notice that the family $\set{\mathfrak{s}_V}$ is different from the family $\set{\mathfrak{s}_\eta}$ that was considered in the previous section. In the previous section, we fixed $V$ and considered a family of stable manifolds on the surface $\Phi(\mathbb{S}_V)$, parameterized by the parameter $\eta$. Here we fix $\eta$ and consider the continuation of the stable manifold $W_\eta^{cs}\cap\Phi(\mathbb{S}_0)$ along the paramter $V$, which is precisely $W_\eta^{cs}\cap\Phi(\mathbb{S}_V)$. Similarily for the unstable manifolds.
\end{rem}

Since $\mathcal{G}^{N+k_0}(\upsilon)$ and $\varsigma$ intersect tangentially, the graphs of $\mathfrak{u}_{V_1}$ and $\mathfrak{s}_{V_1}$ also intersect tangentially in the $xz$ plane; moreover, since the tangency unfolds (as we saw above), for all $V > V_1$, there exist $a_V, b_V$ on the $x$ axis, $a_V < b_V$, such that the graphs of $\mathfrak{u}_V$ and $\mathfrak{s}_V$ intersect transversally at the points $\mathfrak{u}_V(a_V) = \mathfrak{s}_V(a_V)$ and $\mathfrak{u}_V(b_V) = \mathfrak{s}_V(b_V)$. Let $M(V)$ denote the maximum of $(\mathfrak{s}_V - \mathfrak{u}_V)$ on the interval $[a_V, b_V]$. Notice that the function $M: (V_1, V_1 + \delta)\rightarrow \R$ is smooth and can be smoothly extended to the point $V_1$, with $M(V_1) = 0$. We now need to prove that $M'(V_1) \neq 0$, which would imply that the unfolding of the tangency of $\mathcal{G}^{N+k_0}(\upsilon)$ with $\varsigma$ occurs with nonzero speed in the parameter $V$. We have
\begin{claim}\label{c:unfolding-speed}
 There exists $C > 0$ such that for all $V > V_1$ sufficiently close to $V_1$, $M(V) \geq C(V - V_1)$.
\end{claim}

\begin{proof}
 If $p$ denotes the point of tangential intersection of $\upsilon$ and $\varsigma$, let $q$ denote the point $\mathcal{G}^{-(N+k_0)}(p)$. Let $u_V$ denote the arc of the unstable manifold $\mathcal{G}^{N+k_0}(W^{cu}_{\eta_1})\cap \Phi(\mathbb{S}_V)$ that projects orthogonally onto the graph of $\mathfrak{u}_V$ over the interval $[a_V, b_V]$, and let $\tilde{u}_V$ denote $\mathcal{G}^{-(N+k_0)}(u_V)$. Take a smooth, compact curve passing through $q$ and the interior of each $\tilde{u}_V$, $V\in (V_1, V_1 + \delta)$, transversal to the surfaces $\Phi(\mathbb{S}_V)$, and such that its tangent direction at each point falls inside the cone $K_\kappa^u$ from \eqref{eq:cone-field}, with $\kappa = 1$, say; let us call this curve $\beta$. This cone field is invariant under the action by $D\mathcal{G}$ provided that we work inside a sufficiently small neighborhood of the origin (this restriction has already been taken into account in Section \ref{sec:quadratic}). The image of this curve under $\mathcal{G}^{N+k_0}$ gives a smooth, compact curve, say $\alpha$, with the following properties. The curve $\alpha$ passes through the point $p$ and through the interior of each of the arcs $u_V$, and the tangent space of $\alpha$ falls inside of the cone field $K_\kappa^u$. Let us now project $\alpha$ orthogonally into the $xz$ plane to obtain a curve, call it $\mathfrak{a}$, that passes through the point of tangency of $\mathfrak{s}_{V_1}$ with $\mathfrak{u}_{V_1}$ and a point in the interior of the graph of each of the functions $\mathfrak{u}_V$ over the interval $[a_V, b_V]$, $V\in (V_1, V_1 + \delta)$. Now recall from Section \ref{sec:quadratic} that $\abs{\mathfrak{s}'_{V_1}} < 1$. Using this together with the fact that the tangent space of $\mathfrak{a}$ also falls inside the cone field $K_{\kappa}^u$, we get the following.
 \it
 \begin{itemize}\itemsep0.5em
  \item If $l_V$ denotes the line segment parallel to the $z$ axis that passes through the graph of $\mathfrak{s}_{V_1}$ and the point $\mathfrak{a}\cap \mathbf{Graph}(\mathfrak{u}_V|_{[a_V, b_V]})$, and $\abs{l_V}$ denotes the length of $l_V$, then there exists $C_0 > 0$ independent of $V$ such that $\abs{l_V} \geq C_0\abs{\mathfrak{a}_V}$, where $\mathfrak{a}_V$ is the arc along $\mathfrak{a}$ connecting the points $\mathfrak{s}_{V_1}\cap \mathfrak{u}_{V_1}$ and $\mathfrak{a}\cap \mathbf{Graph}(\mathfrak{u}_V|_{[a_V, b_V]})$.
 \end{itemize}
 \rm
 \begin{figure}

 \end{figure}
 On the other hand, since $W^{cs}_{\eta_2}$ is close in the $C^2$ topology to the $xy$ plane, there exists a constant $C_1 > 0$ such that for any $V$, the maximum distance between the graphs of $\mathfrak{s}_V$ and $\mathfrak{s}_{V_1}$ is not larger than $C_1(V - V_1)$. Hence the distance from the point $\mathfrak{a}\cap \mathbf{Graph}(\mathfrak{u}_V|_{[a_V, b_V]})$ is not smaller than $C_0\abs{\mathfrak{a}_V} - C_1(V - V_1)$. Hence, to finish the proof it is enough to show that $\abs{\mathfrak{a}_V} > \frac{C_1}{C_0}(V - V_1)$ uniformly in $V$.

 Let $\beta_V$ denote the arc along $\beta$ with endpoints $q$ and $\tilde{u}_V\cap \beta$. Then there exists $C_2$, independent of $V$ (depending only on the size of the cone $K_\kappa^u$, but we have fixed $\kappa = 1$), such that $\abs{\beta_V} \geq C_2(V - V_1)$. By the expanding property of the vectors inside of the cone field $K_\kappa^u$, we know that there exists $C_3 > 0$, independent of $V$, such that for any $m\in\N$, if for $n = 0, \dots, m$, $\mathcal{G}^n(x) \in \Phi(U)$, then for any vector $v$ in $K_\kappa^u(x)$, $\norm{D\mathcal{G}^m(v)} \geq C_3\sqrt{\lambda^m}\norm{v}$, with $\lambda$ as in \eqref{eq:differential} (assuming that $U$ was initially chosen sufficiently small). Hence we know that $\abs{\mathfrak{a}_V} \geq C_2C_3\sqrt{\lambda^{N+k_0}}(V - V_1)$. Thus it remains only to make sure that $C_2C_3\sqrt{\lambda^{N+k_0}} > \frac{C_1}{C_0}$ independently of $V$. This can be ensured by considering $V_1$ sufficiently close to zero to begin with, since then $N$ will be suitably large.
\end{proof}

\subsubsection{Existence of tangencies}\label{sec:existence}

Now we prove that the tangencies that were postulated in Section \ref{sec:quadratic} indeed occur. That is, we prove that for all $V < 0$ sufficiently close to zero there exist points $q_1$ and $q_2$ in $\Lambda_{V,\epsilon}$ such that the unstable manifold of $q_1$ intersects the stable manifold of $q_2$ tangentially, and the tangency is quadratic and unfolds as in Section \ref{sec:unfolding}. From this statement, together with the fact that there exist periodic points whose global stable (respectively, global unstable) manifolds form a dense sublamination of the stable (respectively, unstable) lamination, it follows that there exists a periodic point $p$ in ${\Lambda}_{V,\epsilon}$, such that the stable and the unstable manifolds of $p$ exhibit a quadratic tangency that unfolds as in Section \ref{sec:unfolding}.

We continue working with $\Lambda_{V, \epsilon}$ and $\Lambda_\epsilon$ from the previous sections. Our arguments in principle follow the technique that was introduced by S. Newhouse in \cite{Newhouse1979}; that is, we compare two Cantor sets of large thickness (greater than one), where one Cantor set does not fall entirely into a gap of the other, and conclude that these Cantor sets must have nonempty intersection. The points of this intersection will be the sought tangencies. These Cantor sets are obtained as transversal intersection of a smooth curve with the stable and the unstable manifolds of $\Lambda_{V, \epsilon}$ on the surface $\mathbb{S}_{V}$. This smooth curve is a curve of tangencies of two foliations, one containing the stable, and the other the unstable, manifolds as a sublamination. Thus, the whole proof consists of the following steps.

\begin{enumerate}[(1)]\itemsep0.5em
\item Construct \emph{stable} and \emph{unstable} $C^1$ foliations with $C^2$ leaves and $C^1$ tangent field, such that stable and unstable manifolds, respectively, form sublaminations of those foliations.

\item Demonstrate that the leaves of the stable foliation contain quadratic tangencies with the leaves of the unstable foliation, and in fact these tangencies form a $C^1$ curve. Moreover, this curve is transversal to both foliations.

\item Show that the intersections of the curve of tangencies constructed in Step (2) with the stable and the unstable manifolds form Cantor sets of large (greater than one) thickness, and one of these Cantor sets does not fall entirely into a gap of the other.

\item Conclude that the two Cantor sets must have nonempty intersection. Therefore, there must exist stable manifolds that intersect tangentially with unstable manifolds (and, combined with the results of Sections \ref{sec:quadratic} and \ref{sec:unfolding}, these tangencies are quadratic and unfold generically in the parameter $V$).

\item Conclude that there exists a periodic point the stable and the unstable manifolds of which intersect in a quadratic tangency; moreover, these tangencies unfold as in Section \ref{sec:unfolding}. These will be the sought homoclinic tangencies.

\end{enumerate}

Before we begin with Step (1), let us note that the notation from the previous sections is carried over to the present section, unless explicitly redefined.

\begin{proof}[Step (1)]

Denote by $\hat{F}^s$ (respectively, $\hat{F}^u$) a fundamental domain for $\mathcal{G}^3$ along the interval $F^u$ (respectively, $F^s$) (recall that $F^s$ is an interval along the positive $y$ axis that lies in $\Phi(U)$ but not in $\Phi(\overline{U}_\epsilon)$, and $F^u$ is an interval along the positive $z$ axis, with the same properties; by taking $\epsilon$ smaller as necessary, we may of course assume that $F^{s,u}$ contains a fundamental domain for $\mathcal{G}^{3}$). As before, by $\eta$ we denote the parameter along $\hat{F}^{s}$ (respectively, $\hat{F}^u$) which indicates the point of intersection with an unstable (respectively, stable) manifold. Let $\eta_r^s$ and $\eta_l^s$ (respectively, $\eta_r^u$ and $\eta_l^u$) denote the two endpoints of $\hat{F}^s$ (respectively, $\hat{F}^u$). Following previously introduced notation, denote by $W^{cs}_\eta$ (respectively, $W^{cu}_\eta$) the center-stable (respectively, center-unstable) manifold that crosses $\hat{F}^u$ (respectively, $\hat{F}^s$) at the point $\eta$.

As has already been mentioned in the previous section, for all $V$ sufficiently close to zero, the orthogonal projections from $\Phi(\mathbb{S}_{V, \epsilon})$ onto $B_\delta^{s}$ and $B_\delta^{u}$ are diffeomorphisms. Let us denote these projections by $\mathbf{P}_V^{s}: B_\delta^s\rightarrow\Phi(\mathbb{S}_{V, \epsilon})$ and $\mathbf{P}_V^{u}: B_\delta^{u}\rightarrow\Phi(\mathbb{S}_{V, \epsilon})$. For a choice of $\delta^s, \delta^u \in (0, \delta)$, let $\hat{B}_{\delta^s, V}^s$ (respectively, $\hat{B}_{\delta^u, V}^u$) denote the subset of $\mathbf{P}_V^s(B_{\delta^s}^s)$ (respectively, $\mathbf{P}_V^u(B_{\delta^u}^u)$) that is bounded between the two curves $W^{cu}_{\eta_l^u}\cap\Phi(\mathbb{S}_{V, \epsilon})$ and $W^{cu}_{\eta_r^u}\cap\Phi(\mathbb{S}_{V, \epsilon})$ (respectively, $W^{cs}_{\eta_l^s}\cap\Phi(\mathbb{S}_{V, \epsilon})$ and $W^{cs}_{\eta_r^s}\cap\Phi(\mathbb{S}_{V, \epsilon})$). Of course, assuming that $\delta$ is initially chosen sufficiently small and $V$ is sufficiently close to zero, we have $\hat{B}_{\delta^s, V}^s \subset U$ and $\hat{B}_{\delta^u, V}^u\subset U$. Then $\hat{B}_{\delta^s, V}^s$ (respectively, $\hat{B}_{\delta^u, V}^u$) is laminated by arcs of one-dimensional unstable (respectively, stable) manifolds that when projected back into $B_\delta^s$ (respectively, $B_\delta^u$) via ${\mathbf{P}_V^s}^{-1}$ (respectively, ${\mathbf{P}_V^u}^{-1}$) intersect $F^u$ (respectively, $F^s$) transversally and stretch completely across $B_\delta^s$ (respectively, $B_\delta^u$). Since $T: \mathbb{S}_V\rightarrow\mathbb{S}_V$, and hence also $\mathcal{G}: \Phi(\mathbb{S}_V)\rightarrow \Phi(\mathbb{S}_V)$, is of smoothness class $C^{> 2}$ (in fact, $C^2$ suffices), this lamination can be extended (locally) to a $C^1$ foliation with $C^2$ leaves that when projected via ${\mathbf{P}_V^{s,u}}^{-1}$ into $B_\delta^{s,u}$ also intersect $\hat{F}^{s,u}$ transversally and stretch completely across $B_\delta^{s,u}$. Moreover, the tangent vector field of this foliation is $C^1$ (for more details, see \cite[Theorem 6.4b]{Hirsch1977}).

Of course these local foliations can be extended to complete $C^1$ foliations of $\hat{B}_{\delta^s, V}^s$ and $\hat{B}_{\delta^u, V}^u$ with $C^2$ leaves and $C^1$ tangent vector field that satisfies the same properties when projected back into $B_\delta^{s,u}$ (i.e. the leaves intersect $F^{s,u}$ transversally and stretch completely across $B_\delta^{s,u}$). Let us call this complete foliation of $\hat{B}_{\delta^s, V}^s$ (respectively, $\hat{B}_{\delta^u, V}^u$), $\mathcal{F}_V^u$ (respectively, $\mathcal{F}_V^s$). The leaves of $\mathcal{F}_V^u$ (respectively, $\mathcal{F}_V^s$) can be parameterized by the parameter $\eta_V$ which is the point of intersection of the projection (via ${\mathbf{P}_V^{s}}^{-1}$, or, respectively, ${\mathbf{P}_V^u}^{-1}$) of the leaf into $B_\delta^s$ (respectively, $B_\delta^u$) with $F^s$ (respectively, $F^u$).  The leaves of $\mathcal{F}_V^{s,u}$ depend continuously in the $C^2$ topology on the parameter $\eta_V$

Denote by $M^s$ the fundamental domain for $\mathcal{G}$ in $\hat{F}^s$ such that $\hat{F}^s = \mathcal{G}^{-1}(M^s)\cup M^s\cup\mathcal{G}(M^s)$. If $\eta_1$ and $\eta_2$ are the two endpoints of $M$, denote by $\hat{M}_{\delta^s, V}^s$ the subset of $\hat{B}_{\delta^s, V}^s$ that is bounded by $W_{\eta_i}^{cu}$, $i = 1, 2$. We define $M^u$ and $\hat{M}_{\delta^u, V}^u$ similarly, with $F^u$, $\hat{F}^u$, $\mathcal{G}^{-1}$, $\hat{B}_{\delta^u, V}^u$ and $W^{cs}$ in place of $F^s$, $\hat{F}^s$ and $\mathcal{G}$, $\hat{B}_{\delta^s, V}^s$ and $W^{cu}$, respectively.
\end{proof}

\begin{proof}[Step (2)]
 We now prove that the leaves of the image of the foliation $\mathcal{F}^u_V$ under $\mathcal{G}^k$ for sufficiently large $k$ (depending on $V$) exhibit quadratic tangencies with the leaves of $\mathcal{F}_V^s$.
 \begin{prop}\label{prop:tangencies}
  For all $\delta^u \in (0, \delta)$ there exists $\delta^s\in (0, \delta)$ and $V_0 < 0$ such that for all $V\in (V_0, 0)$ the following holds. For every leaf $\nu$ of $\mathcal{F}_V^u$ in $\hat{B}_{\delta^s, V}^s$ there exists $k(V, \nu)\in\N$, such that $\mathcal{G}^k(\nu)$ exhibits a quadratic tangency with a leaf of $\mathcal{F}_V^s$ in $\hat{B}_{\delta^u, V}^u$.
 \end{prop}

 \begin{proof}
 With $\epsilon$ denoting the radius of the neighborhood $U_\epsilon$ as above, let $V_0 < 0$ so close to zero that for all $V\in(V_0, 0)$, the center-stable and the center-unstable manifolds of $\Lambda_\epsilon$ intersect the surfaces $\mathbb{S}_V$ transversally in the neighborhood $U$ as in Proposition \ref{prop:c2-continuity}. Denote $\Phi(U\cap \bigcup_{V \in (V_0, 0)}\mathbb{S}_V)$ by $\mathbb{J}$. Let us define two partitions of $\mathbb{J}$, $\mathcal{P}^s$ and $\mathcal{P}^u$, consisting of stable and unstable fundamental domains respectively, as follows.

  % By Proposition \ref{prop:c2-continuity}, we may assume that the center-stable and the center-unstable manifolds that intersect $F^u$ and $F^s$, respectively, inside $\Phi(U)$ can be viewed as graphs of smooth functions defined over the $xy$ and the $xz$ planes, respectively. Furthermore, we may assume that the domain of definition of these functions contains in its interior the orthogonal projection of $\mathbb{J}$ into the $xy$ and the $xz$ planes. Now

  Fix a center-stable (respectively, a center-unstable) manifold, and call it, say, $W^{cs}$ (respectively, $W^{cu}$) and define, for $n\in\Z$, the sets $J_n^s$ (respectively, $J_n^u$) as the set of those points in $\mathbb{J}$ that lie between (or on) the two manifolds $\mathcal{G}^{n}(W^{cs})\cap \mathbb{J}$ and $\mathcal{G}^{(n - 1)}(W^{cs})\cap \mathbb{J}$ (respectively, $\mathcal{G}^{n}(W^{cu})\cap \mathbb{J}$ and $\mathcal{G}^{(n - 1)}(W^{cu})\cap \mathbb{J}$). Now let $\mathcal{P}^s = \set{J^s_{n}: n\in\Z}$ and $\mathcal{P}^u = \set{J^u_{n}: n\in\Z}$. Of
course, $\mathcal{P}^s$ and $\mathcal{P}^u$ are not partitions in the strict sense, since for any $n\in\Z$, the sets $J^s_n$ and $J^u_n$ intersect the sets $J^s_{n\pm 1}$ and $J^u_{n\pm 1}$, respectively, along the boundaries (formed by two center-stable and, respectively, center-unstable manifolds). Also notice that there exist $N^s, N^u \in\N$ such that for all $n^s > N^s$ and $n^u > N^u$, $J^s_{n^s} = J^u_{-n^u} = \emptyset$ (we assume that $N^s$ and $N^u$ are the smallest such numbers). Essentially, we've sliced $\mathbb{J}$ into stable and unstable fundamental domains.

  \begin{lem}\label{lem:tech-1}
   For any choice of $\delta^u\in (0, \delta)$, there exists $\delta^s\in (0, \delta)$ and $V_0 < 0$ not smaller than the $V_0$ above such that for all $V\in (V_0, 0)$ and every arc $\nu$ in $\mathcal{F}_V^u$ that lies in $\hat{B}_{\delta^s, V}^s\cap J_n^s$, for some $n\in\Z$, there exists a $k = k(V, \nu)\in \N$ such that $\mathcal{G}^k(\nu)$ lies in $\hat{B}_{\delta^u, V}^u$.
  \end{lem}

  \begin{proof}
   Observe that there exists $m\in\Z$ such that $\hat{M}_{\delta^u, V}^u$ lies entirely in $J_m^s$, by construction, and $\hat{B}_{\delta^u, V}^u = \mathcal{G}^{-1}(\hat{M}_{\delta^u, V}^u)\cup \hat{M}_{\delta^u, V}^u\cup\mathcal{G}(\hat{M}_{\delta^u, V}^u)$. If we take $\delta^s$ sufficiently small, we know that there exists $V_0 < 0$ such that for all $V\in (V_0, 0)$, every point of $\Phi(\mathbb{S}_V)$ that lies in $\hat{B}_{\delta^s, V}^s$ has its $z$ coordinate strictly smaller than that of every point lying in $\hat{B}_{\delta^u, V}^u$. It follows that there exists $k\in\N$ such that $J_{n + k}^s = J_m^s$; in particular, $\mathcal{G}^k(\nu)$ lies in $J_m^s$. So in order to prove that $\mathcal{G}^k(\nu)$ lies in $\hat{B}_{\delta^u, V}^u$, it remains to prove that the orthogonal projection of $\mathcal{G}^k(\nu)$ into the $xz$ plane lies in $B_{\delta^u}^u$. We prove this next (in the proof we shall require that $k$ is large enough, which will be guaranteed for all $V$ sufficiently close to zero).

   Recall that the surface $\Phi(\mathbb{S}_0)$ exhibits a quadratic tangency with the $xz$ plane along the $z$ axis; recall also that the surfaces $\Phi(\mathbb{S}_V)$ depend continuously on $V$ in the $C^2$ topology away from a neighborhood of the singularities. It follows that for any given $\delta^u$, there exists $\epsilon_0 > 0$ and $V_0 < 0$ such that for any $V\in (V_0, 0)$, any point that lies on $\Phi(\mathbb{S}_V)$ and whose distance from the $xz$ plane is not larger than $\epsilon_0$ orthogonally projects into the strip $(-\delta^u, \delta^u)\times z$ in the $xz$ plane. It therefore suffices to prove that every point of $\mathcal{G}^k(\nu)$ is not further than $\epsilon_0$ away from the $xz$ plane.

   Observe that there exists $d > 0$ such that for all $V \in (V_0, 0)$ and all points $p$ in $\hat{B}_{\delta^s, V}^s$, the distance from $p$ to the $xz$ plane is not larger than $d$. On the other hand, there exists $N\in\N$ such that any point $p\in \Phi(U)$ whose distance from the $xz$ plane is not larger than $d$ and such that for every $n\in \set{1,\dots, N}$, $\mathcal{G}^n(p)\in \Phi(U)$, the distance from $\mathcal{G}^N(p)$ to the $xz$ plane is smaller than $\epsilon_0$. Thus we need to make sure that $k$ is larger than $N$; this can be done by choosing $\delta^s$ sufficiently small as well as $V_0$ sufficiently close to zero (that is, we need to make sure that the distance from any point $p$ in $\hat{B}_{\delta^s, V}^s$ to the $xy$ plane is sufficiently small).
  \end{proof}

  Let $V_0 < 0$ be not smaller than the $V_0$ from Lemma \ref{lem:tech-1} and, in addition, the following holds. The tangent space of every leaf from $\mathcal{F}_V^s$ when projected into the strip $(-\delta, \delta)\times z$ in the $xz$ plane makes the angle with the horizontal not larger than $\pi/8$ radians (this can be ensured by an application of Proposition \ref{prop:c2-continuity}). Let $d$ denote the minimal distance between the two center-stable manifolds (restricted to the neighborhood $\Phi(U)$) that mark the boundary of $J_m^s$, with $\hat{M}_{\delta^u, V}^u$ contained in $J_m^s$. Let us assume that $\delta$ was initially chosen so small, that for every $\delta^u\in (0, \delta)$, the following holds.

  \begin{enumerate}[(C1)]\itemsep0.5em
   \item \emph{If $\gamma$ is any smooth curve in the strip $(-\delta^u, \delta^u)\times z$ that connects two points at least at distance $d$ from each other, then at some point $p\in \gamma$, the tangent space $T_p\gamma$ makes the angle of at least $\pi/4$ radians with the vertical.}
  \end{enumerate}

  Now with the notation as in Lemma \ref{lem:tech-1}, let $\nu$ be a leaf of $\mathcal{F}_V^u$ that lies entirely in $J_n^s$. Observe that since the center-unstable manifolds can be assumed to be arbitrarily close in the $C^2$ topology to the $xz$ plane (again applying Proposition \ref{prop:c2-continuity}), and since $\Phi(\mathbb{S}_0)$ exhibits a quadratic tangency with the $xy$ plane and away from the singularities the surfaces $\Phi(\mathbb{S}_V)$ depend continuously in the $C^2$ topology on $V$, there exists a smooth curve $\tilde{\nu}$, namely a connected subset of the stable manifold that contains the arc $\nu$, with the following properties.

  \begin{enumerate}[(C1)]\itemsep0.5em
  \setcounter{enumi}{1}
   \item \it $\tilde{\nu}$ is connected, does not intersect itself, and contains $\nu$; \rm

   \item $\tilde{\nu}$ is contained in $J_n^s\cup J_{n+1}^s$

   \item \it $\tilde{\nu}\cap J^s_{n+1}$ consists of two connected smooth disjoint branches, say $\tilde{\nu}_1$ and $\tilde{\nu}_2$, with $\tilde{\nu}_1\cap \tilde{\nu}_2 = \emptyset$, and $\tilde{\nu}_i$, $i = 1, 2$, connect the two center-stable manifolds that form the boundary of $J^s_{n+1}$. \rm

   \item \it There exists a cone field $K^u(p)$, $p\in \Phi(U)$, with $K^u(p)$ transversal to the horizontal plane, invariant under $D\mathcal{G}$, and the curves (rather, their tangent spaces) $\tilde{\nu}_i$, $i = 1, 2$, fall into this cone field. \rm

   \item \it The branches $\tilde{\nu}_i$, $i = 1, 2$, are oppositely oriented. More precisely, if we parameterize $\tilde{\nu}$ on some compact interval, then the unit tangent vectors of $\tilde{\nu}_1$, by transversality with the horizontal, must have either strictly positive or strictly negative $z$ component; then those of $\tilde{\nu}_2$ have, respectively, strictly negative or strictly positive $z$ component. \rm
  \end{enumerate}

  \begin{rem}
   We should remark that (C5) is not completely obvious. Albeit above we have already constructed an invariant cone field, which we called $K_\kappa^u$, transverse to the horizontal, the cone field $K^u(p)$ is different in that it does not have a fixed opening angle $\kappa$, but its angle depends on $p$ and approaches $\pi$ radians as $p$ approaches the $xy$ plane. This is needed since we want (C5) to hold independently of any choice of $\delta^s$ and $V$. Nevertheless, such a cone field can be constructed -- see Proposition 3.12 in \cite{Damanik2010a} for the details.
  \end{rem}

  \noindent Now as a direct consequence of the above, we have the following.

  As in the proof of Lemma \ref{lem:tech-1}, apply $\mathcal{G}$ $k$ times such that $J^s_{n+k - 1} =J^s_m$ (hence $J^s_{n+k} = J^s_{m-1}$, and $\mathcal{G}^k(\nu)$ still falls in $\hat{B}_{\delta^u, V}^s$, even though it is one iteration less than what is stated in Lemma \ref{lem:tech-1}: we can  adjust $k$ by taking $V$ closer to zero as necessary). Now from the invariance of the cones from (C5) we get that the branches $\mathcal{G}^k(\tilde{\nu}_i)$, $i = 1, 2$, are transversal to the horizontal and are oppositely oriented as in (C6). Since these branches connect the center-stable manifolds that form the boundary of $J_m^s$, we know from (C1) that for $i = 1, 2$, $\mathcal{G}^k(\tilde{\nu}_i)$ contains a point $p_i$ such that that $T_{p_i}\tilde{\nu}_i$ makes an angle with the horizontal of at least $\pi/4$ radians. Since the two branches are oppositely oriented, and since $\mathcal{G}^k(\tilde{\nu})$ is connected, there exists a point in $\tilde{\nu}$ which is tangent to a leaf of the stable foliation $\mathcal{F}_V^s$, since the leaves of $\mathcal{F}_V^s$ when projected into $(-\delta, \delta)\times z$ were assumed to make the angle with the horizontal of not larger than $\pi/8$ radians. This completes the proof of Proposition \ref{prop:tangencies}.
 \end{proof}
We can now carry out this construction for every leaf of $\mathcal{F}_V^u$, concluding that every leaf of $\mathcal{G}^k(\mathcal{F}_V^u\cap \hat{M}_{\delta^s, V}^s)$ intersects tangentially with a leaf of $\mathcal{F}_V^s$. In Section \ref{sec:quadratic} we proved that this intersection is quadratic. We now need to prove that these tangencies form a $C^1$ curve that is transversal to the foliations $\mathcal{G}^k(\mathcal{F}_V^u\cap \hat{M}_{\delta^s, V}^s)$ and $\mathcal{F}_V^s$. This is done in \cite[Lemma 9]{Newhouse1979}. Let us record this result for later reference as
\begin{prop}\label{prop:curve-of-tangs}
For every $\delta^u\in (0, \delta)$ there exists $\delta^s\in(0,\delta)$ and $V_0 < 0$ such that for all $V\in (V_0, 0)$ the following holds. There exists $k\in\N$ such that for every leaf $\nu$ of $\mathcal{F}_V^u$ in $\hat{M}_{\delta^s, V}^s$, $\mathcal{G}^k(\nu)$ exhibits a quadratic tangency with a leaf of $\mathcal{F}_V^s$ in $\hat{B}_{\delta^u, V}^u$. These tangencies form a $C^1$ curve transversal to both, the stable and the unstable foliations $\mathcal{F}_V^{s}$ and $\mathcal{G}^k(\mathcal{F}^u_V\cap \hat M_{\delta^s, V}^s)$.
\end{prop}
\end{proof}

\begin{proof}[Step (3)]

 We begin by proving a bounded distortion property for $\mathcal{G}|_{\Phi(U)}$, given in the following lemma.
 \begin{lem}\label{lem:bdp}
 For each $p$ in the neighborhood $\Phi(U)$, let $K_1^u(p)$ be the unstable cone at $p$ as defined in \eqref{eq:cone-field}, with $\kappa = 1$. Define the stable cone at $p$, $K_1^s(p)$, similarly, with $v_y$ and $v_z$ interchanged in \eqref{eq:cone-field}. Then if $U$ is taken sufficiently small, we have the following.

 \begin{enumerate}\itemsep0.5em
  \item If $\mathcal{G}(p)\in U$ (respectively, $\mathcal{G}^{-1}(p)\in U$), then $D\mathcal{G}(K_1^u(p))$ lies in the interior of $K_1^u(\mathcal{G}(p))$ (respectively, $D\mathcal{G}^{-1}(K_1^s(p))$ lies in the interior of $K_1^s(\mathcal{G}^{-1}(p))$; moreover, there exists a constant $\tilde{\lambda}>1$ such that the vectors in $K_1^u(p)$ (respectively, $K_1^s(p)$) expand under the action by $D\mathcal{G}$ (respectively, $D\mathcal{G}^{-1})$ by a constant not smaller than $\tilde{\lambda}$.

  \item There exists $C > 0$ such that the following holds. If $k\in\N$ and for every $n\in\set{1,\dots,k}$, $\mathcal{G}^{n}(p)\in U$ (respectively, $\mathcal{G}^{-n}(p)\in U$), then for any $v\in K_1^u(p)$ (respectively, $v\in K_1^s(p)$) with $\norm{v}=1$,
  \begin{align}\label{eq:bdp}
   \left(\frac{\norm{D\mathcal{G}_p^k(v)}}{\norm{D\mathcal{G}_\mathbf{0}^{k}(v)}}\right)^{\pm 1}\leq C\hspace{2mm}\text{ (respectively, }\hspace{1mm} \left(\frac{\norm{D\mathcal{G}_p^{-k}(v)}}{\norm{D\mathcal{G}_\mathbf{0}^{-k}(v)}}\right)^{\pm 1}\leq C\text{)},
  \end{align}
  where $D\mathcal{G}_{\mathbf{0}}$ was defined in \eqref{eq:differential}.

  \item As a corollary of (2), there exists a constant $D > 0$, independent of $k$ from (2), such that the following holds. If $v\in K_1^u(p)$ and $u\in K_1^s(p)$ with $\norm{u}=\norm{v}=1$, and for all $n\in\set{1, \dots, k}$, $\mathcal{G}^{\pm n}(p)\in U$, then
  \begin{align}\label{eq:bdp-cor}
   \left(\frac{\norm{D\mathcal{G}_p^k(v)}}{\norm{D\mathcal{G}_p^{-k}(u)}}\right)^{\pm 1} \leq D.
  \end{align}
 \end{enumerate}
\end{lem}

\begin{proof}
 Observe that $D\mathcal{G}$ can be forced arbitrarily close to $D\mathcal{G}_\mathbf{0}$ in the $C^1$ topology by taking $U$ sufficiently small, so (1) follows by continuity (in fact, this shows that $\tilde{\lambda}$ can be taken arbitrarily close to $\lambda$).

 Let us prove (2). We shall prove the claim in \eqref{eq:bdp} for $D\mathcal{G}$; the claim for $D\mathcal{G}^{-1}$ can be proved similarly.

 Observe that
 \begin{align*}
  \left(\frac{\norm{D\mathcal{G}_p^k(v)}}{\norm{D\mathcal{G}_\mathbf{0}(v)}}\right)^{\pm 1} = \left(\frac{\norm{D\mathcal{G}_{p_0}(v_0)}\norm{D\mathcal{G}_{p_1}(v_1)}\cdots\norm{D\mathcal{G}_{p_{k-1}}(v_{k-1})}}{\norm{D\mathcal{G}_{\mathbf{0}}(u_0)}\norm{D\mathcal{G}_{\mathbf{0}}(u_1)}\cdots\norm{D\mathcal{G}_{\mathbf{0}}(u_{k-1})}}\right)^{\pm 1},
 \end{align*}
 where
 \begin{align*}
  v_0 = u_0 = v,\hspace{2mm}\text{and for}\hspace{2mm} n \geq 1,\hspace{2mm} v_n = \frac{D\mathcal{G}_{p_{n-1}}(v_{n - 1})}{\norm{D\mathcal{G}_{p_{n-1}}(v_{n - 1})}}\hspace{2mm}\text{ and }u_n = \frac{D\mathcal{G}_{\mathbf{0}}(u_{n-1})}{\norm{D\mathcal{G}_{\mathbf{0}}(u_{n-1})}},
 \end{align*}
 and
 \begin{align*}
  p_0 = p,\hspace{2mm}\text{and for}\hspace{2mm} n\geq 1,\hspace{2mm} p_n = \mathcal{G}(p_{n-1}).
 \end{align*}
 Now we have
 \begin{align*}
  \abs{\log\left(\frac{\norm{D\mathcal{G}_p^k(v)}}{\norm{D\mathcal{G}_\mathbf{0}(v)}}\right)^{\pm 1}} & \leq \sum_{n = 0}^{k-1}\abs{\log \norm{D\mathcal{G}_{p_n}(v_n)} - \log\norm{D\mathcal{G}_\mathbf{0}(u_n)}}\\
  &\leq L\sum_{n = 0}^{k-1}\abs{\norm{D\mathcal{G}_{p_n}(v_n)} - \norm{D\mathcal{G}_\mathbf{0}(u_n)}},
 \end{align*}
 where $L > 0$ is a Lipschitz constant for the function $\log|_{[1, \infty)}$. On the other hand,
 \begin{align*}
  & L\sum_{n = 0}^{k-1}\abs{\norm{D\mathcal{G}_{p_n}(v_n)} - \norm{D\mathcal{G}_\mathbf{0}(u_n)}} \\
  &\leq L\sum_{n = 0}^{k-1}\norm{D\mathcal{G}_{p_n}(v_n) - D\mathcal{G}_\mathbf{0}(u_n)}\\
  &= L\sum_{n = 0}^{k-1}\norm{D\mathcal{G}_{p_n}(v_n) - D\mathcal{G}_\mathbf{0}(u_n) + D\mathcal{G}_{p_n}(u_n) - D\mathcal{G}_{p_n}(u_n)}\\
  &\leq \left(\sup_{0\leq n \leq k - 1}\norm{D\mathcal{G}_{p_n}}\right)L\sum_{n = 0}^{k-1}\norm{v_n - u_n} + L\sum_{n = 0}^{k-1}\norm{D\mathcal{G}_{p_n} - D\mathcal{G}_{\mathbf{0}}}.
 \end{align*}
 Since the supremum above is uniformly bounded (in $k$) because $p_n\in U$ for all $n$ and $\mathcal{G}$ is $C^2$, it is enough to bound uniformly in $k$ the two sums
 \begin{align}\label{eq:lem-step-1-final}
  \sum_{n = 0}^{k-1}\norm{v_n - u_n}\hspace{2mm}\text{ and }\hspace{2mm}\sum_{n = 0}^{k-1}\norm{D\mathcal{G}_{p_n} - D\mathcal{G}_\mathbf{0}}.
 \end{align}

 To keep our arguments better organized and easier to read, before we continue, let us state and prove as a separate claim the following simple result.

 \begin{lem}\label{lem:orbit}
  There exists $\mu\in (0, 1)$, $V_0 < 0$, and $C_0 > 1$, such that for all $k\in\N$ and $V\in(V_0, 0]$, the following holds. Suppose that $p_0\in \Phi(U\cap\mathbb{S}_V)$ and for all $n\in \set{1,\dots, k}$, $\mathcal{G}^n(p_0)\in \Phi(U)$. Let $p_n = \mathcal{G}(p_{n-1})$. Let $\tilde{d}_n$ denote the distance from $p_n$ to the origin. Then there exists $n_0\in\set{1,\dots, k}$ such that for all $n_*\in \set{1,\dots, n_0}$ and $n^*\in \set{n_0 + 1, \dots, k}$ we have
  \begin{align}\label{eq:c-orbit-eq}
   \tilde{d}_{n_*} \leq C_0\mu^{n_*}\tilde{d}_0\hspace{2mm}\text{ and }\hspace{2mm} \tilde{d}_{n^*}\leq C_0\mu^{k - n^*}\tilde{d}_k.
  \end{align}
 \end{lem}

 \begin{proof}
  Denote by $d_{n, s} = d_{n, s}(p)$ and $d_{n, u} = d_{n, u}(p)$ the distance from $p_n$ to the $xy$- and the $xz$-plane, respectively. Let $d_n = \sqrt{d_{n, s}^2 + d_{n, u}^2}$ denote the distance from $p_n$ to the $x$ axis. Let $n_0$ be the maximum in $\set{1, \dots, k}$ such that for all $n_* \in \set{1, \dots, n_0}$, $d_{n_*, u}\geq d_{n_*, s}$. Then for all $n_*\in \set{1, \dots, n_0}$ and $n^*\in\set{n_0 + 1, \dots, k}$ we have the following.
  \begin{align*}
   d_{n_*} \leq \sqrt{2}d_{n_*, u}\hspace{2mm}\text{ and }\hspace{2mm} d_{n^*}\leq \sqrt{2}d_{n^*, s}.
  \end{align*}
  On the other hand, there exists $\mu\in (0, 1)$ and a constant $C_1 > 1$ such that for any $k\in\N$ and $p\in\ \Phi(U)$, if for all $n\in\set{1,\dots,k}$, $\mathcal{G}^n(p)\in \Phi(U)$, then
  \begin{align*}
    d_{n, u}\leq C_1\mu^n d_{0, u}\hspace{2mm}\text{ and }\hspace{2mm}d_{n, s}\leq C_1\mu^{k-n}d_{k, s},
  \end{align*}
  which follows easily from $C^1$ closeness of $D\mathcal{G}$ to $D\mathcal{G}_{\mathbf{0}}$. Thus, for all $n_*$ in $\set{1, \dots, n_0}$ and $n^*\in\set{n_0+1,\dots, k}$, we have
  \begin{align*}
  d_{n_*}\leq \sqrt{2}C_1\mu^{n_*}d_{0,u}\leq \sqrt{2}C_1\mu^{n_*}d_0\hspace{2mm}\text{ and }\hspace{2mm}d_{n^*}\leq \sqrt{2}C_1\mu^{k-n_*}d_{k,s}\leq \sqrt{2}C_1\mu^{k-n_*}d_k.
  \end{align*}
  Since $\Phi(U\cap\mathbb{S}_0)$ has a conic singularity at the origin, there exists $C_2> 0$ such that for all $V\in(V_0, 0]$ with $V_0 < 0$ sufficiently close to zero, and for all $p\in \Phi(U\cap \mathbb{S}_V)$, we have
  \begin{align*}
  \dist(p, x)\leq \dist(p, \mathbf{0})\leq C_2\dist(p,x),
  \end{align*}
  where $\dist(p,x)$ and $\dist(p,\mathbf{0})$ are the distances from $p$ to the $x$ axis and to the origin, respectively.
 \end{proof}

 Let us now continue with the proof of (2) of Lemma \ref{lem:bdp}. Since $\mathcal{G}$ is $C^2$, in the notation of Lemma \ref{lem:orbit} we have
 \begin{align*}
  \sum_{n = 0}^{k-1}\norm{D\mathcal{G}_{p_n} - D\mathcal{G}_\mathbf{0}}\leq C_3\sum_{n = 0}^{k-1} \tilde{d}_n,
 \end{align*}
 with $C_3 > 0$ some constant independent of $n$ and $k$, and the right side is bounded, by Lemma \ref{lem:orbit}, independently of $k$. It remains to bound the first of the two sums from \eqref{eq:lem-step-1-final}.

 Since $v_n$ and $u_n$ are unit vectors, it would suffice to bound
 \begin{align}\label{eq:arb0}
  \sum_{n = 0}^{k-1}\measuredangle(v_n, u_n)\leq \sum_{n = 0}^{k-1}\measuredangle(v_n, \mathbf{z}) + \sum_{n = 0}^{k-1}\measuredangle(\mathbf{z},u_n),
 \end{align}
 where $\mathbf{z}$ denotes the direction parallel to the $z$ axis. It is easy to see just from the definition of $D\mathcal{G}_\mathbf{0}$ that there exists $\tilde{\mu}\in (0, 1)$ such that
 \begin{align}\label{eq:arb1}
  \text{for any}\hspace{2mm}v\in K_1^u\hspace{2mm}\text{and}\hspace{2mm}n\in\N,\hspace{2mm}\measuredangle(D\mathcal{G}^n_\mathbf{0}(v), \mathbf{z})\leq \tilde{\mu}^n\measuredangle(v, \mathbf{z}).
 \end{align}
 Hence the second sum on the right in \eqref{eq:arb0} is bounded uniformly in $k$. Let us now demonstrate that the first term in the sum on the right in \eqref{eq:arb0} is also bounded uniformly in $k$.

 We have
 \begin{align}\label{eq:arb4}
 \sum_{n = 0}^{k-1} \measuredangle(v_n, \mathbf{z}) = \sum_{n_* = 0}^{n_0}\measuredangle(v_{n_*}, \mathbf{z}) + \sum_{n^* = n_0 + 1}^{k-1}\measuredangle(v_{n^*}, \mathbf{z}),
 \end{align}
 with $n_0$ as in Lemma \ref{lem:orbit}. From the definition of $D\mathcal{G}_\mathbf{0}$, we see that $\mathbf{x}$, $\mathbf{y}$ and $\mathbf{z}$, where $\mathbf{x}$, $\mathbf{y}$, and $\mathbf{z}$ are the directions parallel to the $x$, the $y$, and the $z$ axes, respectively, are eigendirections of $D\mathcal{G}_\mathbf{0}$, and vectors in the $z$ direction expand by $\lambda > 1$, while vectors in the $x$ direction neither expand nor contract, and vectors in the $y$ direction contract by $1/\lambda < 1$. By continuity, for any $\delta \in(0, \frac{1}{10\lambda})$, $U$ can be taken sufficiently small such that for all $p\in\Phi(U)$ there exist constants $\mu_p^u > 1  > \mu_p^s > 0$ and $\mu_p^c > 0$, and a constant $C_4 > 0$ independent of $p$, such that the following statements hold.
 \it
 \begin{itemize}\itemsep0.5em
  \item $\mu_p^u \in (\lambda - \delta, \lambda + \delta)$, $\mu_p^s \in (1/\lambda - \delta, 1/\lambda + \delta)$, and $\mu_p^c \in (1 - \delta, 1 + \delta)$ are eigenvalues of $D\mathcal{G}_p$.

  \item Denote the eigenspaces of $D\mathcal{G}_p$ corresponding to the eigenvalues $\mu_p^*$ by $E_p^*$, $*\in\set{s,u,c}$. Then the tangent space of $\R^3$ at the point $p$ is
  \begin{align*}
  T_p\R^3 = E_p^u\oplus E_p^c\oplus E_p^s,
  \end{align*}
  and
  \begin{align*}
  \measuredangle(E_p^s, \mathbf{y}),\hspace{2mm} \measuredangle(E_p^u, \mathbf{z}),\hspace{2mm} \measuredangle(E_p^c, \mathbf{x}) \leq \delta.
  \end{align*}
  Consequently, there exists $\hat{\mu} \in (0, 1)$ such that for every $p\in U$ and $v\in K_1^u(p)$, we have
  \begin{align*}
  \measuredangle(D\mathcal{G}_p(v), E_p^u)\leq \hat{\mu}\measuredangle(v, E_p^u).
  \end{align*}

  \item For any $p, q\in \Phi(U)$, we have
  \begin{align*}
  \abs{\mu_p^* - \mu_q^*} \leq C_4\abs{p - q}
  \end{align*}
   and
  \begin{align*}
  \measuredangle(E_p^u, \mathbf{z}) &\leq C_4\dist(p, \mathbf{0}),\\
  \measuredangle(E_p^s, \mathbf{y}) &\leq C_4\dist(p, \mathbf{0}),\\
  \measuredangle(E_p^c, \mathbf{x}) &\leq C_4\dist(p, \mathbf{0}).
  \end{align*}
  \end{itemize}
 \rm

 Now for $n_*\in \set{0,\dots,n_0}$, we can estimate $\measuredangle(v_{n_*}, \mathbf{z})$ as follows. Observe that
 \begin{align}\label{eq:arb2}
 \measuredangle (v_{n_*}, \mathbf{z}) \leq \measuredangle (v_{n_*}, E_{p_{n_*}}^u) + \measuredangle(E_{p_{n_*}}^u, \mathbf{z}) \leq \measuredangle(v_{n_*}, E_{p_{n_*}}^u) + C_4\tilde{d}_{n_*}.
 \end{align}
 On the other hand, we have
 \begin{claim}\label{claim:arb2}
 For $n_* \in \set{0, \dots, n_0}$,
 \begin{align}\label{eq:aux2}
 \measuredangle(v_{n_*}, E_{p_{n_*}}^u) \leq \hat{\mu}^{n_*}\measuredangle(v_0, E_{p_0}^u) + 2C_4C_0\sum_{j = 1}^{n_*}\hat{\mu}^{n_* - j}\mu^{j}\tilde{d}_0,
 \end{align}
 with $\mu$ and $C_0$ as in the statement of Lemma \ref{lem:orbit}, and $\tilde d_0$ as in the proof thereof.
 \end{claim}

 \begin{rem}
 We adopt the convention that for $i < j$, $\sum_{j}^i\bullet = 0$.
 \end{rem}

 \begin{proof}[Proof of Claim \ref{claim:arb2}]
 The result obviously holds with $n_* = 0$. For the case $n_* > 0$ we proceed by induction. In what follows, $\tilde d_n$, $n = 0, 1, \dots$, is as in the proof of Lemma \ref{lem:orbit}. For $n_* = 1$, we indeed have
 \begin{align*}
 \measuredangle(v_1, E_{p_1}^u) &\leq \measuredangle(v_1, E_{p_0}^u) + \measuredangle(E_{p_0}^u, E_{p_1}^u)\\
 &\leq \hat{\mu}\measuredangle(v_0, E_{p_0}^u) + \measuredangle(E_{p_0}^u, E_{p_1}^u)\\
 &\leq \hat{\mu}\measuredangle(v_0, E_{p_0}^u) + \measuredangle(E_{p_0}^u, \mathbf{z}) + \measuredangle(E_{p_1}^u, \mathbf{z})\\
 &\leq \hat{\mu}\measuredangle(v_0, E_{p_0}^u) + C_4\tilde{d}_0 + C_4\tilde{d}_1\\
 & \leq \hat{\mu}\measuredangle(v_0, E_{p_0}^u) + C_4C_0\tilde{d}_0 + C_4C_0\mu\tilde{d}_0,
 \end{align*}
 where the last inequality follows from \eqref{eq:c-orbit-eq}.

 Assuming now that the claim holds for some $n_*\in\set{1,\dots,n_0 - 1}$, for $n_* + 1$ we have the estimate
 \begin{align*}
 \measuredangle(v_{n_* + 1}, E_{p_{n_* + 1}}^u)
 &\leq \hat{\mu}\measuredangle(v_{n_*}, E_{p_{n_*}}^u) + \measuredangle(E_{p_{n_*}}^u, E_{p_{n_* + 1}}^u) \\
 &\leq \hat{\mu}\left(\hat{\mu}^{n_*}\measuredangle(v_0, E_{p_0}^u) + 2C_4C_0\sum_{j = 1}^{n_*}\hat{\mu}^{n_* - j}\mu^{j}\tilde{d_0}\right) + 2C_4C_0\mu^{n_*+1}\tilde{d}_0,
 \end{align*}
 which gives \eqref{eq:aux2}.
 \end{proof}

 Now combining \eqref{eq:arb2} with \eqref{eq:aux2}, we obtain
 \begin{align*}
 \measuredangle(v_{n_*}, \mathbf{z}) \leq \hat{\mu}^{n_*}\measuredangle(v_0, E_{p_0}^u) + 2C_4C_0\sum_{j=1}^{n_*}\hat{\mu}^{n_* - j}\mu^{j}\tilde{d}_0 + C_4\tilde{d}_{n_*}.
 \end{align*}
 Hence we have
 \begin{align*}
 \sum_{n_* = 0}^{n_0}\measuredangle(v_{n_*}, \mathbf{z})
 \leq \sum_{n_* = 0}^{n_0}\hat{\mu}^{n_*}\measuredangle(v_0,E_{p_0}^u)
 + 2C_4C_0\sum_{n_* = 0}^{n_0}\sum_{j=1}^{n_*}\hat{\mu}^{n_* - j}\mu^{j}\tilde{d}_0 + C_4\sum_{n_* = 0}^{n_0}d_{n_*},
 \end{align*}
 which is bounded (uniformly in $n_0$). Similarly one can show that the second sum on the right of \eqref{eq:arb4} is uniformly bounded in $k$ and $n_0$.

 To prove (3), keeping in mind (2), it suffices of course to prove that there exists a constant $C_5 > 0$ such that for any unit vector $v\in K_1^u$ and any unit vector $u\in K_1^s$, and any $n\in \N$,
 \begin{align*}
  \left(\frac{\norm{D\mathcal{G}^n_\mathbf{0}(v)}}{\norm{D\mathcal{G}^{-n}_\mathbf{0}(u)}}\right)^{\pm 1} \leq C_5;
 \end{align*}
 but this follows immediately from the definition of $D\mathcal{G}_\mathbf{0}$.
\end{proof}

\begin{lem}\label{lem:rect}
For any $\delta > 0$, there exists a sufficiently small neighborhood $U$ of $P_1$ such that the change of coordinates $\Phi: U\rightarrow\R^3$ can be chosen in such a way that in addition to the properties listed in Section \ref{sec:proof4}, $\Phi$ also satisfies the following.

There exists a nonempty open subset $\mathcal{N}$ of $\Phi(U)$, and constants $C >0$ and $V_0<0$, such that for all $V\in(V_0, 0)$, $\Phi(\mathbb{S}_V\cap U)\setminus \mathcal{N}\neq \emptyset$, and we have the following.

\begin{enumerate}\itemsep0.5em

\item The part of the plane $\set{y = z, y>0}\cap \Phi(U)$ lies in $\mathcal{N}$.

\item For all $p\in\mathcal{N}$, $p_z, p_y\neq 0$ and $\left(\frac{p_z}{p_y}\right)^{\pm 1}\leq C$, where, as above, $p_y$ and $p_z$ are the $y$ and the $z$ coordinates of $p$, respectively.

\item For all $V\in(V_0, 0)$ and $p\in\Phi(\mathbb{S}_V\cap U)\setminus\mathcal{N}$, the tangent plane of $\Phi(\mathbb{S}_V)$ at $p$ makes the angle of at most $\delta$ radians with the $xz$ plane (respectively, the $xy$ plane) if $p_z > p_y$ (respectively, if $p_z \leq p_y$).

\end{enumerate}
\end{lem}

\begin{proof}
Notice that $P_1$ is a nondegenerate singularity of the Fricke-Vogt invariant $I(x,y,z)=x^2+y^2+z^2-2xyz-1$. By the Morse lemma, we can find a diffeomorphism $\Psi$ from a neighborhood of $P_1$ into $\R^3$, mapping the surfaces $S_V$ to the hyperboloids $x^2+y^2-z^2 = V$, and the point $P_1$ to the origin. Denote, as above, the strong-stable (respectively, the strong-unstable) manifold of $P_1$ on $\mathbb{S}_0$ by $W^{ss}$ (respectively, $W^{uu}$). Denote by $(r, \theta, z)$ the cylindrical coordinates on $\R^3$. The stable and the unstable foliations on $\mathbb{S}_0$ are uniformly transversal in a neighborhood of $P_1$ (see Lemma 3.1 in \cite{Damanik2009}); consequently, the tangent directions of $\pi_{xy}\circ\Psi(W^{ss})$ and $\pi_{xy}\circ\Psi(W^{uu})$, where $\pi_{xy}$ denotes the orthogonal projection into the $xy$ plane, are uniformly transversal on the $xy$ plane in a neighborhood of the origin. Thus there exists a smooth map $a:[0, 2\pi)\times\R\rightarrow [0, 2\pi)\times\R$, such that $\alpha:(r, \theta, z)\mapsto(r, a(\theta, z), z)$ is a diffeomorphism on $\R^3$ and $\alpha\circ \Psi(W^{ss})$ (respectively, $\alpha\circ \Psi(W^{uu})$) is part of the line $\set{(0, t, t): t\in\R}$ (respectively, $\set{(0, t, -t): t\in\R}$); clearly $\alpha$ preserves the hyperboloids $x^2 + y^2 - z^2 = V$. Finally, let $\Theta: \R^3\rightarrow \R^3$ denote the rotation by $-\frac{\pi}{4}$ radians about the $x$ axis. Then $\Theta\circ\alpha\circ\Psi$ maps the surfaces $S_V$ to the hyperboloids $x^2+y^2-z^2=V$ rotated about the $x$ axis by $-\frac{\pi}{4}$ radians, such that $P_1$ is mapped to the origin, $W^{ss}$ is mapped onto part of the positive $y$ axis, and $W^{uu}$ is mapped onto part of the positive $z$ axis.

Denote the center-stable (respectively, the center-unstable) manifold that contains $W^{ss}$ (respectively, $W^{uu}$) by $W^{cs}$ (respectively, $W^{cu}$). Let $\Pi$ denote the smooth two-dimensional submanifold of $\R^3$ that is foliated by the curves $\set{\xi_u}$ with $\xi_0 = \Theta\circ\alpha\circ\Psi(W^{cs}\cap W^{cu}\cap U)$, and for $u\neq 0$, $\xi_u$ is the translation of $\xi_0$ by $u$ units in the direction of the vector $(0, 1, 1)$. Let $\widetilde{\mathcal{G}}$ denote the push-forward of the Fibonacci trace map $T$ by the diffeomorphism $\Theta\circ\alpha\circ\Psi$ in the neighborhood $U$. Let $\widetilde{\mathcal{N}}_k$ denote the region of $\Theta\circ\alpha\circ\Psi(U)\setminus\xi_0$ bounded by the surfaces $\widetilde{\mathcal{G}}^{-k}(\Pi)$ and $\widetilde{\mathcal{G}}^k(\Pi)$, $k\in\N$. If $U$ is sufficiently small, (1)-(3) hold with $\widetilde{\mathcal{N}}_k$ in place of $\mathcal{N}$ and $\Theta\circ\alpha\circ\Psi$ in place of $\Phi$, for $k\in\N$ sufficiently large, a suitable choice of $C_0 > 0$ in place of $C$ (depending on $k$), and $V_0 < 0$ sufficiently close to zero.

Observe that $\Theta\circ\alpha\circ\Psi(W^{cs}\cap U)$ (respectively, $\Theta\circ\alpha\circ\Psi(W^{cs}\cap U)$) is tangent to the cone $\Theta\circ\alpha\circ\Psi(\mathbb{S}_0\cap U)$ along the positive $y$ (respectively, the positive $z$) axis. Thus for every $\epsilon > 0$ and sufficiently small neighborhood $U$ (depending on $\epsilon$), we can construct a diffeomorphism $\Phi_\epsilon: \Theta\circ\alpha\circ\Psi(U)\rightarrow\R^3$ mapping $\Theta\circ\alpha\circ\Psi(W^{cs}\cap U)$ (respectively, $\Theta\circ\alpha\circ\Psi(W^{cu}\cap U)$) into the $xy$ (respectively, the $xz$) plane, with $\norm{\Phi_\epsilon - \mathrm{Id}}_{C^1} < \epsilon$ (here $\mathrm{Id}$ denotes the identity map, and $\norm{\cdot}_{C^1}$ is the $C^1$ norm). Then, for sufficiently small $\epsilon > 0$, (1)-(3) hold with $\Phi=\Phi_\epsilon\circ\Theta\circ\alpha\circ\Psi$ and $\mathcal{N}$ the interior of $\Phi_\epsilon(\widetilde{\mathcal{N}}_k)$, with suitably chosen $C\geq C_0$ and $V_0 < 0$.
\end{proof}

In what follows, $\mathcal{N}$, $U$, and $\Phi$ are as in Lemma \ref{lem:rect}. We continue to denote by $\mathcal{G}$ the pushforward of the Fibonacci trace map $T$ by $\Phi$ in the neighborhood $U$ (and, as above, $\mathcal{G}_V$ denotes the pushforward of $T_V$). Implicit in this notation is the dependence on the choice of $\delta$ in the statement of Lemma \ref{lem:rect}. In what follows, we choose $\delta = \frac{\pi}{6}$ radians.

\begin{rem}\label{rem:N-fund}
Notice that by construction, $\mathcal{N}$ is a fundamental domain. More precisely, if $p\in\Phi(U)\setminus\mathcal{N}$ with $p_y > p_z$, and for some $k\in\N$, $\mathcal{G}^k(p)\in\Phi(U)\setminus\mathcal{N}$ with $\mathcal{G}^k(p)_z > \mathcal{G}^k(p)_y$, then for some $n\in\set{1, \dots, k}$, $\mathcal{G}^n(p)\in\mathcal{N}$.
\end{rem}

\begin{lem}\label{lem:passage-step-estimate}
	There exists an open neighborhood $O$ of $P_1$ with $\overline{O}\subset U$, such that the following holds.

  Let $l_s$ and $l_u$ be points on the positive $y$ and the positive $z$ axis, respectively, in $\Phi(O)$. For all $N\in\N$ there exists $V_0 < 0$ (which depends on $N$) and $m\in\N$ (which does not depend on $N$) such that for all $V\in (V_0, 0)$ and $p\in\Phi(O\cap \mathbb{S}_V)$ with $p_z < l_u$, $p_y > l_s$, and $p_z < p_y$, there exists $\hat{n} = \hat{n}(p)\in\N$ such that
	\begin{enumerate}\itemsep0.5em

  \item $\mathcal{G}^{\hat{n}}(p)_z > l_u$ and for all $n\in\set{1,\dots, \hat{n}}$, $\mathcal{G}^{n}(p)\in\Phi(U)$.

	\item There exist $n_*, n^*\in\set{1,\dots,\hat{n}}$, $1 < n_* < n^* < \hat{n}$, such that for all $n\in\set{n_*, \dots, n^*}$, $\mathcal{G}^{n}(p)\in\mathcal{N}$, and for all $n\in\set{n^*+1,\dots, \hat{n}}$ and $n\in\set{1, \dots, n_*-1}$, $\mathcal{G}^{n}(p)\notin \mathcal{N}$.

  \item For all $n\in\set{n^*+1,\dots, \hat{n}}$ (respectively, $n\in\set{1,\dots,n_*-1}$), $\mathcal{G}^{n}(p)_z > \mathcal{G}^{n^*}(p)_y$ (respectively, $\mathcal{G}^{n}(p)_y > \mathcal{G}^{n}(p)_z$).

  \item With $k$ as in the proof of Lemma \ref{lem:rect} (with $\delta = \frac{\pi}{6}$ as above), we have $2k - 1\leq n^* - n_* \leq 2k$.

	\item $\hat{n} - n^* \leq m + n_*$ (and, similarly, $n_* \leq \hat{n} - n^* + m$).

  \item $\hat{n} - n^* > N$.
	\end{enumerate}
\end{lem}

\begin{proof}
  With $O$ initially chosen sufficiently small, (1) follows from \eqref{eq:differential} and the $C^1$ closeness of $\mathcal{G}$ to $D\mathcal{G}_\mathbf{0}$. By Remark \ref{rem:N-fund} and the construction of $\mathcal{N}$ (see the proof of Lemma \ref{lem:rect}), (2), (3) and (4) follow.

  Let us record, for ease of reference, the following simple fact.
	\begin{claim}\label{claim:geom-obvious}
	There exists $C_0 > 0$ such that the following holds. Suppose $\gamma$ is a compact smooth regular curve in $\Phi(U)$ with the endpoints $p$ and $q$. Suppose $p$ lies in the $xz$ plane (respectively, the $xy$ plane), and $q$ does not. Suppose further that the tangent vector at every point $r$ of $\gamma$ belongs to the cone $K_{1/2}^s(r)$ (respectively, $K_{1/2}^u(r)$) from \eqref{eq:cone-field}. Then, denoting by $d$ the distance from $q$ to the $xz$ plane (respectively, the $xy$ plane) and by $l$ the length of $\gamma$, we have $(d/l)^{\pm 1} \leq C_0$.
	\end{claim}
  \begin{proof}
  Let $\beta$ denote the line segment connecting $p$ and $q$. Since the tangent space of $\gamma$ lies inside the cone field $K_{1/2}^s$ or $K_{1/2}^u$, $(d/\beta)^{\pm 1}$ is bounded independently of $\gamma$. Thus it is enough to prove existence of a constant $\tilde C_0 > 0$ independent of $\gamma$ such that $(\beta/l)^{\pm 1}\leq \tilde C_0$.

  Let us introduce an orthogonal coordinate system $(x_1, x_2, x_3)$ with its origin at $p$ such that $\beta$ lies along the positive $x_1$ axis, so that $p$ and $q$ lie on the $x_1$ axis, with $p$ at the origin. Write $q_1$ for $q$ measured along the $x_1$ axis. Thus the length of $\beta$ is $q_1$.

  Then $x_1$ lies inside of the cone field. In this coordinate system, write $\gamma_1(t)$ for the component of $\gamma(t)$ along the $x_1$ axis, parameterized, say, on $[0, 1]$. By regularity, if $\gamma_1'(t) = 0$ then $\measuredangle(x_1, \gamma'(t)) = \pi/2$, which violates the cone condition. Thus for all $t$, $\gamma_1'(t)\neq 0$, and we may reparameterize $\gamma$ along $x_1$, that is, we may write $\gamma(t) = (t, \gamma_2(t), \gamma_3(t))$.

  We have $\abs{\gamma_j'(t)} = \tan\theta$, where $\theta$ is the angle between $x_1$ and the projection of $\gamma'(t)$ into the $(x_1, x_j)$ plane. Because the cone field is of size $1/2 < 1$, there exists $\tilde\theta \in (0, \pi/2)$ independent of $\gamma$ such that for all $t$, $\abs{\gamma_j'(t)} < \tan\tilde\theta$. Now we have
  \begin{align*}
  l = \int_{0}^{q_1}\norm{\gamma'(t)}dt \leq \int_{0}^{q_1}\sum_j \abs{\gamma_j'(t)}dt \leq q_1(1 + 2\tan\tilde\theta).
  \end{align*}
  The result follows with $\tilde C_0 = 1 + 2\tan\tilde \theta$.
  \end{proof}

	Suppose that $p\in \mathcal{N}$ such that $p_y < l_s$ and $p_z < l_u$. Assume that $n^{(s)}_p, n^{(u)}_p\in\N$ are the smallest natural numbers such that $\mathcal{G}^{n^{(u)}_p}(p)_z \geq l_u$ and $\mathcal{G}^{-n^{(s)}_p}(p)_y \geq l_s$. Then (5) would follow if we could demonstrate existence of a universal upper bound (i.e. independent of all $V < 0$ sufficiently close to zero and $p\in\mathcal{N}$) on $\abs{n^{(s)}_p - n^{(u)}_p}$.

  Let $\beta_s$ (respectively, $\beta_u$) be the shortest line segments connecting $p$ and the $xz$ (respectively, the $xy$) plane. Let us denote by $\dist(p, xy)$ and $\dist(p, xz)$ the distance from the point $p$ to the $xy$ and the $xz$ planes, respectively. Clearly there exists a universal constant $C_1 > 0$ such that
  \begin{align*}
  \left(\frac{\dist(\mathcal{G}^{-n_p^{(s)}}(p)_y, xz)}{\dist(\mathcal{G}^{n_p^{(u)}}(p)_z, xy)}\right)^{\pm 1}\leq C_1.
  \end{align*}
  Observe that $\beta_s$ and $\beta_u$ fall inside the cone fields $K_{1/2}^s$ and $K_{1/2}^u$, respectively. We may assume that these two cone fields are invariant under $D\mathcal{G}$ and $D\mathcal{G}^{-1}$, respectively, by taking smaller $U$ to begin with.

  By the invariance of the cone fields $K^s_{1/2}$ and $K^u_{1/2}$, we see that for every point $r$ on the curve $\mathcal{G}^{-n_p^{(s)}}(\beta_s)$, the tangent vector to the curve at $r$ belongs to $K^s_{1/2}(r)$; similarly for $\beta_u$ with $n_p^{(u)}$ in place of $-n_p^{(s)}$ and $K^u_{1/2}(r)$ in place of $K^s_{1/2}(r)$. Now using Claim \ref{claim:geom-obvious}, with $C_0$ as in the claim, we get
  \begin{align}\label{eq:arb-19}
  \left(\frac{\len(\mathcal{G}^{-n_p^{(s)}}(\beta_s))}{\len(\mathcal{G}^{n_p^{(u)}}(\beta_u))}\right)^{\pm 1}\leq C_1C_0,
  \end{align}
  where $\len(\cdot)$ denotes the length of a curve.

  Let us assume that $n_p^{(s)}>n_p^{(u)}$. We have
  \begin{align*}
  \len(\mathcal{G}^{-n_p^{(s)}}(\beta_s))=\len(\mathcal{G}^{n_p^{(u)}-n_p^{(s)}}\circ\mathcal{G}^{-n_p^{(u)}}(\beta_s)).
  \end{align*}
  Since we have $K_{1/2}^s\subset K_1^s$ and $K_{1/2}^u\subset K_1^u$, Lemma \ref{lem:bdp} is applicable with $K_{1/2}$ in place of $K_1$. So by Lemma \ref{lem:bdp} (1), there exists a universal constant $\tilde{\lambda} > 1$ such that
  \begin{align}\label{eq:arb-20}
  \len(\mathcal{G}^{n_p^{(u)}-n_p^{(s)}}\circ\mathcal{G}^{-n_p^{(u)}}(\beta_s))\geq \tilde{\lambda}^{n_p^{(s)}-n_p^{(u)}}\len(\mathcal{G}^{-n_p^{(u)}}(\beta_s)).
  \end{align}
  Now with $D$ as in Lemma \ref{lem:bdp} (3), from \eqref{eq:arb-19}, \eqref{eq:arb-20}, and Lemma \ref{lem:bdp} (3) we obtain
  \begin{align*}
  \tilde{\lambda}^{n_p^{(s)}-n_p^{(u)}}\frac{\len(\beta_s)}{\len(\beta_u)}\leq DC_1C_0.
  \end{align*}
  Now with $C$ as in (2) of Lemma \ref{lem:rect}, from Lemma \ref{lem:rect} and the above inequality we get
  \begin{align*}
  \tilde{\lambda}^{n_p^{(s)}-n_p^{(u)}}\leq CDC_1C_0.
  \end{align*}
  This gives a universal upper bound on $n_p^{(s)}-n_p^{(u)}$ in case $n_p^{(s)}>n_p^{(u)}$. The case $n_p^{(s)}<n_p^{(u)}$ is handled similarly. This proves (5).

  Finally, (6) follows easily since the neighborhood $\mathcal{N}\cap\Phi(\mathbb{S}_V)$ approaches the origin in the Hausdorff metric as $V\rightarrow 0$.
	\end{proof}

  We are now ready to prove the main result of Step (3), namely, with the notation from Proposition \ref{prop:curve-of-tangs}, we have
  \begin{prop}\label{prop:s3}
  For every $\Delta > 0$ there exists $\epsilon = \epsilon(\Delta) > 0$ such that the following holds. For all $V < 0$ sufficiently close to zero, let $\Lambda_{V, \epsilon}$ be the continuation of $\Upsilon_\epsilon$ from $\mathbb{S}_0$ to $\mathbb{S}_V$ as in Section \ref{sec:proof-1-2}. There exists $V_0 < 0$ depending on $\epsilon$ that satisfies the conclusion of Proposition \ref{prop:curve-of-tangs} with some fixed $\delta^u > 0$, and we have the following.

 For all $V\in (V_0, 0)$ let $\gamma_{V, \epsilon}$ denote the curve of tangencies on $\Phi(\mathbb{S}_V)$ that has been constructed in Step (2), such that with $\hat{M}_{\delta^s, V}^s$ and $k = k(V)$ as in Proposition \ref{prop:curve-of-tangs}, $\mathcal{G}^{-k}(\gamma_{V, \epsilon})$ lies in $\hat{M}_{\delta^s, V}^s$; denote by $\mathcal{C}_{V, \epsilon}^s$ and $\mathcal{C}_{V, \epsilon}^u$ the sets formed by intersections of $\gamma_{V, \epsilon}$ with the stable and the unstable laminations of $\Phi(\Lambda_{V, \epsilon})$, respectively. Then $\mathcal{C}_{V, \epsilon}^s$ and $\mathcal{C}_{V, \epsilon}^u$ are nonempty Cantor sets, $\mathcal{C}_{V,\epsilon}^s$ does not fall entirely into a gap of $\mathcal{C}_{V,\epsilon}^u$, and $\tau(\mathcal{C}_{V,\epsilon}^s), \tau(\mathcal{C}_{V,\epsilon}^u) > \Delta$ (where $\tau$ is the thickness).
  \end{prop}

  In the proof of the proposition, we use the notation from Steps (1) and (2).

  \begin{proof}[Proof of Proposition \ref{prop:s3}]
  That for all $\epsilon > 0$ sufficiently small and for all $V < 0$ sufficiently close to zero, $\mathcal{C}_{V, \epsilon}^s$ and $\mathcal{C}_{V, \epsilon}^u$ are nonempty Cantor sets and one does not fall into a gap of the other follows by construction of $\gamma_{V, \epsilon}$ (indeed, recall that $\hat{M}_{\delta^s, V}^s$ is a fundamental domain).

  Fix $\Delta > 0$ and $\delta, \delta^u > 0$ with $\delta^u\in(0, \delta)$ as in Proposition \ref{prop:tangencies}. Fix $\epsilon > 0$ so small, that when the stable lamination of $\Phi(\Lambda_{0, \epsilon})$ is orthogonally projected via ${\mathbf{P}_V^u}^{-1}$ into $B_\delta^u$, its intersection with the $z$ axis forms a Cantor set of thickness larger than $\Delta$. Indeed, this is possible due to Lemma \ref{lem:horseshoes-on-S0} with the observation that since $\Phi$ is a diffeomorphism, and in particular $\Phi$ and $\Phi^{-1}$ are both Lipschitz, the distortion of the thickness under the action by $\Phi$ is bounded by multiplication by a constant.
  \begin{lem}\label{lem:arb1}
  Denote by $W^s_{V, \epsilon}$ the image of the stable lamination of $\Phi(\Lambda_{V, \epsilon})$ under ${\mathbf{P}_V^u}^{-1}$. There exists $\tilde{\epsilon}>0$, $\tilde{V}_0 < 0$, and $\delta^u\in(0, \delta)$ sufficiently small, such that for all $V\in (\tilde{V}_0, 0)$ the following holds. Assume that $\gamma$ is a compact curve in $\hat{B}_{\delta^u, V}^u$ such that the distance from ${\mathbf{P}_V^u}^{-1}(\gamma)$ to the $z$ axis in the $C^1$ topology is not larger than $\tilde{\epsilon}$. Then $\gamma$ intersects $\mathbf{P}_{V}^u(W^s_{V, \epsilon})$ transversally in a Cantor set whose thickness is larger than $\Delta$.
  \end{lem}

  \begin{proof}
  That, given the assumptions with $\tilde{\epsilon}> 0$ sufficiently small and $\tilde{V}_0$ sufficiently close to zero, $\gamma$ intersects $\mathbf{P}_V^u(W_{V, \epsilon}^s)$ transversally is obvious since $W^s_{V, \epsilon}$ intersects the $z$ axis transversally. Let us prove that this intersection forms a Cantor set of thickness larger than $\Delta$ (possibly taking $\tilde{\epsilon}$ and taking $\tilde{V}_0$ closer to zero as necessary).

  By assumptions (the paragraph preceding the statement of Lemma \ref{lem:arb1}) the intersection of $W_{0, \epsilon}^s$ with the $z$ axis forms a Cantor set of thickness larger than $\Delta$. On the other hand, the thickness of  the intersection of $W_{V, \epsilon}^s$ with the $z$ axis varies continuously in $V$ (see Theorem 2 in Chapter 4 of \cite{Palis1993}). It follows that for all $V < 0$ sufficiently close to zero, the thickness of the intersection of $W_{V, \epsilon}^s$ with the $z$ axis is also larger than $\Delta$.

  Now view $W_{V, \epsilon}^s$ as a sublamination of ${\mathbf{P}_V^u}^{-1}(\mathcal{F}_V^s)$. For any $V < 0$ sufficiently close to zero and any curve $\tilde{\gamma}$ in ${\mathbf{P}_V^s}^{-1}(\hat{B}_{\delta^u, V}^s)$ sufficiently close to the $z$ axis in the $C^1$ topology, the intersection of $\tilde{\gamma}$ with $W_{V, \epsilon}^s$ can be viewed as the $C^1$ image of the intersection of $W_{V,\epsilon}^s$ with the $z$ axis under the holonomy map defined as projection along the foliation ${\mathbf{P}_V^u}^{-1}(\mathcal{F}_V^s)$. Moreover, the foliation $\mathcal{F}_V^s$ can be constructed in such a way as to depend continuously in the $C^1$ topology on $V$ (hence ${\mathbf{P}_V^u}^{-1}(\mathcal{F}_V^s)$ also depends continuously on $V$ in the $C^1$ topology); see Theorem 8 in Appendix 1 of \cite{Palis1993}. It follows that for all $L > 1$ there exist $\tilde{\epsilon} > 0$ and $\tilde{V}_0 < 0$ such that for all $V\in (\tilde{V}_0, 0)$ and every $\tilde{\gamma}$ in ${\mathbf{P}_V^u}^{-1}(\hat{B}_{\delta^u, V}^s)$ $\tilde{\epsilon}$-close to the $z$ axis in the $C^1$ topology, the holonomy map along ${\mathbf{P}_V^u}^{-1}(\mathcal{F}_V^s)$ from $\tilde{\gamma}$ to the $z$ axis, as well as its inverse, are Lipschitz with a Lipschitz constant $L$.

  Now taking $L$ above sufficiently close to one, we can make sure that the thickness of the intersection of $\tilde{\gamma}$ with $W_{V, \epsilon}^s$ is also larger than $\Delta$. Finally, projection of $\tilde{\gamma}$ by $\mathbf{P}_V^u$ does not destroy these bounds provided that $V < 0$ and $\delta^u > 0$ are sufficiently close to zero (so that $\mathbf{P}_V^u$ is close in the $C^1$ topology to the identity map).
  \end{proof}

  We wish to apply Lemma \ref{lem:arb1} with $\gamma = \gamma_{V, \epsilon}$. For this we need to make sure that as $V\rightarrow 0$, $\gamma_{V, \epsilon}$ approaches the $z$ axis in the $C^1$ topology. This is done in the next lemma.

  \begin{lem}\label{lem:arb2}
  With $\gamma_{V, \epsilon}$ as in the statement of Proposition \ref{prop:s3}, $\gamma_{V, \epsilon}$ approaches the $z$ axis in the $C^1$ topology as $V\rightarrow 0$.
  \end{lem}

  \begin{proof}
  Fix $\tilde{\epsilon} > 0$ arbitrarily small. It is enough to show that there exists $\tilde{V}_0 < 0$ sufficiently close to zero, such that for all $V \in (\tilde{V}_0, 0)$, $\gamma_{V, \epsilon}$ is $\tilde{\epsilon}$-close to the $z$ axis in the $C^0$ topology, and the tangent space of $\gamma_{V, \epsilon}$ falls inside the cone field $K^u_{\tilde{\epsilon}}$ from \eqref{eq:cone-field}.

  That for all $V < 0$ sufficiently close to zero, $\gamma_{V, \epsilon}$ is $\tilde{\epsilon}$-close to the $z$ axis in the $C^0$ topology follows from (the proof of) Lemma \ref{lem:tech-1}.

  Observe that for all $\tilde{\epsilon}>0$ there exists a tubular neighborhood of the positive $z$ axis, call it $\mathcal{T}_{\tilde{\epsilon}}$, such that for all $x\in\mathcal{T}_{\tilde{\epsilon}}$, $D\mathcal{G}(K_{\tilde{\epsilon}}^u(x))\subset K_{\tilde{\epsilon}}^u(\mathcal{G}(x))$. Indeed, this follows easily by continuity (the positive $z$ axis is invariant under $\mathcal{G}$ and $D\mathcal{G}$ and we have \eqref{eq:differential}). Now, to finish the proof, let us show that for all $V < 0$ sufficiently close to zero, there exists $n^*\in\N$, $n^* < k$ (here $k$ is as in the statement of Proposition \ref{prop:s3}), such that for all $n\in\set{0, \dots, n^*}$, $\mathcal{G}^{-n}(\gamma_{V, \epsilon})$ lies in $\mathcal{T}_{\tilde{\epsilon}}$, and the tangent space of $D\mathcal{G}^{-n^*}(\gamma_{V, \epsilon})$ falls inside the cone field $K_{\tilde{\epsilon}}^u$.

  Observe, again by (the proof of) Lemma \ref{lem:tech-1} that given any $n^*\in\N$, $n^* < k$, for all $V < 0$ sufficiently close to zero and for all $n\in\set{0, \dots, n^*}$, $\mathcal{G}^{-n}(\gamma_{V, \epsilon})$ indeed lies in $\mathcal{T}_{\tilde{\epsilon}}$. On the other hand, given $\tilde{\delta} > 0$, there exists $n^*\in\N$, $n^* < k$, such that for all $V < 0$ sufficiently close to zero, the orthogonal projection into the $xz$ plane of every leaf of $\mathcal{G}^{-n^*}(\mathcal{F}_V^s)$ is $\tilde{\delta}$-close to the horizontal in the $C^1$ topology, which follows from Proposition \ref{prop:c2-continuity}. But then given an arbitrarily small $\theta\in (0, \frac{\pi}{2})$, we can choose $\tilde{\delta} > 0$ sufficiently small and $n^*$ as above, such that for all $V < 0$ sufficiently close to zero, the orthogonal projection of $\mathcal{G}^{-n^*}(\gamma_{V, \epsilon})$ into the $xz$ plane makes the angle with the horizontal of at least $\frac{\pi}{2}-\theta$ radians (this can be shown using exactly the same argument as in Section \ref{sec:unfolding}, particularly Claim \ref{claim:arb3}). Finally, we can choose $\mathcal{T}_{\tilde{\epsilon}}$ sufficiently small above to begin with, such that for all $V < 0$ sufficiently close to zero, the vectors tangent to $\mathcal{G}^{-n^*}(\gamma_{V, \epsilon})$ make the angle of at most $\theta$ degrees with the $xz$ plane (this follows since away from the singularities, $\Phi(\mathbb{S}_V)$ tends to $\Phi(\mathbb{S}_0)$ in the $C^2$ topology as $V$ tends to zero). Thus, provided that we take $\tilde{\delta}>0$ and $\theta\in(0, \frac{\pi}{2})$ sufficiently small, there exists $n^*\in\N$, $n^* < k$, such that for all $V < 0$ sufficiently small the tangent space of $\mathcal{G}^{-n^*}(\gamma_{V, \epsilon})$ falls inside the cone field $K_{\tilde{\epsilon}}^u$, and, from the previous paragraph, for all $n\in\set{0, \dots, n^*}$, $\mathcal{G}^{-n}(\gamma_{V, \epsilon})$ lies in $\mathcal{T}_{\tilde{\epsilon}}$.
  \end{proof}

  Thus we see that there exists $\epsilon > 0$ such that for all $V < 0$ sufficiently close to zero, $\tau(\mathcal{C}_{V, \epsilon}^s) > \Delta$. Next we show that $\epsilon>0$ can be chosen sufficiently small, such that for all $V < 0$ sufficiently close to zero, $\tau(\mathcal{C}_{V, \epsilon}^u) > \Delta$.

  Fix $\tilde{\Delta}\geq \Delta$. With $k = k(V)$ as above, so that $\mathcal{G}^{-k}(\gamma_{V, \epsilon})$ lies in $\hat{M}_{\delta^s, V}^s$, observe that by the same arguments as above we can choose $\epsilon > 0$ sufficiently small such that for all $V < 0$ sufficiently close to zero, the thickness of $\tilde{\mathcal{C}}_{V, \epsilon}^u\eqdef \mathcal{G}^{-k}(\mathcal{C}_{V, \epsilon}^u)$ is larger than $\tilde{\Delta}$. Let us prove that with $\tilde{\Delta}$ sufficiently large, for all $\epsilon > 0$ sufficiently small and all $V < 0$ sufficiently close to zero, the thickness of $\mathcal{G}^k(\tilde{C}_{V, \epsilon}^u)$ is larger than $\Delta$. To control the distortion of the thickness under the action by $\mathcal{G}$, we use the bounded distortion property from Lemmas \ref{lem:bdp} and \ref{lem:passage-step-estimate}.

  Let $O$ be as in Lemma \ref{lem:passage-step-estimate}. We may assume that $\hat{M}_{\delta^s, V}^s$ and $\mathcal{G}^k(\hat{M}_{\delta^s, V}^s)$ lie in $O$ and, by Lemma \ref{lem:passage-step-estimate}, for all $n\in\set{1, \dots, k}$ and every point $p$ of $\gamma_{V, \epsilon}$, $\mathcal{G}^{-n}(p)\in U$.

  \begin{lem}\label{lem:bdp-prep}
  For all $V < 0$ sufficiently close to zero the following holds.

  Take a point $p = p(V)$ in $\mathcal{G}^{-k}(\gamma_{V, \epsilon})$. Let $n_* = n_*(p)\in \N$ and $n^* = n^*(p) \in \N$ be such that $\mathcal{G}^{n_* + 1}(p)\in\mathcal{N}$, $\mathcal{G}^{n^*-1}(p)\in\mathcal{N}$, $\mathcal{G}^{n_*}(p)\notin\mathcal{N}$, and $\mathcal{G}^{n^*}(p)\notin \mathcal{N}$; here $\mathcal{N}$ is as in Lemma \ref{lem:passage-step-estimate}. Let $u = u(p)$ be the unit tangent vector to $\mathcal{G}^{-k}(\gamma_{V, \epsilon})$ at the point $p$. Then $D\mathcal{G}^{n_*}(u)\in K_1^s(\mathcal{G}^{n_*}(p))$ and $D\mathcal{G}^{n^*}(u)\in K_1^u(\mathcal{G}^{n^*}(p))$.
  \end{lem}

  \begin{proof}
  Let us prove that $D\mathcal{G}^{n^*}(u)\in K_1^u(\mathcal{G}^{n^*}(p))$. Since we are working under the assumption that for all $p\in\Phi(U\cap \mathbb{S}_V)\setminus\mathcal{N}$ with $p_z > p_y$, the tangent plane to $\Phi(\mathbb{S}_V)$ at $p$ makes the angle of at most $\frac{\pi}{6}$ radians with the $xz$ plane, there exists $\tilde{\epsilon} > 0$ such that if the orthogonal projection of $D\mathcal{G}^{n^*}(u)$ into the $xz$ plane makes the angle of at most $\tilde{\epsilon}$ radians with the $z$ axis, then $D\mathcal{G}^{n^*}(u)\in K_1^u(\mathcal{G}^{n^*}(p))$.

  Observe that the leaf of the stable foliation which intersects $\mathcal{G}^{n^*-k}(\gamma_{V, \epsilon})$ in the point $\mathcal{G}^{n^*}(p)$ can be ensured to be arbitrarily close to the horizontal in the $C^1$ topology, by ensuring that $\mathcal{G}^{n^*}(p)_z$ is sufficiently small (see Proposition \ref{prop:c2-continuity}), which can in turn be ensured for all $V < 0$ sufficiently close to zero. Also, the leaf of the unstable foliation that intersects $\mathcal{G}^{n^*-k}(\gamma_{V, \epsilon})$ in the point $\mathcal{G}^{n^*}(p)$, and is tangent to the leaf of the stable foliation at that point, can be ensured to have an arbitrarily large curvature in the point $\mathcal{G}^{n^*}(p)$ by ensuring that $n^*$ is sufficiently large, which can again be ensured for all $V < 0$ sufficiently close to zero (see Section \ref{sec:quadratic}). Thus, employing Claim \ref{claim:arb3} we can ensure, for all $V < 0$ sufficiently close to zero, that $D\mathcal{G}^{n^*}(u)$ projects onto a vector in the $xz$ plane that makes the angle of at most $\tilde{\epsilon}$ radians with the vertical.

  The case of $n_*$ and $K_1^u$ in place of $n^*$ and $K_1^s$, respectively, is handeled similarly.
  \end{proof}
  We can now conclude the proof of the proposition by employing Lemma \ref{lem:passage-step-estimate} together with Lemma \ref{lem:bdp}. Observe that for $p$ a point on the curve $\mathcal{G}^{-k}(\gamma_{V, \epsilon})$, the orbit $\set{p, \mathcal{G}(p),\dots, \mathcal{G}^k(p)}$ is broken into
  \begin{align*}
  \set{p, \mathcal{G}(p), \dots, \mathcal{G}^{n_*}(p)}  &\subset \Phi(\mathbb{S}_V\cap U)\setminus\mathcal{N}, \\
  \set{\mathcal{G}^{n^*}(p), \dots, \mathcal{G}^k(p)} &\subset \Phi(\mathbb{S}_V\cap U)\setminus\mathcal{N},
  \end{align*}
  and
  \begin{align*}
  \set{\mathcal{G}^{n_* + 1}(p), \dots, \mathcal{G}^{n^* - 1}(p)}\subset\mathcal{N}\cap \Phi(\mathbb{S}_V\cap U),
  \end{align*}
  such that, by Lemma \ref{lem:passage-step-estimate}, $n^* - n_*$ and $\abs{(k - n^*) - n_*}$ are bounded uniformly, independently of $V$ and $p$, although $k$ depends on $V$ (the $k$ here is of course different from the $k$ in the statement of Lemma \ref{lem:passage-step-estimate}). Now, by Lemma \ref{lem:bdp}, there exists a universal constant $D$, independent of $V$ and $p$, such that for $u$ as in the statement of Lemma \ref{lem:bdp-prep}, we have
  \begin{align*}
  \left(\frac{\norm{D\mathcal{G}^{k-n^*}(D\mathcal{G}^{n^*}(u))}}{\norm{D\mathcal{G}^{-n_*}(D\mathcal{G}^{n_*}(u))}}\right)^{\pm 1}\leq D\left(\frac{\norm{D\mathcal{G}^{n^*}(u)}}{\norm{D\mathcal{G}^{n_*}(u)}}\right)^{\pm 1},
  \end{align*}
  and obviously the right side of the above inequality is universally bounded since $n^* - n_*$ is. It follows that
  \begin{align*}
  \left(\frac{\norm{D\mathcal{G}^k(u)}}{\norm{u}}\right)^{\pm 1}=\norm{D\mathcal{G}^k(u)}^{\pm 1}
  \end{align*}
  is universally bounded. From this, we obtain a universal bound on
  \begin{align*}
  \left(\frac{\tau(\tilde{C}^u_{V, \epsilon})}{\tau(C^u_{V, \epsilon})}\right)^{\pm 1}.
  \end{align*}
  Thus, we indeed can choose $\tilde{\Delta} > 0$ sufficiently large, such that if $\tau(\tilde{C}^u_{V, \epsilon}) > \tilde{\Delta}$, then $\tau(C^u_{V, \epsilon}) > \Delta$.
  \end{proof}
  This completes Step (3).
\end{proof}

\begin{proof}[Step (4)]
This follows immediately from Step (3) and the Gap Lemma of S. Newhouse (see the proof of Lemma 4 in \cite{Newhouse1979}).
\end{proof}

\begin{proof}[Step (5)]
This follows from the previous steps, noting that there exist periodic points the stable and the unstable manifolds of which are dense in the stable and the unstable laminations, respectively (see \cite[Appendix 1]{Palis1993})
\end{proof}

\subsection{Proof of Theorem \ref{thm:main}, (5), (6), and (7)}\label{sec:proof5-7}
Recall that, given a diffeomorphism $f$ of a compact two-dimensional manifold and a hyperbolic saddle $P$ of $f$, the \emph{homoclinic class} $H(P, f)$ is defined to be the closure of the union of all the transversal homoclinic points of $P$. Also recall that $H(P, f)$ is $f$-invariant and transitive (see \cite{Gorodetski2012}).

The statements (5), (6) and (7) of Theorem \ref{thm:main} follow from \cite[Theorem 4]{Gorodetski2012}. Namely, as part of Theorem 4 in \cite{Gorodetski2012}, we have

\begin{thm}[Gorodetski 2012]\label{thm:Anton}
Let $f_0$ be a $C^\infty$ area preserving diffeomorphism of a compact two-dimensional manifold $M$ with an orbit $\mathcal{O}$ of quadratic homoclinic tangencies associated to some hyperbolic fixed point $P_0$, and $\set{f_\mu}$ be a generic unfolding of $f_0$ along the parameter $\mu$. Then for any $\delta > 0$ there is an open set $\mathcal{U}\subseteq \R^1$, $0\in\overline{\mathcal{U}}$, such that the following holds.

There is a residual subset $\mathcal{R}\subseteq \mathcal{U}$ such that for every $\mu\in\mathcal{R}$,

\begin{enumerate}\itemsep0.5em
\item the homoclinic class $H(P_\mu, f_\mu)$ is accumulated by $f_\mu$'s generic elliptic points,

\item the homoclinic class $H(P_\mu, f_\mu)$ contains hyperbolic sets of Hausdorff dimension arbitrarily close to 2; in particular, the Hausdorff dimension of $H(P_\mu, f_\mu)$ is equal to 2.

\end{enumerate}
% \end{enumerate}
\end{thm}

Clearly the fixed point $P_0$ and its continuation $\set{P_\mu}$ can be replaced by a periodic point and, respectively, its continuation along the parameter $\mu$.

Now, in the notation of Theorem \ref{thm:Anton}, we can take $\mathcal{U} = (V_0, 0)$, where $V_0 < 0$ is sufficiently close to zero such that Theorem \ref{thm:main} (1)-(4) hold, $\mu = V$, $f_\mu = T_V$, and $P_0 = P_{V'}$, $V'\in (V_0, 0)$ (and its continuation $\set{P_V}_{V\in (V_0, 0)}$) a periodic point exhibiting a homoclinic tangency. Then Theorem \ref{thm:Anton} (1) gives Theorem \ref{thm:main} (5) and (7), and Theorem \ref{thm:Anton} (2) gives Theorem \ref{thm:main} (6).

\bibliographystyle{amsplain}
\bibliography{bibliography}

\end{document}